\documentclass[twoside,11pt]{article}
\usepackage{graphicx}
\usepackage{color}
\usepackage{amsmath,amssymb,amsthm,amsfonts}
\usepackage{graphicx}
\usepackage{color}
\usepackage{multicol}

\usepackage{mathptmx}
\usepackage{enumerate}
\usepackage[numbers,sort&compress]{natbib}
\numberwithin{equation}{section}
\newcommand{\R}{{\mathbb R}}

\newcommand{\be}{\begin{equation}}
\newcommand{\ee}{\end{equation}}

\newcommand{\sq}{\sqrt}

\newcommand{\pt}{\partial}

\newcommand{\ben}{\begin{eqnarray*}}
\newcommand{\enn}{\end{eqnarray*}}

\newcommand{\Om}{\Omega}

\newcommand{\ka}{\kappa}

\newcommand{\va}{\varepsilon}
\newcommand{\ti}{\tilde}

\newtheorem{theorem}{\textbf Theorem}[section]
\newtheorem{lemma}{\textbf Lemma}[section]
\newtheorem{rem}{\textbf Remark}[section]

\newtheorem{defin}{\textbf Definition}[section]

\def\endProof{{\hfill$\Box$}}

\numberwithin{equation}{section}
\begin{document}

\title{{\textbf{Boundary layer problem on chemotaxis-Navier-Stokes system with Robin boundary conditions} }}
\author{Qianqian Hou\thanks{Institute for Advanced Study in Mathematics, Harbin Institute of Technology, Harbin 150001, People's Republic of China
({\tt qianqian.hou@hit.edu.cn}).}}
\date{}
\maketitle

\begin{quote}
{\sc Abstract} This paper is concerned with the boundary layer problem on a chemotaxis-Navier-Stokes system modelling boundary layer formation of aerobic bacteria in fluid. Completing the system with physical
Robin-type boundary conditions for oxygen, no-flux and Dirichlet boundary conditions for bacteria and fluid velocity, we show that the gradients of its radial solutions in a region between two concentric spheres possessing boundary layer effects as the oxygen diffusion rate $\va$ goes to zero and the boundary-layer thickness is of order $\mathcal{O}(\va^\alpha)$ with $0<\alpha<\frac{1}{2}$.

\noindent
{\sc MSC}:{35A01, 35B40, 35B44, 35K57, 35Q92, 92C17}

\noindent
{\sc Keywords}: {Chemotaxis, Robin boundary conditions, Boundary layers, BL-thickness}
\maketitle
\end{quote}
\maketitle
\section{Introduction}
\textbf{Chemotaxis-Navier-Stokes model.}
Chemotaxis is a biological process, where bacteria orient their movement toward a chemically more favourable environment. Taking into account the fact that aerobic bacteria (e.g. \emph{Bacillus subtilis}) often live in fluid with dissolved oxygen (e.g. water), Tuval et al. conducted a detailed experimental and theoretical study on the interaction of bacterial chemotaxis, chemical diffusion and fluid convection. In particular, by placing a drop of water containing \emph{Bacillus subtilis} in a chamber with its upper surface open to the atmosphere, they observed that bacterial cells quickly get densely packed in a relatively thin fluid layer below the water-air interface through which the oxygen diffuses to the water drop (cf. \cite{Tuval}). The dynamics of the bacterial concentration $w(x,t)$, oxygen concentration $c(x,t)$ and fluid velocity $\vec{u}(x,t)$ can be encapsuled in  the following mathematical model
\be\label{cns}
\left\{
\begin{array}{lll}
w_t+\vec{u}\cdot\nabla w+\nabla\cdot(w \nabla c)=D_w \Delta w, \quad \ \ \ \ \  x\in \Om,\ \ t>0,
\\
c_t+\vec{u}\cdot \nabla c+w c=\va\, \Delta c,\qquad \qquad\qquad\   x\in \Om,\ \ t>0,\\
\vec{u}_t+\vec{u}\cdot\nabla \vec{u}+\nabla p+w\nabla \phi=D\Delta\vec{u},\quad\quad \ \ \ x\in \Om,\ \ t>0,\\
\nabla\cdot\vec{u}=0,\qquad\qquad\qquad\qquad \qquad \quad \ \ \ x\in \Om,\ \ t>0.\\
\end{array}
\right.
\ee
Here $\Om$ is a domain in $\R^{n}$ modelling the water drop, the gravitational potential $\phi(x)$ is a known smooth function and $p(x,t)$ denotes the associated scalar pressure of the fluid. The nonnegative constants $D_w$, $\va$ and $D$ are diffusion rates of the bacterial cells, oxygen and fluid velocity, respectively. The first two equations in system \eqref{cns} describe the chemotactic movement of bacterial cells toward increasing gradients of oxygen concentration where both bacteria and oxygen diffuse and are convected with the fluid. The third and fourth equations in \eqref{cns} are the well-known Navier-Stokes describing the evolution of viscous incompressible fluids with an additional term $w\nabla \phi$ corresponding to the gravity force effect of bacterial cells on the fluid. Neglecting the influence of surroundings on the oxygen concentration of the fluid, lots of analytical studies on \eqref{cns} in the literature adopt the following boundary conditions:
\be\label{bc1}
\nabla w(x,t)\cdot \vec{\nu}=0,\qquad \nabla c(x,t)\cdot \vec{\nu}=0,\qquad \vec{u}(x,t)=\mathbf{0},
\qquad x\in \partial \Om,\ \ t>0,
\ee
where $\vec{\nu}$ is the outward unit normal to the boundary $\partial \Om$ and the homogeneous Neumann boundary condition has been prescribed for $c$. See \cite{lorz2010,wang-winkler-xiang2018,winkler2012,winkler2014arma,winkler2016,winkler2017,zhang-li2015} for the solvability of \eqref{cns}-\eqref{bc1} when $\Om$ is a bounded domain in $\R^{n}\,(n=2,3)$ and see \cite{chae-kang-li2012,chae-kang-li2014,duan-lorz-markowich2010,liu-lorz2011,zhang-zheng2014}
when $\Om=\R^{n}$.
\newline
~\\
\textbf{Physical boundary conditions.}
Since part of the water drop in the chamber (cf. \cite{Tuval}) is surrounded by air, which leads to inevitable exchange between the dissolved oxygen in the drop and the free oxygen in surroundings, it has been suggested to use different, more realistic, inhomogeneous boundary conditions for the oxygen concentration $c$ (see \cite{braukhoff2017,braukhoff-tang2020} and \cite{Tuval}). Employing the well-known Henry's law that models the dissolution of gas in water (cf. \cite{atkins-paula2006}), Braukhoff-Tang proposed the following physical conditions:
\be\label{bc2}
\left\{
\aligned
&(D_w \nabla w-w\nabla c)(x,t)\cdot \vec{\nu}=0,\quad
\nabla c(x,t)\cdot \vec{\nu}=\kappa(x)(\lambda(x)-c(x,t)),\\
&\vec{u}(x,t)=\mathbf{0},
\ \qquad x\in \partial \Om,\ \ t>0,
\endaligned
\right.
\ee
in \cite{braukhoff2017,braukhoff-tang2020} to investigate the global solvability of \eqref{cns} with \eqref{bc2}, where $\lambda(x)>0$ denotes the maximal saturation of oxygen in the fluid and the influx of oxygen on the surface is proportional to the difference between the maximal and current concentrations. $\kappa(x)\geq 0$ is the absorption rate of gaseous oxygen into the fluid with $\kappa(x)>0$ on the water-air interface and $\kappa(x)=0$ on the water-solid interface. Authors in \cite{braukhoff2017,braukhoff-tang2020} proved the global existence of classical solutions in two-dimensional case and of weak solutions in three-dimensional case. The global existence and large-time behavior of weak solutions for a generalized system obtained from \eqref{cns} and \eqref{bc2} with the linear cell diffusion replaced by a nonlinear one have been derived in \cite{wu-xiang2020} for 2D case and in \cite{tian-xiang2020} for 3D case. For the fluid-free version of \eqref{cns} and \eqref{bc2}, Braukhoff-Lankeit in \cite{braukhoff-lankeit2019} showed the existence and uniqueness of nonconstant stationary solutions. Global classical solutions to the corresponding parabolic-elliptic system were later proved converging to these stationary solutions as $t$ goes to infinity (cf. \cite{fuest-lankeit-mizukami2021}). By formally passing $\kappa$ to $\infty$ in \eqref{bc2} (cf. \cite[Proposition 5.1]{braukhoff-lankeit2019}), the following boundary condition was adopted in \cite{Tuval}
\be\label{bc3}
(D_w \nabla w-w\nabla c)(x,t)\cdot \vec{\nu}=0,\quad 0=(\lambda-c(x,t)),\quad \vec{u}(x,t)=\mathbf{0},\quad x\in \partial \Om,\ \ t>0
\ee
with the constant $\lambda\geq 0$. Analytic study on \eqref{cns} and \eqref{bc3} was started with \cite{lorz2010} where the existence of local weak solutions in two-dimensional setting was proved. \cite{peng-xiang2019} showed the global existences of classical solutions around some constant equilibrium with the same boundary setting. Recently, generalized solutions for \eqref{cns} and \eqref{bc3} and radial solutions for the corresponding fluid-free system were constructed in \cite{wang-winkler-xiang2021} and \cite{lankeit-winkler2022}, respectively. For the latter case,  Carrillo-Li-Wang investigated the existence and structure of boundary layer solutions in \cite{carrillo-li-wang2021}. The above-mentioned appear to the only analytical results of \eqref{cns} with physical boundary
conditions \eqref{bc2} or \eqref{bc3} in the literature as far as we know.
\newline
~\\
\textbf{Boundary layer formation.}
It should be emphasized that one of the most important findings in the experiment conducted in \cite{Tuval} is the boundary layer formation of bacterial cells near the water-air interface and one intriguing topic in analytic studies is to uncover the underlying mechanism of this boundary layer phenomenon. However, rigourous analysis on boundary-layer formation for system \eqref{cns} with the physical boundary condition \eqref{bc2} seems unavailable in the literature so far.
The purpose of this paper is to make some progress on this issue by considering the boundary layer effects of its radially symmetric solutions. To this end,
 we rewrite \eqref{cns} with \eqref{bc2} in its radially symmetric form by assuming that the solutions $(w,c,\vec{u})$ depend only on the radial variable $r=|x|$ and time variable $t$. In a domain bounded by two concentric spheres, i.e. $\Omega=\{x=(x_1,x_2,\cdots,x_n)\in \mathbb{R}^n\,|\,\,\, 0<a<|x|<b
 \}$, \eqref{cns} and \eqref{bc2} read as
 \be\label{e1}
\left\{
\begin{array}{lll}
w_t=w_{rr}+\frac{n-1}{r}w_{r}-(wc_r)_r-\frac{n-1}{r}(wc_r),\quad (r,t)\in (a,b)\times(0,\infty),\\
c_t=\va c_{rr}+\va\frac{n-1}{r}c_r-wc,\\
w(r,0)=w_0(r),\quad c(r,0)=c_0(r),
\end{array}
\right.
\ee
with the following boundary conditions
\be\label{e2}
\left\{
\aligned
&[w_r-wc_r](a,t)=[w_r-wc_r](b,t)=0,\qquad\qquad\qquad&\\ &c_r(a,t)=-\kappa[\lambda-c(a,t)],\ \
c_r(b,t)=\kappa[\lambda-c(b,t)]\
\ & \mathrm{if}\ \va>0,\\
&[w_r-wc_r](a,t)=[w_r-wc_r](b,t)=0
\ & \mathrm{if}\ \va=0,
\endaligned
\right.
\ee
where $D_w=D=1$ have been assumed without loss of generality and the third and fourth equations in \eqref{cns} governing evolution of the fluid vanish since $\vec{u}=\mathbf{0}$ in the radially symmetric case due to the incompressible condition $\nabla\cdot\vec{u}=0$ and Dirichlet boundary condition $\vec{u}=\mathbf{0}$ on $\partial \Om$ (cf. \cite{hoff1992}). The constants $\lambda>0$, $\kappa\geq 0$ and no boundary condition is preassigned for $c$ in the case of $\va=0$ since its values at $r=a,b$ are intrinsically determined by the
second equation of (1.3) as $c|_{r=a,b}=[c_0e^{-\int_0^t w(r,\tau)d\tau}]|_{r=a,b}$.

We conclude this section by briefly recalling the previous analytic results on boundary layer formation in chemotaxis systems. The author and her collaborators investigated the boundary layer problem on a chemotaxis system with logarithmic sensitivity obtained from \eqref{cns}-\eqref{bc1} by letting $\vec{u}=p=\phi=0$ and replacing the first equation with $w_t+\nabla\cdot(w\nabla \ln c)=\Delta w$ and the boundary conditions with Dirichlet boundary condition for $w$ and Robin boundary condition for $c$, where the gradient of $w$ has been proved possessing boundary layer effects (cf. \cite{HWZ,HLWW,hou2019convergence}) as $\va$ goes to zero. Results in \cite{HWZ} were extended to the case with time-dependent boundary data (cf. \cite{peng-wang-zhao-zhu2018}). Recently, Lee-Wang-Yang investigated the existence and structure of  stationary boundary layer solutions for the fluid-free version of \eqref{cns} subject to the physical boundary condition \eqref{bc3} (cf. \cite{lee-wang-yang2020}).  Stationary boundary layer solutions for the chemotaxis system with logarithmic sensitivity subject to boundary condition \eqref{bc3} were later derived in \cite{carrillo-li-wang2021}.

The rest of the paper is organized as follows. We introduce some notations and state our main results in section 2. The proof of Theorem 2.1 and Theorem 2.2 are given in section 3 and section 4, respectively.

\section{Main Results}
For convenience we first state some notations.
~\\
\newline
\textbf{Notation.} We use $C_0$ to denote a generic constant which depends only on $a$, $b$, $n$ and use $C$ to denote a generic constant depending on $a$, $b$, $n$, $w_0$, $c_0$ and $T$.
 The generic constant $C_{\lambda}$ depends on $a$, $b$, $n$, $w_0$, $c_0$, $T$ and $\lambda$. All of $C_0$, $C$ and $C_{\lambda}$ may be different from line to line. $L^p$ with $1\leq p\leq \infty$ represents $L^p(a,b)$ and $H^m$ with $m\geq 1$ denotes $H^m(a,b)$.

The existence and uniqueness for global classical solutions of \eqref{e1}-\eqref{e2} with $\va>0$ and $\va=0$ have been proved in \cite[Theorem 1.1]{braukhoff-tang2020} and \cite[Theorem 4.1]{morales-rodrigo2010}, respectively. We cite these results in the following lemma for later use.

\begin{lemma}\label{l9} Let $\kappa\geq 0$, $\lambda>0$.
Assume $w_0$, $c_0\in H^2$ satisfy $w_0\geq 0$, $c_0>0$ on $[a,b]$. Then system \eqref{e1}-\eqref{e2} with $\va\geq 0$ admits a unique global classical solution $(w^\va,c^\va)$ depending on $\va$, fulfilling
\ben
w^\va(r,t)\geq 0, \quad c^\va(r,t)>0\ \ \ \text{in}\ (a,b)\times (0,\infty)
\enn
and
\ben
w^\va,\,c^\va\in C^{2+2\delta,1+\delta}([a,b]\times (0,T))\times C([a,b]\times [0,T))
\enn
for some $0<\delta<\frac{1}{2}$ and any $0<T<\infty$ and the equations in \eqref{e1}-\eqref{e2} are satisfied pointwise.
\end{lemma}

Next we recall the conventional definition of boundary layer (BL)-thickness (cf. \cite{frid-shelukhin99,frid-shelukhin20}).
\begin{defin}\label{D1}
Let $(w^\va, c^\va)$ and $(w^0, c^0)$ be the solution of \eqref{e1}-\eqref{e2} with $\va>0$ and $\va=0$, respectively. If there is a non-negative function $\delta=\delta(\va)$ satisfying $\delta(\va)\downarrow 0$ as $\va \downarrow 0$ such that
\ben
\begin{aligned}
&\lim_{\va\rightarrow 0}\|w^\va-w^0\|_{L^{\infty}(0,T;C[a,b])}=0,\ \quad
\lim_{\va\rightarrow 0}\|c^\va-c^0\|_{L^{\infty}(0,T;C[a,b])}=0,\\
&\lim_{\va\rightarrow 0}\|w_r^{\va}-w_r^0\|_{L^{\infty}(0,T;C[a+\delta,b-\delta])}=0,\ \quad
\lim_{\va\rightarrow 0}\|c_r^{\va}-c_r^0\|_{L^{\infty}(0,T;C[a+\delta,b-\delta])}=0,\\
&\liminf_{\va\rightarrow 0}\|w_r^{\va}-w_r^0\|_{L^{\infty}(0,T;C[a,b])}>0,\ \quad
\liminf_{\va\rightarrow 0}\|c_r^{\va}-c_r^0\|_{L^{\infty}(0,T;C[a,b])}>0,\\
\end{aligned}
\enn
we say that the gradients of solution $(w^\va,c^\va)$ to problem \eqref{e1}-\eqref{e2} have boundary layer effects as $\va \to 0$ and $\delta(\varepsilon)$ is called a BL-thickness.
\end{defin}

\begin{rem}\label{rod}
\em{
As mentioned in \cite{frid-shelukhin99,frid-shelukhin20}, the above definition does not determine the BL-thickness uniquely since any function $\delta_*(\va)$ satisfying $\delta_*(\va)\to 0$ as $\va\to 0$ and the inequality $\delta_*(\va)>\delta(\va)$ is also a BL-thickness. }
\end{rem}
We are now in a position to state our main results on the boundary layer problem. To this end we need two uniform-in-$\va$ estimate for the solutions of \eqref{e1}-\eqref{e2} with $\va>0$ corresponding to the case of $\kappa>0$ and $\kappa=0$, respectively, which are the key to show the existence of boundary layer solutions in case of $\kappa>0$ and the convergence of solutions in case of $\kappa=0$.
\begin{theorem}\label{t1}
Let $\lambda>0$, $\ka\geq0$ and $0<T<\infty$.
  Assume $w_0$, $c_0\in H^2$ with $w_0\geq 0$ and $c_0>0$ on $[a,b]$ satisfy the following compatibility conditions:
 \be\label{a00}
 \left\{
\begin{array}{lll}
 w_{0r}(a)-w_{0}(a)c_{0r}(a)=w_{0r}(b)-w_{0}(b)c_{0r}(b)=0,
 \\
 c_{0r}(a)=-\kappa [\lambda-c_{0}(a)],
 \quad
  c_{0r}(b)=\kappa [\lambda-c_{0}(b)].
 \end{array}
 \right.
 \ee
 Let $(w^\va,c^\va)$ with $0\leq \va<1$ be the unique solutions of \eqref{e1}-\eqref{e2}, derived in Lemma \ref{l9}. Then the following holds true:\\
(i) If $\kappa>0$. There is a constant $C_{\lambda}$ depending on $a$, $b$, $n$, $\|w_0\|_{H^2}$, $\|c_0\|_{H^2}$, $T$ and $\lambda$ such that
\be\label{a0}
\begin{split}
\|w^\va&\|_{L^\infty(0,T;H^1)}^2+\|c^\va\|_{L^\infty(0,T;H^1)}^2
+\|w^\va_t\|_{L^2(0,T;H^1)}^2
+\|c^\va_t\|_{L^2(0,T;H^1)}^2\\
&+
\va^{\frac{1}{2}}\|c^\va_{rr}\|_{L^\infty(0,T;L^2)}^2
+\va^{\frac{3}{2}}\|c^\va_{rrr}\|_{L^2(0,T;L^2)}^2
\leq C_{\lambda}\big(1+\exp\{\exp\{\exp\{C_{\lambda}\ka^2\}\}\}\big)
\end{split}
\ee
for $0<\va<1$.
 \newline
 (ii) If $\kappa=0$. There is a constant $C$ depending on $a$, $b$, $n$, $\|w_0\|_{H^2}$, $\|c_0\|_{H^2}$ and $T$ such that
\be\label{c19}
\begin{split}
\|w^\va&\|_{L^\infty(0,T;H^1)}^2+\|c^\va\|_{L^\infty(0,T;H^1)}^2
+\|w^\va_t\|_{L^2(0,T;H^1)}^2
+\|c^\va_t\|_{L^2(0,T;H^1)}^2\\
&+
\|c^\va_{rr}\|_{L^\infty(0,T;L^2)}^2
+\va\|c^\va_{rrr}\|_{L^2(0,T;L^2)}^2
\leq C
\end{split}
\ee
for $0<\va<1$.
 \newline
(iii) Assume further $c_0\in H^3$. Then there exists a constant $C$ depending on $a,b$, $n$, $\|w_0\|_{H^2}$, $\|c_0\|_{H^3}$ and $T$ such that
\be\label{a000}
\|w^0\|_{L^\infty (0,T;H^2)}+\|c^0\|_{L^\infty (0,T;H^3)}
+\|c^0_t\|_{L^\infty (0,T;H^2)}
+\|w^0\|_{L^2 (0,T;H^3)}^2+\|c^0_t\|_{L^2 (0,T;H^3)}\leq C.
\ee
\end{theorem}

Then the main results on existence of boundary layer solutions for system \eqref{e1}-\eqref{e2} are stated in the following Theorem.

\begin{theorem}\label{t2}
Let $\lambda>0$, $\ka\geq0$ and $0<T<\infty$.
  Assume $w_0\in H^2$, $c_0\in H^3$ satisfy $w_0\geq 0$, $c_0>0$ on $[a,b]$ and the compatibility conditions in \eqref{a00}.
 Let $(w^\va,c^\va)$ with $0\leq \va<1$ be the unique solutions of \eqref{e1}-\eqref{e2}, derived in Lemma \ref{l9}. Then the following holds true:\\
(i) If $\ka>0$. Then any non-negative function $\delta(\va)$ satisfying
\ben
\delta(\va)\rightarrow 0\,\,\, {\rm and}\,\,\, \va^{1/2}/\delta(\va)\rightarrow 0,\,\,\, {\rm as}\,\, \va\rightarrow 0
\enn
is a BL-thickness of \eqref{e1}-\eqref{e2} such that
\be\label{b1}
\|w^\va-w^0\|_{L^{\infty}(0,T;C[a,b])}
+\|c^\va-c^0\|_{L^{\infty}(0,T;C[a,b])}\leq C_{\lambda,\ka} \va^{\frac{1}{4}}
\ee
and
\be\label{b2}
\|w_r^{\va}-w_r^{0}\|_{L^{\infty}(0,T;C[a+\delta,b-\delta])}
+\|c_r^{\va}-c_r^{0}\|_{L^{\infty}(0,T;C[a+\delta,b-\delta])}
<C_{\lambda,\ka}\delta^{-1/2}\va^{1/4},
\ee
where the constant $C_{\lambda,\ka}$ depends on $a$, $b$, $n$, $\|w_0\|_{H^2}$, $\|c_0\|_{H^3}$, $T$, $\lambda$ and $\kappa$.
Moreover,
\be\label{b3}
\liminf_{\va\rightarrow 0}\|w_r^{\va}-w_r^{0}\|_{L^{\infty}(0,T;C[a,b])}>0,
\quad
\liminf_{\va\rightarrow 0}\|c_r^{\va}-c_r^{0}\|_{L^{\infty}(0,T;C[a,b])}>0
\ee
if and only if
\be\label{sc1}
\begin{split}
w^0(a,t)>0\quad \text{or}\quad w^0(b,t)>0 \ \ \ \text{for} \ \text{some}\  t\in [0,T].
\end{split}
\ee
 That is, the gradients of solutions $(w^\va,c^\va)$ have boundary layer effects as $\varepsilon \to 0$ iff (\ref{sc1}) holds.
\newline
(ii) If $\ka=0$. Then there exists a constant $C$ depending on $a$, $b$, $n$, $\|w_0\|_{H^2}$, $\|c_0\|_{H^3}$ and $T$ such that
\be\label{b4}
\|w^\va-w^0\|_{L^\infty(0,T;H^2)}^2+\|c^\va-c^0\|_{L^\infty(0,T;H^2)}^2
\leq C\va^2.
\ee
\end{theorem}

\section{Proof of Theorem \ref{t1}}
In this section we are devoted to deriving the uniform estimates \eqref{a0} and \eqref{c19} with $0<\va<1$ and the regularities in \eqref{a000} for solutions in the non-diffusive case with $\va=0$.
Except the existence for solutions of \eqref{e1}-\eqref{e2} with $\va>0$, authors in \cite{braukhoff-tang2020} also established some energy estimates for the solutions. However those estimates in \cite{braukhoff-tang2020} are dependent of $\va$ or $\kappa$, which can not be sued in the derivation of the uniform-in-$\va$ estimates \eqref{a0} and \eqref{c19}.  The proof of \eqref{a0} and \eqref{c19} is given in Lemma \ref{l1} - Lemma \ref{l4} and the proof of \eqref{a000} is given in Lemma \ref{l5}. Based on these lemmas, we shall prove Theorem \ref{t1} at the end of this section.

\begin{lemma}\label{l1} Suppose the assumptions in Theorem \ref{t1} hold. Let $(w^\va,c^\va)$ with $0<\va<1$ be the global solutions of \eqref{e1}-\eqref{e2}, derived in Lemma \ref{l9}. \\
(i) If $\kappa>0$. Then
\be\label{a1}
\sup_{t>0}\|c^\va(t)\|_{L^\infty}\leq \max \{\|c_0\|_{L^\infty}, \lambda\},\qquad \sup_{t>0}\int_a^b r^{n-1}w^\va(t)dr=\int_a^b r^{n-1}w_0dr
\ee
and
\be\label{a2}
\begin{split}
\frac{d}{dt}&\int_a^b r^{n-1}(w^\va\log w^\va-w^\va+1)dr+
2\frac{d}{dt}\int_a^b r^{n-1} (\partial_r\sqrt{c^\va})^2dr\\
&+\ka b^{n-1}\frac{d}{dt} \big[\lambda\log{\frac{\lambda}{c^\va(b,t)}}-\lambda+c^\va(b,t)\big]
+\ka a^{n-1}\frac{d}{dt} \big[\lambda\log{\frac{\lambda}{c^\va(a,t)}}-\lambda+c^\va(a,t)\big]\\
&+4\int_a^b r^{n-1}(\partial_r \sqrt{w^\va})^2dr+\va\int_a^b r^{n-1}c^{\va} \left[(\log c^\va)_{rr}
\right]^2dr\\
&
+\frac{\va\kappa\lambda}{2}\left\{b^{n-1}\frac{[c^\va_r(b,t)]^2}{[c^\va(b,t)]^2}
+a^{n-1}\frac{[c^\va_r(a,t)]^2}{[c^\va(a,t)]^2}
\right\}\\
\leq&
\va(n-1)a^{n-2}\frac{[c^\va_r(a,t)]^2}{c^\va(a,t)}
+\frac{\va\kappa}{2}\left\{b^{n-1}\frac{[c^\va_r(b,t)]^2}{c^\va(b,t)}
+a^{n-1}\frac{[c^\va_r(a,t)]^2}{c^\va(a,t)}
\right\}\\
&+\ka b^{n-1}[\lambda-c^\va(b,t)]w^\va(b,t)
+\ka a^{n-1}[\lambda-c^\va(a,t)]w^\va(a,t)
\end{split}
\ee
for each $t>0$.\\
(ii) If $\kappa=0$. Then
\be\label{b25}
\sup_{t>0}\|c^\va(t)\|_{L^\infty}\leq \|c_0\|_{L^\infty},\qquad \sup_{t>0}\int_a^b r^{n-1}w^\va(t)dr=\int_a^b r^{n-1}w_0dr.
\ee
\end{lemma}
\begin{proof}
 The $L^\infty$-estimates of $c^\va$ in \eqref{a1} and \eqref{b25} follow directly from the maximum principle and the nonnegativity of $w^\va$ and $c^\va$.
  We proceed to estimating $\int_a^b r^{n-1}w^\va(t)dr$. Multiplying the first equation of \eqref{e1} by $r^{n-1}$ then integrating the resulting equation over $(a,b)$ and using integration by parts, one gets
\be\label{a3}
\frac{d}{dt}\int_a^b r^{n-1}w^\va dr=[r^{n-1}(w^\va_r-w^\va c^\va_r)]|_{r=a}^{r=b}=0,
\ee
thanks to the first boundary condition in \eqref{e2}. We thus derive the $L^1$ estimate for $w^\va$ in \eqref{a1} and \eqref{b25} by integrating \eqref{a3} with respect to $t$.

It remains to prove \eqref{a2}. First, multiplying the first equation of \eqref{e1} with $r^{n-1}\log w^\va$ and integrating the resulting equation over $(a,b)$, we have
\be\label{a4}
\frac{d}{dt}\int_a^b r^{n-1}(w^\va\log w^\va-w^\va+1)dr
+4\int_a^b r^{n-1}(\partial_r \sqrt{w^\va})^2dr
=\int_a^b r^{n-1}w^\va_r c^\va_rdr.
\ee
It follows from the second equation of \eqref{e1} that
\ben
\begin{split}
2\partial_t\sqrt{c^\va}=\frac{c^\va_t}{\sqrt{c^\va}}
=&\va \frac{c^\va_{rr}}{\sqrt{c^\va}}
+\va \frac{n-1}{r}\frac{c^\va_r}{\sqrt{c^\va}}-w^\va\sqrt{c^\va}\\
=& 2\va\left[\partial^2_r\sqrt{c^\va}+\frac{(n-1)\partial_r
\sqrt{c^\va}}{r}\right]+2\va \frac{(\partial_r\sqrt{c^\va})^2}{\sqrt{c^\va}}-w^\va\sqrt{c^\va},
\end{split}
\enn
which gives rise to
\ben
\begin{split}
\partial_t (\partial_r\sqrt{c^\va})^2=&2(\partial_t\sqrt{c^\va})_r
\cdot(\partial_r\sqrt{c^\va})\\
=&2\va\left[\partial^2_r\sqrt{c^\va}+\frac{(n-1)\partial_r
\sqrt{c^\va}}{r}\right]_r
\cdot(\partial_r\sqrt{c^\va})\\
&+2\va \left[\frac{(\partial_r\sqrt{c^\va})^2}{\sqrt{c^\va}}\right]_r
\cdot(\partial_r\sqrt{c^\va})-(w^\va\sqrt{c^\va})_r
\cdot(\partial_r\sqrt{c^\va}),
\end{split}
\enn
where
\ben
\begin{split}
2\va&\left[\partial^2_r\sqrt{c^\va}+\frac{(n-1)\partial_r
\sqrt{c^\va}}{r}\right]_r
\cdot(\partial_r\sqrt{c^\va})\\
=&2\va\left\{\left[\partial^2_r\sqrt{c^\va}
+\frac{(n-1)\partial_r
\sqrt{c^\va}}{r}\right]
\cdot(\partial_r\sqrt{c^\va})\right\}_r
-2\va\left[\partial^2_r\sqrt{c^\va}+\frac{(n-1)\partial_r
\sqrt{c^\va}}{r}\right]
\cdot\partial_r^2\sqrt{c^\va}\\
=&\va\left\{\left[(\partial_r\sqrt{c^\va})^2\right]_{r}
+\frac{n-1}{r}(\partial_r\sqrt{c^\va})^2
\right\}_r
+\va\left[\frac{n-1}{r}(\partial_r\sqrt{c^\va})^2\right]_r
-2\va (\partial^2_r\sqrt{c^\va})^2
- \va\frac{n-1}{r}\left[(\partial_r\sqrt{c^\va})^2\right]_r\\
=&\va \frac{\left\{r^{n-1}\left[(\partial_r\sqrt{c^\va})^2\right]_{r}
\right\}_r}
{r^{n-1}}
-2\va (\partial^2_r\sqrt{c^\va})^2
-2\va \frac{n-1}{r^2}(\partial_r\sqrt{c^\va})^2
\end{split}
\enn
and
\ben
2\va \left[\frac{(\partial_r\sqrt{c^\va})^2}{\sqrt{c^\va}}\right]_r
\cdot(\partial_r\sqrt{c^\va})
=-2\va\frac{(\pt_r\sq{c^\va})^4}{c^\va}
+4\va\frac{(\pt_r\sq{c^\va})^2(\pt_r^2\sq{c^\va})}{\sq{c^\va}}.
\enn
Hence
\be\label{a5}
\begin{split}
\partial_t (\partial_r\sqrt{c^\va})^2
=&\va \frac{\left\{r^{n-1}\left[(\partial_r\sqrt{c^\va})^2\right]_{r}
\right\}_r}
{r^{n-1}}
-2\va\left[\pt^2_r\sq{c^\va}-\frac{(\pt_r\sq{c^\va})^2}{\sq{c^\va}}
\right]^2\\
&-2\va \frac{n-1}{r^2}(\partial_r\sqrt{c^\va})^2
-(w^\va\sqrt{c^\va})_r
\cdot(\partial_r\sqrt{c^\va})\\
=&\va \frac{\left\{r^{n-1}\left[(\partial_r\sqrt{c^\va})^2\right]_{r}
\right\}_r}
{r^{n-1}}
-2\va\left[\frac{\sq{c^\va}}{2}\left(\frac{ c^\va_{rr}}{c^\va}-\frac{(c^\va_r)^2}{(c^\va)^2}\right)
\right]^2\\
&-2\va \frac{n-1}{r^2}(\partial_r\sqrt{c^\va})^2
-(w^\va\sqrt{c^\va})_r
\cdot(\partial_r\sqrt{c^\va})\\
=&\va \frac{\left\{r^{n-1}\left[(\partial_r\sqrt{c^\va})^2\right]_{r}
\right\}_r}
{r^{n-1}}
-\frac{\va}{2}c^\va\left[(\log c^\va)_{rr}
\right]^2\\
&-2\va \frac{n-1}{r^2}(\partial_r\sqrt{c^\va})^2
-w^\va
(\partial_r\sqrt{c^\va})^2
-\frac{1}{2}w^\va_rc^\va_r.
\end{split}
\ee
Multiplying \eqref{a5} with $2r^{n-1}$ and integrating the resulting equality over $(a,b)$, one derives
\be\label{a6}
\begin{split}
2\frac{d}{dt}&\int_a^b r^{n-1} (\partial_r\sqrt{c^\va})^2dr
+\va\int_a^b r^{n-1}c^{\va} \left[(\log c^\va)_{rr}
\right]^2dr\\
&+4\va(n-1)\int_a^b r^{n-3}(\partial_r\sqrt{c^\va})^2dr
+2\int_a^b r^{n-1}w^\va
(\partial_r\sqrt{c^\va})^2 dr\\
=&2\va \left\{r^{n-1}\left[(\partial_r\sqrt{c^\va})^2\right]_{r}
\right\}\Big|_{r=a}^{r=b}
-\int_a^br^{n-1}w^\va_rc^\va_r dr.
\end{split}
\ee
To estimate the first term on the right-hand side of \eqref{a6}, we multiply the second equation of \eqref{e1} with $\ka\left(1-\frac{\lambda}{c^\va}\right)$ and deduce that
\be\label{b6}
\begin{split}
\frac{d}{dt}&\ka (\lambda\log{\frac{\lambda}{c^\va}}-\lambda+c^\va)\\
&=\ka c^\va_t\left(1-\frac{\lambda}{c^\va}\right)\\
&=\ka\left(1-\frac{\lambda}{c^\va}\right)\va c^\va_{rr}
+\ka\left(1-\frac{\lambda}{c^\va}\right)\va\frac{n-1}{r}c^\va_r
-\ka\left(1-\frac{\lambda}{c^\va}\right)w^\va c^\va,
\end{split}
\ee
which, along with the boundary condition
\ben
\frac{c^\va_r(b,t)}{c^\va(b,t)}
=-\ka\big[1-\frac{\lambda}{c^\va(b,t)}\big]
\enn
 and the identity
\be\label{b5}
\frac{c^\va_rc^\va_{rr}}{c^\va}=2[(\partial_r \sqrt{c^\va})^2]_{r}+\frac{(c^\va_r)^3}{2(c^\va)^2}
\ee
gives rise to
\be\label{b7}
\begin{split}
&\frac{d}{dt}\ka \big[\lambda\log{\frac{\lambda}{c^\va(b,t)}}-\lambda+c^\va(b,t)\big]
\\
=&-2\va \big[(\pt_r\sq{c^\va})^2\big]_r(b,t)
-\frac{\va\kappa\lambda}{2}\frac{[c^\va_r(b,t)]^2}{[c^\va(b,t)]^2}
+\frac{\va\kappa}{2}\frac{[c^\va_r(b,t)]^2}{c^\va(b,t)}\\
&-\va\frac{n-1}{b}\frac{[c^\va_r(b,t)]^2}{c^\va(b,t)}
-\ka\big[1-\frac{\lambda}{c^\va(b,t)}\big]w^\va(b,t) c^\va(b,t).
\end{split}
\ee
Similarly, it follows from \eqref{b6}, the boundary condition
$\frac{c^\va_r(a,t)}{c^\va(a,t)}
=\ka\big[1-\frac{\lambda}{c^\va(a,t)}\big]$ and \eqref{b5} that
\be\label{b8}
\begin{split}
&\frac{d}{dt}\ka \big[\lambda\log{\frac{\lambda}{c^\va(a,t)}}-\lambda+c^\va(a,t)\big]
\\
=&2\va \big[(\pt_r\sq{c^\va})^2\big]_r(a,t)
-\frac{\va\kappa\lambda}{2}\frac{[c^\va_r(a,t)]^2}{[c^\va(a,t)]^2}
+\frac{\va\kappa}{2}\frac{[c^\va_r(a,t)]^2}{c^\va(a,t)}\\
&+\va\frac{n-1}{a}\frac{[c^\va_r(a,t)]^2}{c^\va(a,t)}
-\ka\big[1-\frac{\lambda}{c^\va(a,t)}\big]w^\va(a,t) c^\va(a,t).
\end{split}
\ee
Substituting \eqref{b7} and \eqref{b8} into \eqref{a6} one gets
\ben
\begin{split}
2\frac{d}{dt}&\int_a^b r^{n-1} (\partial_r\sqrt{c^\va})^2dr
+\ka b^{n-1}\frac{d}{dt} \big[\lambda\log{\frac{\lambda}{c^\va(b,t)}}-\lambda+c^\va(b,t)\big]\\
&+\ka a^{n-1}\frac{d}{dt} \big[\lambda\log{\frac{\lambda}{c^\va(a,t)}}-\lambda+c^\va(a,t)\big]
+\va\int_a^b r^{n-1}c^{\va} \left[(\log c^\va)_{rr}
\right]^2dr\\
&+\frac{\va\kappa\lambda}{2}\left\{b^{n-1}\frac{[c^\va_r(b,t)]^2}{[c^\va(b,t)]^2}
+a^{n-1}\frac{[c^\va_r(a,t)]^2}{[c^\va(a,t)]^2}
\right\}\\
\leq&
-\int_a^br^{n-1}w^\va_rc^\va_r dr
+\va(n-1)a^{n-2}\frac{[c^\va_r(a,t)]^2}{c^\va(a,t)}\\
&+\frac{\va\kappa}{2}\left\{b^{n-1}\frac{[c^\va_r(b,t)]^2}{c^\va(b,t)}
+a^{n-1}\frac{[c^\va_r(a,t)]^2}{c^\va(a,t)}
\right\}\\
&+\ka b^{n-1}[\lambda-c^\va(b,t)]w^\va(b,t)
+\ka a^{n-1}[\lambda-c^\va(a,t)]w^\va(a,t),
\end{split}
\enn
which, added to \eqref{a4} gives \eqref{a2}. The proof is completed.

\end{proof}

\begin{lemma}\label{l10} Suppose the assumptions in Theorem \ref{t1} hold. Let $(w^\va,c^\va)$ with $0<\va<1$ be the global solutions of \eqref{e1}-\eqref{e2}, derived in Lemma \ref{l9}.
\\
(i) If $\kappa>0$. Then for any $0<T<\infty$, there exists a constant $C_{\lambda}$ depending on $a$, $b$, $n$, $\|w_0\|_{L^2}$, $\|c_0\|_{H^1}$, $T$ and $\lambda$ such that
\be\label{b16}
\begin{split}
\sup_{0<t\leq T}&\left[\int_a^b r^{n-1}(w^\va\log w^\va-w^\va+1)(t)dr+
\int_a^b r^{n-1} (\partial_r\sqrt{c^\va})^2(t)dr\right]\\
+&\int_0^T\int_a^b r^{n-1}(\partial_r \sqrt{w^\va})^2drdt+\va\int_0^T\int_a^b r^{n-1}c^{\va} \left[(\log c^\va)_{rr}
\right]^2drdt\\
+&\frac{\va\kappa\lambda}{4}\int_0^T\left\{b^{n-1}\frac{[c^\va_r(b,t)]^2}{[c^\va(b,t)]^2}
+a^{n-1}\frac{[c^\va_r(a,t)]^2}{[c^\va(a,t)]^2}
\right\}dt\\
\leq
C_{\lambda}&(1+e^{C_{\lambda} \kappa^2}).
\end{split}
\ee
\\
(ii) If $\kappa=0$.
Then for any $0<T<\infty$, there exists a constant $C$ depending on $a$, $b$, $n$, $\|w_0\|_{L^2}$, $\|c_0\|_{H^1}$ and $T$ such that
\be\label{b17}
\begin{split}
\sup_{0<t\leq T}&\left[\int_a^b r^{n-1}(w^\va\log w^\va-w^\va+1)(t)dr+
\int_a^b r^{n-1} (\partial_r\sqrt{c^\va})^2(t)dr\right]\\
+&\int_0^T\int_a^b r^{n-1}(\partial_r \sqrt{w^\va})^2drdt
+\va\int_0^T\int_a^b r^{n-1}c^{\va} \left[(\log c^\va)_{rr}
\right]^2drdt
\leq
C.
\end{split}
\ee
\end{lemma}
\begin{proof}
We first prove \eqref{b16}. To bound the first three terms on the right-hand side of \eqref{a2}, we employ the Gagliardo-Nirenberg interpolation inequality to deduce that
\be\label{b10}
\begin{split}
\|\partial_r\sqrt{c^\va}\|_{L^\infty}^2
\leq& C_0\|\partial_r^2\sqrt{c^\va}\|_{L^2}
\|\partial_r\sqrt{c^\va}\|_{L^2}
+C_0\|\partial_r\sqrt{c^\va}\|_{L^2}^2\\
\leq &\theta \|r^{\frac{n-1}{2}}\pt_r^2\sqrt{c^\va}
\|_{L^2}^2+\frac{C_1}{\theta}\|r^{\frac{n-1}{2}}\pt_r\sqrt{c^\va}
\|_{L^2}^2
\end{split}
\ee
for any $0<\theta\leq1$, where the constants $C_0$ and $C_{1}$ depend on $a$, $b$ and $n$. A direct computation gives
\ben
\begin{split}
\partial_r^2\sqrt{c^\va}=\frac{c^\va_{rr}}{2\sqrt{c^\va}}
-\frac{(c^\va_r)^2}{4c^\va\sqrt{c^\va}},
\end{split}
\enn
which, along with \eqref{b10} leads to
\be\label{b11}
\begin{split}
\|\partial_r\sqrt{c^\va}\|_{L^\infty}^2
\leq \frac{\theta}{2}\int_a^b r^{n-1}\frac{(c^\va_{rr})^2}{c^\va}\,dr
+\frac{\theta}{8}\int_{a}^b r^{n-1}\frac{(c^\va_r)^4}{(c^\va)^3}\,dr
+\frac{C_1}{\theta}\|r^{\frac{n-1}{2}}\pt_r\sqrt{c^\va}
\|_{L^2}^2.
\end{split}
\ee
Integration by parts entails that
\be\label{b47}
\begin{split}
2\int_a^b r^{n-1} \frac{(c^\va_r)^4}{(c^\va)^3}dr
=&2\left[r^{n-1}\frac{(c^\va_r)^3}{(c^\va)^3}\cdot c^\va
\right]\Big|_{r=a}^{r=b}
-2(n-1) \int_a^b r^{n-2} \frac{(c^\va_r)^3}{(c^\va)^2}dr\\
&-6\int_a^b r^{n-1}\frac{(c^\va_r)^2}{(c^\va)^2}\cdot (\log c^\va)_{rr}\cdot c^\va dr\\
:=&I_1+I_2+I_3.
\end{split}
\ee
It follows from the fact $c^\va>0$ in $(a,b)\times(0,\infty)$ and the boundary conditions $c^\va_r(b,t)=\kappa[\lambda-c^{\va}(b,t)]$, $c^\va_r(a,t)=-\kappa[\lambda-c^{\va}(a,t)]$ that
\ben
I_1\leq 2\kappa\lambda\left\{b^{n-1}\frac{[c^\va_r(b,t)]^2}{[c^\va(b,t)]^2}
+a^{n-1}\frac{[c^\va_r(a,t)]^2}{[c^\va(a,t)]^2}
\right\}.
\enn
The Cauchy-Schwarz inequality gives
\be\label{b50}
I_3\leq \frac{1}{2}\int_a^b r^{n-1} \frac{(c^\va_r)^4}{(c^\va)^3}dr
+C_2\int_a^b r^{n-1}c^\va [(\log c^\va)_{rr}]^2dr,
\ee
where the constant $C_2$ depends on $a$, $b$ and $n$.
 The Cauchy-Schwarz inequality and \eqref{a1} imply that  there exists a constant $C_3$ depending on $a$, $b$ and $n$, such that for any $\eta>0$, the following holds true:
 \ben
 \begin{split}
 I_2\leq& \eta \int_a^b r^{n-1}\frac{(c^\va_r)^4}{(c^\va)^\frac{8}{3}}dr
 +\frac{C_3}{\eta}\\
 \leq &\eta (\|c_0\|_{L^\infty}+\lambda)^{\frac{1}{3}}
 \int_a^b r^{n-1}\frac{(c^\va_r)^4}{(c^\va)^3}dr
 +\frac{C_3}{\eta}.
 \end{split}
 \enn
 Setting $\eta=\frac{1}{2}(\|c_0\|_{L^\infty}+\lambda)^{-\frac{1}{3}}$, the above inequality gives rise to
 \be\label{b49}
 I_2\leq \frac{1}{2} \int_a^b r^{n-1}\frac{(c^\va_r)^4}{(c^\va)^3}dr
 +2C_3(\|c_0\|_{L^\infty}+\lambda)^{\frac{1}{3}}.
 \ee
 Substituting the above estimates for $I_1$ - $I_3$ into \eqref{b47}, one gets
\be\label{b12}
\begin{split}
\int_a^b r^{n-1} \frac{(c^\va_r)^4}{(c^\va)^3}dr
\leq & 2\kappa\lambda\left\{b^{n-1}\frac{[c^\va_r(b,t)]^2}{[c^\va(b,t)]^2}
+a^{n-1}\frac{[c^\va_r(a,t)]^2}{[c^\va(a,t)]^2}
\right\}\\
&+C_2\int_a^b r^{n-1}c^\va [(\log c^\va)_{rr}]^2dr
+2C_3(\|c_0\|_{L^\infty}+\lambda)^{\frac{1}{3}}.
\end{split}
\ee
A direct calculation yields
\be\label{b48}
\begin{split}
\int_a^b r^{n-1} \frac{(c^\va_{rr})^2}{c^\va} dr
\leq& 2 \int_a^b r^{n-1} c^\va
\left[\frac{c^\va_{rr}}{c^\va}-\frac{(c^\va_r)^2}{(c^\va)^2}
\right]^2 dr
+2 \int_a^b r^{n-1} c^\va\frac{(c^\va_r)^4}{(c^\va)^4}
dr\\
=&2\int_a^b  r^{n-1}c^\va[(\log c^\va)_{rr}]^2dr
+2\int_a^b r^{n-1}\frac{(c^\va_r)^4}{(c^\va)^3}
dr,
\end{split}
\ee
which, along with \eqref{b12} gives
\be\label{b13}
\begin{split}
\int_a^b r^{n-1} \frac{(c^\va_{rr})^2}{c^\va} dr
\leq&
4\kappa\lambda\left\{b^{n-1}\frac{[c^\va_r(b,t)]^2}{[c^\va(b,t)]^2}
+a^{n-1}\frac{[c^\va_r(a,t)]^2}{[c^\va(a,t)]^2}
\right\}\\
&+2(C_2+1)\int_a^b  r^{n-1}c^\va[(\log c^\va)_{rr}]^2dr+4C_3(\|c_0\|_{L^\infty}+\lambda)^{\frac{1}{3}}.
\end{split}
\ee
Substituting \eqref{b12} and \eqref{b13} into \eqref{b11}, one immediately gets
\be\label{b14}
\begin{split}
\|\partial_r\sqrt{c^\va}\|_{L^\infty}^2
\leq& \frac{9}{4}\kappa\lambda\theta\left\{b^{n-1}\frac{[c^\va_r(b,t)]^2}{[c^\va(b,t)]^2}
+a^{n-1}\frac{[c^\va_r(a,t)]^2}{[c^\va(a,t)]^2}
\right\}\\
&+\big(\frac{9}{8}C_2+1\big)\theta\int_a^b  r^{n-1}c^\va[(\log c^\va)_{rr}]^2dr
+\frac{C_1}{\theta}\|r^{\frac{n-1}{2}}\pt_r\sqrt{c^\va}
\|_{L^2}^2\\
&+\frac{9}{4}C_3\theta(\|c_0\|_{L^\infty}+\lambda)^{\frac{1}{3}}.
\end{split}
\ee
Let
\ben
\begin{split}
\theta=\min\{&[72(n-1)a^{n-2}]^{-1}, [(18C_2+16)(n-1)a^{n-2}]^{-1}, \\ &\ \ [36(a^{n-1}+b^{n-1})\kappa]^{-1}, [(9C_2+8)(a^{n-1}+b^{n-1})\kappa]^{-1},1\}.
\end{split}
\enn
Then from \eqref{b14} one deduces that
\be\label{b15}
\begin{split}
\va(n-1)a^{n-2}&\frac{[c^\va_r(a,t)]^2}{c^\va(a,t)}
+\frac{\va\kappa}{2}\left\{b^{n-1}\frac{[c^\va_r(b,t)]^2}{c^\va(b,t)}
+a^{n-1}\frac{[c^\va_r(a,t)]^2}{c^\va(a,t)}
\right\}\\
\leq &4\va(n-1)a^{n-2} \|\partial_r\sqrt{c^\va}\|_{L^\infty}^2
+2\va\kappa (a^{n-1}+b^{n-1})\|\partial_r\sqrt{c^\va}\|_{L^\infty}^2\\
\leq& \frac{\va\kappa\lambda}{4}\left\{b^{n-1}\frac{[c^\va_r(b,t)]^2}{[c^\va(b,t)]^2}
+a^{n-1}\frac{[c^\va_r(a,t)]^2}{[c^\va(a,t)]^2}
\right\}
+\frac{1}{2}\va\int_a^b  r^{n-1}c^\va[(\log c^\va)_{rr}]^2dr\\
&+C_0(\kappa+1)\|r^{\frac{n-1}{2}}\pt_r\sqrt{c^\va}
\|_{L^2}^2+C_0(\|c_0\|_{L^\infty}+\lambda)^{\frac{1}{3}}
\end{split}
\ee
for each $t>0$, where the constant $C_0$ depends on $a$, $b$ and $n$.
By the Gagliardo-Nirenberg interpolation inequality, the fourth and fifth term on the right-hand side of \eqref{a2} enjoy the following bounds:
\ben
\begin{split}
 \ka b^{n-1}&[\lambda-c^\va(b,t)]w^\va(b,t)
+\ka a^{n-1}[\lambda-c^\va(a,t)]w^\va(a,t)\\
\leq & \lambda\kappa(a^{n-1}+b^{n-1})\|\sqrt{w^\va}\|_{L^{\infty}}\\
\leq&C_0\lambda\kappa (\|\sqrt{w^\va}\|_{L^{2}}^{\frac{1}{2}}\|\partial_r\sqrt{w^\va}\|_{L^{2}}^{\frac{1}{2}}
+\|\sqrt{w^\va}\|_{L^{2}})\\
\leq& \|r^{\frac{n-1}{2}}\partial_r\sqrt{w^\va}\|_{L^{2}}^2
+C_0\|\sqrt{w^\va}\|_{L^{2}}^2 +C_0 \lambda^2 \kappa^2
\end{split}
\enn
for each $t>0$, which, along with \eqref{b15} and \eqref{a2} leads to
\ben
\begin{split}
\frac{d}{dt}&\int_a^b r^{n-1}(w^\va\log w^\va-w^\va+1)dr+
2\frac{d}{dt}\int_a^b r^{n-1} (\partial_r\sqrt{c^\va})^2dr\\
&+\ka b^{n-1}\frac{d}{dt} \big[\lambda\log{\frac{\lambda}{c^\va(b,t)}}-\lambda+c^\va(b,t)\big]
+\ka a^{n-1}\frac{d}{dt} \big[\lambda\log{\frac{\lambda}{c^\va(a,t)}}-\lambda+c^\va(a,t)\big]\\
&+3\int_a^b r^{n-1}(\partial_r \sqrt{w^\va})^2dr+\frac{\va}{2}\int_a^b r^{n-1}c^{\va} \left[(\log c^\va)_{rr}
\right]^2dr\\
&+\frac{\va\kappa\lambda}{4}\left\{b^{n-1}\frac{[c^\va_r(b,t)]^2}{[c^\va(b,t)]^2}
+a^{n-1}\frac{[c^\va_r(a,t)]^2}{[c^\va(a,t)]^2}
\right\}\\
\leq&
C_0\|w^\va\|_{L^{1}}
+C_0(\kappa+1)\|r^{\frac{n-1}{2}}\pt_r\sqrt{c^\va}
\|_{L^2}^2
+C_0(\|c_0\|_{L^\infty}+\lambda)^{\frac{1}{3}}+C_0\lambda^2\kappa^2.
\end{split}
\enn
Applying Gronwall's inequality to this inequality, using \eqref{a1} and the following fact
\be\label{c2}
\begin{split}
&\qquad\quad(w^\va\log w^\va-w^\va+1)>0 \quad \text{for}\ \ (r,t)\in (a,b)\times [0,\infty),\\
&\lambda\log{\frac{\lambda}{c^\va(b,t)}}-\lambda+c^\va(b,t),\ \ \lambda\log{\frac{\lambda}{c^\va(a,t)}}-\lambda+c^\va(a,t)>0
\quad \text{for}\ \ t\in[0,\infty)
\end{split}
\ee
thanks to the nonnegativity of $w^\va$ and $c^\va$, one derives \eqref{b16}. We proceed to proving \eqref{b17}. When $\kappa=0$, from \eqref{e2} we know that
 \be\label{c4}
 w^\va_r(a,t)=w^\va_r(b,t)=c^\va_r(a,t)=c^\va_r(b,t)=0,
  \ee
 which, indicates that
\be\label{c1}
2\va \left\{r^{n-1}\left[(\partial_r\sqrt{c^\va})^2\right]_{r}
\right\}\Big|_{r=a}^{r=b}=0.
 \ee
 Thus by adding \eqref{a4} to \eqref{a6} and using \eqref{c1} one immediately gets
\ben
\begin{split}
\frac{d}{dt}&\int_a^b r^{n-1}(w^\va\log w^\va-w^\va+1)dr+
2\frac{d}{dt}\int_a^b r^{n-1} (\partial_r\sqrt{c^\va})^2dr
+4\int_a^b r^{n-1}(\partial_r \sqrt{w^\va})^2dr\\
&+\va\int_a^b r^{n-1}c^{\va} \left[(\log c^\va)_{rr}
\right]^2dr
+4\va(n-1)\int_a^b r^{n-3}(\partial_r\sqrt{c^\va})^2dr
+2\int_a^b r^{n-1}w^\va
(\partial_r\sqrt{c^\va})^2 dr=0.
\end{split}
\enn
Integrating the above equality with respect to $t$ and using \eqref{c2} one gets \eqref{b17}.
 The proof is completed.

\end{proof}

We proceed to derive the higher order estimates for $(w^\va,c^\va)$.
\begin{lemma}\label{l2} Let the assumptions in Theorem \ref{t1} hold and let $(w^\va,c^\va)$ with $0<\va<1$ be the global solutions of \eqref{e1}-\eqref{e2}, derived in Lemma \ref{l9}.\\
(i) If $\kappa>0$.
 Then for any $0<T<\infty$, there exists a constant $C_{\lambda}$ depending on $a$, $b$, $n$, $\|w_0\|_{L^2}$, $\|c_0\|_{H^1}$, $T$ and $\lambda$ such that
\be\label{b18}
\begin{split}
\sup_{0<t\leq T}&\left[\int_a^b r^{n-1}(c^\va_r)^2(t)dr
+\int_{a}^b r^{n-1}(w^\va)^2(t)dr
\right]\\
&\ \ +\int_0^T\int_{a}^b r^{n-1}(w^\va_r)^2drdt
+\int_0^T\int_a^b r^{n-1}(c^\va_t)^2 drdt
\leq C_{\lambda}(1+\exp\{\exp\{C_{\lambda}\ka^2\}\}).
\end{split}
\ee
(ii) If $\kappa=0$.  Then for any $0<T<\infty$, there exists a constant $C$ depending on $a$, $b$, $n$, $\|w_0\|_{L^2}$, $\|c_0\|_{H^1}$ and $T$ such that
\be\label{b19}
\begin{split}
\sup_{0<t\leq T}&\left[\int_a^b r^{n-1}(c^\va_r)^2(t)dr
+\int_{a}^b r^{n-1}(w^\va)^2(t)dr
\right]\\
&\ \ +\int_0^T\int_{a}^b r^{n-1}(w^\va_r)^2drdt
+\int_0^T\int_a^b r^{n-1}(c^\va_t)^2 drdt
\leq C.
\end{split}
\ee
\end{lemma}

\begin{proof}
We first prove \eqref{b18}.
\eqref{a1} and \eqref{b16} leads to
\be\label{a9}
\begin{split}
\sup_{0<t\leq T}\int_a^b r^{n-1}(c^\va_r)^2(t)dr
\leq& 4\left[\sup_{0<t\leq T}\int_a^b r^{n-1}(\pt_r\sq{c^\va})^2(t)dr
\right]\cdot
\left[\sup_{t>0}\|c^\va(t)\|_{L^\infty}
\right]\\
\leq& C_{\lambda}(1+e^{C_{\lambda}\ka^2}).
\end{split}
\ee
Multiplying the first equation of \eqref{e1} with $r^{n-1}w^\va$, then integrating the resulting equality over $(a,b)$ and employing the Gagliardo-Nirenberg interpolation inequality to have
\ben
\begin{split}
\frac{1}{2}\frac{d}{dt}&\int_{a}^b r^{n-1}(w^\va)^2dr
+\int_{a}^b r^{n-1}(w^\va_r)^2dr\\
=&\int_a^b r^{n-1}w^\va w^\va_r c^\va_r dr\\
\leq & \|r^{\frac{n-1}{2}}c^\va_r\|_{L^2}\|r^{\frac{n-1}{2}}w^\va_r\|_{L^2}
\|w^\va\|_{L^\infty}
\\
\leq &C_0\|r^{\frac{n-1}{2}}c^\va_r\|_{L^2}\|r^{\frac{n-1}{2}}w^\va_r\|_{L^2}
\Big(\|r^{\frac{n-1}{2}}w^\va\|_{L^2}
+\|r^{\frac{n-1}{2}}w^\va\|^{\frac{1}{2}}_{L^2}
\|r^{\frac{n-1}{2}}w^\va_r\|^{\frac{1}{2}}_{L^2}
\Big)\\
\leq &\frac{1}{2}\|r^{\frac{n-1}{2}}w^\va_r\|^2_{L^2}
+C_0(\|r^{\frac{n-1}{2}}c_r^\va\|^2_{L^2}
+\|r^{\frac{n-1}{2}}c_r^\va\|^4_{L^2})
\|r^{\frac{n-1}{2}}w^\va\|^2_{L^2},
\end{split}
\enn
that is,
\be\label{b20}
\begin{split}
\frac{d}{dt}\int_{a}^b r^{n-1}(w^\va)^2dr
+\int_{a}^b r^{n-1}(w^\va_r)^2dr
\leq
C_0(\|r^{\frac{n-1}{2}}c_r^\va\|^2_{L^2}
+\|r^{\frac{n-1}{2}}c_r^\va\|^4_{L^2})
\|r^{\frac{n-1}{2}}w^\va\|^2_{L^2},
\end{split}
\ee
which, along with Gronwall's inequality and \eqref{a9} gives
\be\label{a10}
\sup_{0<t\leq T}\int_{a}^b r^{n-1}(w^\va)^2(t)dr
+\int_0^T\int_{a}^b r^{n-1}(w^\va_r)^2drdt
\leq C_{\lambda}(1+\exp\{\exp\{C_{\lambda}\ka^2\}\}).
\ee
Using \eqref{a1}, \eqref{b13} and \eqref{b16}, one has
\ben
\begin{split}
\va\int_0^T\int_a^b r^{n-1}(c^\va_{rr})^2 drdt
\leq& \sup_{t>0}\|c^\va(t)\|_{L^\infty}
\cdot\va\int_0^T\int_a^b r^{n-1}\frac{(c^\va_{rr})^2}{c^\va} drdt\\
\leq &C_0\kappa\lambda\va(\|c_0\|_{L^\infty}+\lambda)
\int_0^T\left\{b^{n-1}\frac{[c^\va_r(b,t)]^2}{[c^\va(b,t)]^2}
+a^{n-1}\frac{[c^\va_r(a,t)]^2}{[c^\va(a,t)]^2}
\right\}dt\\
&+C_0\va(\|c_0\|_{L^\infty}+\lambda)
\int_0^T\int_a^b  r^{n-1}c^\va[(\log c^\va)_{rr}]^2drdt\\
&+C_0\va (\|c_0\|_{L^\infty}+\lambda)^{\frac{4}{3}}\\
\leq& C_{\lambda}(1+e^{C_{\lambda}\ka^2}),
\end{split}
\enn
which, along with the second equation of \eqref{e1}, \eqref{a9}, \eqref{a10}, \eqref{a1} and the fact $0<\va<1$ gives rise to
\be\label{a17}
\begin{split}
\int_0^T\int_a^b r^{n-1} (c^\va_t)^2drdt
\leq &C_0\va^2\int_0^T\int_a^b r^{n-1} (c^{\va}_{rr})^2drdt
+C_0(n-1)\va^2\int_0^T\int_a^b r^{n-3} (c^\va_r)^2 drdt\\
&+C_0\int_0^T\int_a^b r^{n-1} (w^\va)^2drdt\cdot\sup_{t>0}\|c^\va(t)\|_{L^\infty}^2\\
\leq& C_{\lambda}(1+\exp\{\exp\{C_{\lambda}\ka^2\}\}).
\end{split}
\ee
Collecting \eqref{a9}, \eqref{a10} and \eqref{a17} one derives \eqref{b18}.

 We proceed to proving \eqref{b19}. Similar to the derivation of \eqref{a9}, \eqref{b25} and \eqref{b17} lead to
 \be\label{b21}
\begin{split}
\sup_{0<t\leq T}\int_a^b r^{n-1}(c^\va_r)^2(t)dr
\leq 4\left[\sup_{0<t\leq T}\int_a^b r^{n-1}(\pt_r\sq{c^\va})^2(t)dr
\right]\cdot
\left[\sup_{t>0}\|c^\va(t)\|_{L^\infty}
\right]
\leq C.
\end{split}
\ee
Applying Gronwall's inequality to \eqref{b20} and using
 \eqref{b21} one gets
\be\label{b22}
\sup_{0<t\leq T}\int_{a}^b r^{n-1}(w^\va)^2(t)dr
+\int_0^T\int_{a}^b r^{n-1}(w^\va_r)^2drdt\leq C.
\ee
Similar to the derivation of \eqref{b49}, the Cauchy-Schwarz inequality and \eqref{b25} entail that
\be\label{b52}
-2(n-1)\int_a^b r^{n-2}\frac{(c^\va_r)^3}{(c^\va)^2}dr
\leq \frac{1}{2}\int_a^b r^{n-1}
\frac{(c^\va_r)^4}{(c^\va)^3}dr
+2C_3 \|c_0\|_{L^\infty}^{\frac{1}{3}}.
\ee
Since $\kappa=0$, from \eqref{c4} one deduces that
\be\label{b51}
2\left[r^{n-1}\frac{(c^\va_r)^3}{(c^\va)^3}\cdot c^\va
\right]\Big|_{r=a}^{r=b}=0.
\ee
Substituting \eqref{b52}, \eqref{b51} and \eqref{b50} into \eqref{b47}, one gets
\be\label{b53}
\begin{split}
\int_a^b r^{n-1} \frac{(c^\va_r)^4}{(c^\va)^3}dr
\leq
C_2\int_a^b r^{n-1}c^\va [(\log c^\va)_{rr}]^2dr
+2C_3\|c_0\|_{L^\infty}^{\frac{1}{3}}
\end{split}
\ee
for each $t>0$, where the constant $C_2$ and $C_3$ are as in \eqref{b50} and \eqref{b49}, respectively. Plugging \eqref{b53} into \eqref{b48} to have
\ben
\begin{split}
\int_a^b r^{n-1} \frac{(c^\va_{rr})^2}{c^\va} dr
\leq 2(C_2+1)\int_a^b  r^{n-1}c^\va[(\log c^\va)_{rr}]^2dr
+4C_3\|c_0\|_{L^\infty}^{\frac{1}{3}},
\end{split}
\enn
which, along with \eqref{b25}, \eqref{b17} and the fact $0<\va<1$ indicates
\be\label{c3}
\begin{split}
\va\int_0^T\int_a^b r^{n-1} (c^\va_{rr})^2 drdt
\leq&\sup_{t>0}\|c^\va(t)\|_{L^\infty}\cdot2(C_2+1)\va\int_0^T\int_a^b  r^{n-1}c^\va[(\log c^\va)_{rr}]^2dr\\
&+\sup_{t>0}\|c^\va(t)\|_{L^\infty}\cdot4C_3\va \|c_0\|_{L^\infty}^{\frac{1}{3}}\\
\leq& C.
\end{split}
\ee
Then it follows from the second equation of \eqref{e1}, \eqref{b25}, \eqref{b21}, \eqref{b22}, \eqref{c3} and the fact $0<\va<1$ that
\ben
\begin{split}
\int_0^T\int_a^b r^{n-1} (c^\va_t)^2drdt
\leq &C_0\va^2\int_0^T\int_a^b r^{n-1} (c^{\va}_{rr})^2drdt
+C_0(n-1)\va^2\int_0^T\int_a^b r^{n-3} (c^\va_r)^2 drdt\\
&+C_0\int_0^T\int_a^b r^{n-1} (w^\va)^2drdt\cdot\sup_{t>0}\|c^\va(t)\|_{L^\infty}^2\\
\leq& C,
\end{split}
\enn
which, along with \eqref{b21} and \eqref{b22} gives \eqref{b19}.
 The proof is completed.

\end{proof}

\begin{lemma}\label{l3} Let the assumptions in Theorem \ref{t1} hold and let $(w^\va,c^\va)$ with $0<\va<1$ be the global solutions of \eqref{e1}-\eqref{e2}, derived in Lemma \ref{l9}. \\
(i) If $\kappa>0$.
Then for any $0<T<\infty$, there exists a constant $C_{\lambda}$ depending on $a$, $b$, $n$, $\|w_0\|_{H^2}$, $\|c_0\|_{H^2}$, $T$ and $\lambda$ such that
\be\label{c5}
\begin{split}
&\sup_{0<t\leq T} \left[\int_a^b r^{n-1} (w^\va_t)^2(t) dr
+\int_a^b r^{n-1} (w^\va_r)^2(t) dr
\right]\\
&\ \ \ +\int_0^T\int_a^b r^{n-1}(c^\va_{rt})^2drdt
+\int_0^T\int_a^b r^{n-1}(w^\va_{rt})^2 drdt
\leq C_{\lambda}(1+\exp\{\exp\{\exp\{C_{\lambda}\ka^2\}\}\}).
\end{split}
 \ee
 (ii) If $\kappa=0$.
Then for any $0<T<\infty$, there exists a constant $C$ depending on $a$, $b$, $n$, $\|w_0\|_{H^2}$, $\|c_0\|_{H^2}$ and $T$ such that
\be\label{c6}
\begin{split}
\sup_{0<t\leq T}& \left[\int_a^b r^{n-1} (w^\va_t)^2(t) dr
+\int_a^b r^{n-1} (w^\va_r)^2(t) dr
\right]\\
&
+\int_0^T\int_a^b r^{n-1}(c^\va_{rt})^2drdt
+\int_0^T\int_a^b r^{n-1}(w^\va_{rt})^2 drdt
\leq C.
\end{split}
 \ee
\end{lemma}
\begin{proof}
We first prove \eqref{c5}. Taking the $L^2$ inner product of the second equation of \eqref{e1} with $(r^{n-1}c_r^\va)_{rt}$ and using integration by parts, one gets
\be\label{a11}
\begin{split}
\frac{1}{2}\va\frac{d}{dt}\int_a^b\frac{[(r^{n-1}c^\va_r)_r]^2}{r^{n-1}}dr
+\int_a^b r^{n-1}(c^\va_{rt})^2 dr
=&\left[r^{n-1}c^\va_{rt}c^\va_t\right]\big|_{r=a}^{r=b}
+\left[r^{n-1}c^\va_{rt}w^\va c^\va\right]\big|_{r=a}^{r=b}\\
&-\int_a^br^{n-1}c^\va_{rt}(w^\va c^\va)_r dr.
\end{split}
\ee
It follows form \eqref{e2} that
\ben
[r^{n-1}c^\va_{rt}c^\va_{t}]|_{r=a}^{r=b}
=-\ka \{b^{n-1}[c^\va_t(b,t)]^2+a^{n-1}[c^\va_t(a,t)]^2\},
\enn
which, in conjunction with the Gagliardo-Nirenberg interpolation inequality entails that
\be\label{a12}
\begin{split}
\left[r^{n-1}c^\va_{rt}c^\va_t\right]\big|_{r=a}^{r=b}
\leq C_0\ka \|c^\va_t\|_{L^\infty}^2
\leq& C_0\ka(\|r^{\frac{n-1}{2}}c^\va_t\|^2_{L^2}
+\|r^{\frac{n-1}{2}}c^\va_t\|_{L^2}\|r^{\frac{n-1}{2}}c^\va_{rt}\|_{L^2})\\
\leq & \frac{1}{4}\|r^{\frac{n-1}{2}}c^\va_{rt}\|^2_{L^2}
+C_0(\ka+\ka^2)\|r^{\frac{n-1}{2}}c^\va_t\|^2_{L^2}.
\end{split}
\ee
\eqref{e2} further leads to
\ben
\begin{split}
[r^{n-1}c^\va_{rt}w^\va c^\va]|_{r=a}^{r=b}
=-\kappa [b^{n-1}(c^\va_t w^\va c^\va)(b,t)
+a^{n-1}(c^\va_t w^\va c^\va)(a,t)],
\end{split}
\enn
which, along with the Sobolev embedding inequality and \eqref{a1} gives
\be\label{a13}
\begin{split}
\big[&r^{n-1}c^\va_{rt} w^\va c^\va\big]\big|_{r=a}^{r=b}\\
&\leq C_0\ka\|c^\va\|_{L^\infty}\|c^\va_t\|_{L^\infty}
\|w^\va\|_{L^\infty}\\
&\leq C_0(\|c_0\|_{L^\infty}+\lambda)
\ka(\|r^{\frac{n-1}{2}}c^\va_t\|_{L^2}
+\|r^{\frac{n-1}{2}}c^\va_{rt}\|_{L^2})
(\|r^{\frac{n-1}{2}}w^\va\|_{L^2}
+\|r^{\frac{n-1}{2}}w^\va_{r}\|_{L^2})\\
&\leq \frac{1}{4}\|r^{\frac{n-1}{2}}c^\va_{rt}\|^2_{L^2}
+\|r^{\frac{n-1}{2}}c^\va_t\|_{L^2}^2
+C_0(\|c_0\|_{L^\infty}+\lambda)^2\ka^2
(\|r^{\frac{n-1}{2}}w^\va\|^2_{L^2}
+\|r^{\frac{n-1}{2}}w^\va_{r}\|^2_{L^2}).
\end{split}
\ee
By \eqref{a1} and the Sobolev embedding inequality one deduces
\be\label{a14}
\begin{split}
-\int_a^b&r^{n-1}c^\va_{rt}(w^\va c^\va)_r dr\\
=&-\int_a^b r^{n-1}c^\va_{rt}w^\va_r c^\va dr
-\int_a^b r^{n-1}c^\va_{rt}w^\va c^\va_r dr\\
\leq& \|c^\va\|_{L^\infty}\|r^{\frac{n-1}{2}}w^\va_r\|_{L^2}
\|r^{\frac{n-1}{2}}c^\va_{rt}\|_{L^2}
+\|w^\va\|_{L^\infty}\|r^{\frac{n-1}{2}}c^\va_r\|_{L^2}
\|r^{\frac{n-1}{2}}c^\va_{rt}\|_{L^2}\\
\leq& \frac{1}{4}\|r^{\frac{n-1}{2}}c^\va_{rt}\|^2_{L^2}
+C_0(\|c_0\|_{L^\infty}+\lambda)^2
\|r^{\frac{n-1}{2}}w^\va_r\|^2_{L^2}\\
&+C_0\|r^{\frac{n-1}{2}}c^\va_r\|^2_{L^2}
(\|r^{\frac{n-1}{2}}w^\va\|^2_{L^2}
+\|r^{\frac{n-1}{2}}w^\va_r\|^2_{L^2}).
\end{split}
\ee
Substituting \eqref{a12}-\eqref{a14} into \eqref{a11}, we obtain
 \ben
\begin{split}
\frac{1}{2}&\va\frac{d}{dt}\int_a^b\frac{[(r^{n-1}c^\va_r)_r]^2}{r^{n-1}}dr
+\frac{1}{4}\int_a^b r^{n-1}(c^\va_{rt})^2 dr\\
&\leq
C_0(1+\ka+\ka^2)\|r^{\frac{n-1}{2}}c^\va_t\|_{L^2}^2\\
&\ \ \ +C_0[(\|c_0\|_{L^\infty}+\lambda)^2 \ka^2+\|r^{\frac{n-1}{2}}c^\va_r\|_{L^2}^2]
(\|r^{\frac{n-1}{2}}w^\va\|_{L^2}^2
+\|r^{\frac{n-1}{2}}w^\va_r\|_{L^2}^2).
\end{split}
\enn
 Integrating this inequality with respect to $t$ and using \eqref{b18}, we arrive at
\be\label{a18}
\va\sup_{0<t\leq  T}\int_a^b\frac{[(r^{n-1}c^\va_r(t))_r]^2}{r^{n-1}}dr
+\int_0^T\int_a^b r^{n-1}(c^\va_{rt})^2 drdt
\leq C_{\lambda}(1+\exp\{\exp\{C_{\lambda}\ka^2\}\}).
\ee
Multiplying the first equation of \eqref{e1} with $r^{n-1}w^\va_t$ and using integration by parts to have
\be\label{a19}
\begin{split}
\frac{1}{2}\frac{d}{dt}&\int_a^b r^{n-1} (w^\va_r)^2 dr
+\int_a^b r^{n-1} (w^\va_t)^2 dr\\
=&\int_a^b r^{n-1}w^\va c^\va_r w^\va_{rt}dr\\
\leq &\|r^{\frac{n-1}{2}}w^\va_{rt}\|_{L^2} \|w^\va\|_{L^\infty}
\|r^{\frac{n-1}{2}}c^\va_r\|_{L^2}\\
\leq &\frac{1}{4}\|r^{\frac{n-1}{2}}w^\va_{rt}\|^2_{L^2}
+C_0
\|r^{\frac{n-1}{2}}c^\va_r\|^2_{L^2}(\|r^{\frac{n-1}{2}}w^\va\|^2_{L^2}
+\|r^{\frac{n-1}{2}}w^\va_r\|^2_{L^2}),
\end{split}
\ee
where in the last inequality we have used the Sobolev embedding inequality.
Differentiating the first equation of \eqref{e1} with respect to $t$ and testing the resulting equation with $r^{n-1}w^\va_t$, one gets from integration by parts that
\be\label{a20}
\begin{split}
\frac{1}{2}\frac{d}{dt}&\int_a^b r^{n-1} (w^\va_t)^2 dr
+\int_a^b r^{n-1} (w^\va_{rt})^2 dr\\
=&\int_a^b r^{n-1}(w^\va c^\va_r)_t w^\va_{rt}dr\\
\leq & \|r^{\frac{n-1}{2}}c^\va_{rt}\|_{L^2}\|w^\va\|_{L^\infty}
\|r^{\frac{n-1}{2}}w^\va_{rt}\|_{L^2}
+\|r^{\frac{n-1}{2}}c^\va_{r}\|_{L^2}\|w^\va_t\|_{L^\infty}
\|r^{\frac{n-1}{2}}w^\va_{rt}\|_{L^2}\\
\leq & C_0\|r^{\frac{n-1}{2}}c^\va_{rt}\|_{L^2}\|w^\va\|_{H^1}
\|r^{\frac{n-1}{2}}w^\va_{rt}\|_{L^2}\\
&+C_0\|r^{\frac{n-1}{2}}c^\va_{r}\|_{L^2}
(\|r^{\frac{n-1}{2}}w^\va_t\|_{L^2}
+\|r^{\frac{n-1}{2}}w^\va_t\|^{\frac{1}{2}}_{L^2}
\|r^{\frac{n-1}{2}}w^\va_{rt}\|^{\frac{1}{2}}_{L^2})
\|r^{\frac{n-1}{2}}w^\va_{rt}\|_{L^2}\\
\leq &\frac{1}{4}\|r^{\frac{n-1}{2}}w^\va_{rt}\|^2_{L^2}
+C_0\|r^{\frac{n-1}{2}}c^\va_{rt}\|_{L^2}^2(\|r^{\frac{n-1}{2}}w^\va\|_{L^2}^2
+\|r^{\frac{n-1}{2}}w^\va_r\|_{L^2}^2)\\
&+C_0(\|r^{\frac{n-1}{2}}c^\va_r\|_{L^2}^2
+\|r^{\frac{n-1}{2}}c^\va_r\|_{L^2}^4)
\|r^{\frac{n-1}{2}}w^\va_t\|_{L^2}^2,
\end{split}
\ee
 where in the second inequality we have used the Sobolev embedding inequality and Gagliardo-Nirenberg interpolation inequality.
  Adding \eqref{a20} to \eqref{a19} we have
 \be\label{c9}
\begin{split}
\frac{d}{dt}\big[&\int_a^b r^{n-1} (w^\va_r)^2 dr
+\int_a^b r^{n-1} (w^\va_t)^2 dr\big]
+\int_a^b r^{n-1} (w^\va_t)^2 dr
+\int_a^b r^{n-1} (w^\va_{rt})^2 dr\\
\leq &
C_0(\|r^{\frac{n-1}{2}}c^\va_{r}\|_{L^2}^2+\|r^{\frac{n-1}{2}}c^\va_{rt}\|_{L^2}^2)
(\|r^{\frac{n-1}{2}}w^\va\|_{L^2}^2
+\|r^{\frac{n-1}{2}}w^\va_r\|_{L^2}^2)\\
&+C_0(\|r^{\frac{n-1}{2}}c^\va_r\|_{L^2}^2
+\|r^{\frac{n-1}{2}}c^\va_r\|_{L^2}^4)
\|r^{\frac{n-1}{2}}w^\va_t\|_{L^2}^2.
\end{split}
\ee
Then applying Gronwall's inequality to \eqref{c9}, using \eqref{b18} and \eqref{a18} one derives
 \ben
 \begin{split}
\sup_{0<t\leq T}& \left[\int_a^b r^{n-1} (w^\va_t)^2(t) dr
+\int_a^b r^{n-1} (w^\va_r)^2(t) dr\right]
+\int_0^T\int_a^b r^{n-1} (w^\va_{rt})^2 drdt\\
\leq & C_{\lambda}(1+\exp\{\exp\{\exp\{C_{\lambda}\ka^2\}\}\}),
\end{split}
 \enn
 which, in conjunction with \eqref{a18} gives \eqref{c5}.

   We proceed to proving \eqref{c6}. When $\kappa=0$, it follows from \eqref{c4} that
\be\label{c8}
   \left[r^{n-1}c^\va_{rt}c^\va_t\right]\big|_{r=a}^{r=b}=0,
\quad \left[r^{n-1}c^\va_{rt}w^\va c^\va\right]\big|_{r=a}^{r=b}=0,\ \ \  \forall \ t\geq 0.
\ee
Similar to the derivation of \eqref{a14}, \eqref{b25} and the Sobolev embedding inequality leads to
\be\label{b54}
\begin{split}
-\int_a^b&r^{n-1}c^\va_{rt}(w^\va c^\va)_r dr\\
\leq& \frac{1}{4}\|r^{\frac{n-1}{2}}c^\va_{rt}\|^2_{L^2}
+C_0\|c_0\|_{L^\infty}^2
\|r^{\frac{n-1}{2}}w^\va_r\|^2_{L^2}\\
&+C_0\|r^{\frac{n-1}{2}}c^\va_r\|^2_{L^2}
(\|r^{\frac{n-1}{2}}w^\va\|^2_{L^2}
+\|r^{\frac{n-1}{2}}w^\va_r\|^2_{L^2}).
\end{split}
\ee
Substituting \eqref{b54} into \eqref{a11} and using \eqref{c8}, one gets
 \ben
\begin{split}
\frac{1}{2}&\va\frac{d}{dt}\int_a^b\frac{[(r^{n-1}c^\va_r)_r]^2}{r^{n-1}}dr
+\frac{3}{4}\int_a^b r^{n-1}(c^\va_{rt})^2 dr\\
&\leq C_0\|c_0\|_{L^\infty}^2
\|r^{\frac{n-1}{2}}w^\va_r\|^2_{L^2}
+C_0\|r^{\frac{n-1}{2}}c^\va_r\|^2_{L^2}
(\|r^{\frac{n-1}{2}}w^\va\|^2_{L^2}
+\|r^{\frac{n-1}{2}}w^\va_r\|^2_{L^2}).
\end{split}
\enn
Integrating the above inequality with respect to $t$ and using \eqref{b19} we have
\be\label{c7}
\va\sup_{0<t\leq T}\int_a^b\frac{[(r^{n-1}c^\va_r(t))_r]^2}{r^{n-1}}dr
+\int_0^T\int_a^b r^{n-1}(c^\va_{rt})^2 drdt
\leq C.
\ee
Applying Gronwall's inequality to \eqref{c9} and using \eqref{b19} and \eqref{c7} one derives
 \be\label{c10}
 \begin{split}
\sup_{0<t\leq T} \left[\int_a^b r^{n-1} (w^\va_t)^2(t) dr
+\int_a^b r^{n-1} (w^\va_r)^2(t) dr\right]
+\int_0^T\int_a^b r^{n-1} (w^\va_{rt})^2 drdt
\leq  C.
\end{split}
 \ee
 Collecting \eqref{c7} and \eqref{c10} we end up with \eqref{c6}.
  The proof is finished.

 \end{proof}

 \begin{lemma}\label{l4} Let the assumptions in Theorem \ref{t1} hold and let $(w^\va,c^\va)$ with $0<\va<1$ be the solutions of \eqref{e1}-\eqref{e2}, derived in Lemma \ref{l9}. \\
 (i) If $\kappa>0$.
 Then for any $0<T<\infty$, there exists a constant $C_{\lambda}$ depending on $a$, $b$, $n$, $\|w_0\|_{H^2}$, $\|c_0\|_{H^2}$, $T$ and $\lambda$ such that
 \be\label{c11}
\begin{split}
\sup_{0<t\leq T} \va^\frac{1}{2} \int_a^b (c^\va_{rr})^2(t)dr
+\va^{\frac{3}{2}}\int_0^T \int_a^b (c^\va_{rrr})^2drdt
\leq C_{\lambda}\left(1+\exp\{\exp\{\exp\{C_{\lambda}\ka^2\}\}\}\right).
\end{split}
 \ee
 (ii) If $\kappa=0$.
 Then for any $0<T<\infty$, there exists a constant $C$ depending on $a$, $b$, $n$, $\|w_0\|_{H^2}$, $\|c_0\|_{H^2}$ and $T$ such that
 \be\label{c12}
\begin{split}
\sup_{0<t\leq T} \int_a^b (c^\va_{rr})^2(t)dr
+\va\int_0^T \int_a^b (c^\va_{rrr})^2drdt
\leq C.
\end{split}
 \ee
 \end{lemma}
 \begin{proof}
  We first prove \eqref{c11}. Differentiating the second equation of \eqref{e1} with respect to $r$ to have
 \ben
 c^\va_{rt}=\va c^\va_{rrr}+\va \frac{n-1}{r}c^\va_{rr}-\va\frac{n-1}{r^2}c^\va_r-(w^\va c^\va)_r,
 \enn
 which, multiplied with $-\va c^\va_{rrr}$ in $L^2$ gives rise to
 \be\label{a22}
 \begin{split}
 \frac{1}{2}\va \frac{d}{dt}\int_a^b (c^\va_{rr})^2dr+\va^2 \int_a^b (c^\va_{rrr})^2dr
 =&\va [c^\va_{rt}c^\va_{rr}]|_{r=a}^{r=b}-\va^2\int_{a}^b
 \frac{n-1}{r}c^\va_{rr} c^\va_{rrr}dr\\
 &+\va^2\int_a^b\frac{n-1}{r^2}c^\va_r c^\va_{rrr}dr
 +\va\int_a^b (w^\va c^\va)_r c^\va_{rrr}dr\\
 :=&\sum_{k=4}^{k=7}I_{k}.
 \end{split}
 \ee
 It follows from the boundary conditions in \eqref{e2} and Gagliardo-Nirenberg interpolation inequality that
 \be\label{b55}
 \begin{split}
 I_4=&-\va\ka [(c^\va_t c^\va_{rr})(b,t)+(c^\va_t c^\va_{rr})(a,t)]\\
 \leq &2\va\ka \|c^\va_t\|_{L^\infty}
 \|c^\va_{rr}\|_{L^\infty}\\
 \leq& C_0\va\ka\|c^\va_t\|_{H^1}(\|c^\va_{rr}\|_{L^2}
 +\|c^\va_{rr}\|^{\frac{1}{2}}_{L^2}\|c^\va_{rrr}\|^{\frac{1}{2}}_{L^2})\\
 \leq &\frac{1}{8}\va^2 \|c^\va_{rrr}\|^2_{L^2}
 +C_0\va \|c^\va_{rr}\|_{L^2}^2+C_0
 (\va^{\frac{1}{2}}+\va)\ka^2\|c^\va_t\|_{H^1}^2.
 \end{split}
 \ee
 The Cauchy-Schwarz inequality leads to
 \be\label{b56}
 I_5+I_6
 \leq \frac{1}{8}\va^2 \|c^\va_{rrr}\|^2_{L^2}
 +C_0\va^2( \|c^\va_{rr}\|^2_{L^2}+\|c^\va_{r}\|^2_{L^2}).
 \ee
 To bound $I_7$, by integration by parts we first rewrite it as follows:
 \be\label{b59}
 \begin{split}
 I_7=&\va [(w^\va c^\va)_r c^\va_{rr}]|_{r=a}^{r=b}
 -\va\int_a^b (w^\va c^\va)_{rr} c^\va_{rr}dr\\
 =&\va [(w^\va c^\va)_r c^\va_{rr}]|_{r=a}^{r=b}
 -\va\int_a^b w^\va_{rr} c^\va c^\va_{rr}dr
 -2\va\int_a^b w^\va_r c^\va_{r} c^\va_{rr}dr
 -\va\int_a^b w^\va c^\va_{rr} c^\va_{rr}dr\\
 :=&\sum_{k=1}^{k=4}J_k.
 \end{split}
 \ee
 By the boundary conditions in \eqref{e2}, \eqref{a1} and Gagliardo-Nirenberg interpolation inequality, one gets
 \ben
 \begin{split}
 J_1&=\va[w^\va_r c^\va c^\va_{rr}]|_{a}^{b}
 +\va[w^\va c^\va_r c^\va_{rr}]|_{a}^{b}\\
 &=\kappa \va[\lambda-c^{\va}(b,t)][c^{\va}(b,t)+1]
 w^{\va}(b,t)c^{\va}_{rr}(b,t)
 +\kappa \va[\lambda-c^{\va}(a,t)][c^{\va}(a,t)+1]
 w^{\va}(a,t)c^{\va}_{rr}(a,t)\\
 &\leq 2\kappa \va (\lambda+\|c^\va\|_{L^\infty}) (\|c^\va\|_{L^\infty}+1)
 \|w^\va\|_{L^\infty}\|c^\va_{rr}\|_{L^\infty}
 \\
 &\leq C_0\kappa\va (\|c_0\|_{L^\infty}+2\lambda+1)^2 \|w^\va\|_{H^1}(\|c^\va_{rr}\|_{L^2}
 +\|c^\va_{rr}\|^{\frac{1}{2}}_{L^2}\|c^\va_{rrr}\|^{\frac{1}{2}}_{L^2})
 \\
 &\leq \frac{1}{8}\va^2\|c^\va_{rrr}\|^2_{L^2}
 +C_0\va
 \|c^\va_{rr}\|^2_{L^2}+C_0
 (\|c_0\|_{L^\infty}+2\lambda+1)^4\kappa^2
 \big(\va+\va^{\frac{1}{2}}\big)
 \|w^\va\|_{H^1}^2.
 \end{split}
 \enn
 It follows from the first equation of \eqref{e1}, Sobolev embedding inequality and \eqref{a1} that
 \be\label{b57}
 \begin{split}
 J_2
 \leq& C_0 \va\|c^\va\|_{L^\infty} \left(\|w^\va_t\|_{L^2}+\|w^\va_r\|_{L^2}
 +\|w^\va c^\va_r\|_{H^1}\right)\|c^\va_{rr}\|_{L^2}\\
 \leq &C_0\va(\|c_0\|_{L^\infty}+\lambda) \left[\|w^\va_t\|_{L^2}+\|w^\va\|_{H^1}
 +\|w^\va\|_{H^1} \|c^\va_{rr}\|_{L^2}
 +\|w^\va\|_{H^1} \|c^\va_r\|_{L^2}\right]
 \|c^\va_{rr}\|_{L^2}\\
 \leq &C_0(\|c_0\|_{L^\infty}+\lambda)  (\|w^\va\|_{H^1}+1)\cdot\va\|c^\va_{rr}\|^2_{L^2}
 +C_0\va (\|w^\va_t\|^2_{L^2}+\|w^\va\|_{H^1}^2+
 \|w^\va\|_{H^1}^2\|c^\va_r\|^2_{L^2}).
 \end{split}
 \ee
 The Sobolev embedding inequality entails that
 \be\label{b62}
 \begin{split}
 J_3+J_4
 \leq& C_0\va\|w^\va_r\|_{L^2}\|c^\va_r\|_{L^\infty}\|c^\va_{rr}\|_{L^2}
 +C_0\va \|w^\va\|_{L^\infty}\|c^\va_{rr}\|_{L^2}^2\\
 \leq &C_0\va \|w^\va_r\|_{L^2}(\|c^\va_r\|_{L^2}+\|c^\va_{rr}\|_{L^2})\|c^\va_{rr}\|_{L^2}
 +C_0\va \|w^\va\|_{H^1}\|c^\va_{rr}\|^2_{L^2}\\
 \leq &C_0(\|w^\va\|_{H^1}+\|w^\va\|_{H^1}^2)\cdot\va \|c^\va_{rr}\|^2_{L^2}
 +C_0\va\|c^\va_{r}\|^2_{L^2}.
 \end{split}
 \ee
 Collecting the above estimates for $J_1$-$J_4$ we arrive at
 \ben
 \begin{split}
 I_7\leq& \frac{1}{8}\va^2\|c^\va_{rrr}\|^2_{L^2}
 +C_0(\|c_0\|_{L^\infty}+\lambda+1) (\|w^\va\|_{H^1}+\|w^\va\|_{H^1}^2+1)
 \cdot\va\|c^\va_{rr}\|_{L^2}^2\\
 &+C_0\va (\|w^\va_t\|^2_{L^2}+\|w^\va\|_{H^1}^2+
 \|w^\va\|_{H^1}^2\|c^\va_r\|_{L^2}^2)
 +C_0\va\|c^\va_r\|_{L^2}^2\\
 &+C_0(\|c_0\|_{L^\infty}+2\lambda+1)^4
 \kappa^2\big(\va+\va^{\frac{1}{2}}\big)
 \|w^\va\|_{H^1}^2.
 \end{split}
 \enn
 Substituting the above estimates for $I_4$-$I_7$ into \eqref{a22} and using the fact $0<\va<1$ we have
 \be\label{a52}
 \begin{split}
 \va \frac{d}{dt}\int_a^b &(c^\va_{rr})^2dr+\va^2 \int_a^b (c^\va_{rrr})^2dr\\
 &\leq C_0(\|c_0\|_{L^\infty}+\lambda+1) (\|w^\va\|_{H^1}+\|w^\va\|_{H^1}^2+1)
 \cdot\va\|c^\va_{rr}\|^2_{L^2}\\
 &\ \ \ \ +C_0\va (\|w^\va_t\|^2_{L^2}+\|w^\va\|_{H^1}^2+
 \|w^\va\|_{H^1}^2\|c^\va_r\|_{L^2}^2)
 +C_0\va\|c^\va_r\|^2_{L^2}\\
 &\ \ \ \ +C_0(\|c_0\|_{L^\infty}+2\lambda+1)^4\ka^2\va^{\frac{1}{2}}(\|w^\va\|_{H^1}^2+\|c^\va_t\|_{H^1}^2),
 \end{split}
 \ee
 which, along with Gronwall's inequality, \eqref{b18} and \eqref{c5} gives \eqref{c11}.

 We proceed to proving \eqref{c12}.
  Since $\kappa=0$, by \eqref{c4} we know that the boundary term $J_1$ in \eqref{b59} equals to $0$, that is
  \be\label{b60}
  J_1=0.
  \ee
  By a similar procedure used in the derivation of \eqref{b57}, one deduces from \eqref{e1}, \eqref{b25} and the Sobolev embedding inequality that
  \be\label{b61}
 \begin{split}
 J_2
 \leq &C_0\|c_0\|_{L^\infty}  (\|w^\va\|_{H^1}+1)\cdot\va\|c^\va_{rr}\|^2_{L^2}
 +C_0\va (\|w^\va_t\|^2_{L^2}+\|w^\va\|_{H^1}^2+
 \|w^\va\|_{H^1}^2\|c^\va_r\|^2_{L^2}).
 \end{split}
 \ee
 Substituting \eqref{b62}, \eqref{b60}, \eqref{b61} into \eqref{b59} one gets in the case of $\kappa=0$ that
 \be\label{b58}
 \begin{split}
 I_7\leq&
 C_0(\|c_0\|_{L^\infty}+1) (\|w^\va\|_{H^1}+\|w^\va\|_{H^1}^2+1)
 \cdot\va\|c^\va_{rr}\|_{L^2}^2\\
 &+C_0\va (\|w^\va_t\|^2_{L^2}+\|w^\va\|_{H^1}^2+
 \|w^\va\|_{H^1}^2\|c^\va_r\|_{L^2}^2)
 +C_0\va\|c^\va_r\|_{L^2}^2.
 \end{split}
 \ee
 Substituting \eqref{b55}, \eqref{b56}, \eqref{b58} into \eqref{a22} and setting $\kappa=0$ in the resulting inequality one gets
 \ben
 \begin{split}
 \va \frac{d}{dt}\int_a^b (c^\va_{rr})^2dr+\va^2 \int_a^b (c^\va_{rrr})^2dr
 \leq& C_0 (\|c_0\|_{L^\infty}+1) (\|w^\va\|_{H^1}+\|w^\va\|_{H^1}^2+1)
 \cdot\va\|c^\va_{rr}\|^2_{L^2}\\
 &+C_0\va (\|w^\va_t\|^2_{L^2}+\|w^\va\|_{H^1}^2+
 \|w^\va\|_{H^1}^2\|c^\va_r\|_{L^2}^2)
 +C_0\va\|c^\va_r\|^2_{L^2}.
 \end{split}
 \enn
 Applying Gronwall's inequality to the above inequality and using \eqref{b19} and \eqref{c6} we get \eqref{c12}.
  The proof is completed.

\end{proof}

We next prove \eqref{a000} in the following lemma.
\begin{lemma}\label{l5}
Suppose $w_0\in H^2$, $c_0\in H^3$ and the assumptions in Theorem \ref{t1} hold true. Let $(w^0,c^0)$ be the classical solution of \eqref{e1}-\eqref{e2} with $\va=0$, derived in Lemma \ref{l9}. Then for any $0<T<\infty$, there exists a constant $C$ depending on $a$, $b$, $n$, $\|w_0\|_{H^2}$, $\|c_0\|_{H^3}$ and $T$ such that
\ben
\sup_{0<t\leq T}\left[
\|w^0(t)\|_{H^2}^2+\|c^0(t)\|_{H^3}^2
+\|c^0_t(t)\|_{H^2}^2
\right]
+\int_0^T \left[\|w^0(t)\|_{H^3}^2+\|c^0_t(t)\|_{H^3}^2\right]dt\leq C.
\enn
\end{lemma}
\begin{proof}
For clarification, we rewrite system \eqref{e1}-\eqref{e2} with $\va=0$ here:
\be\label{e3}
\left\{
\begin{array}{lll}
w^0_t=w^0_{rr}+\frac{n-1}{r}w^0_{r}-(w^0c^0_r)_r
-\frac{n-1}{r}(w^0c^0_r),\quad (r,t)\in (a,b)\times(0,\infty),\\
c^0_t=-w^0c^0,\\
w^0(r,0)=w_0(r),\quad c^0(r,0)=c_0(r),\\
(w^{0}_{r}-w^0 c^{0}_{r})(a,t)=(w^{0}_{r}-w^0 c^{0}_{r})(b,t)=0.
\end{array}
\right.
\ee
From the second equation of \eqref{e3} one easily deduces that
  \ben
  c^0(r,t)=c_0(r)e^{-\int_0^t w^{0}(r,\tau)d\tau},
  \enn
  which, implies
  \be\label{c13}
  \sup_{t>0}\|c^0(t)\|_{L^\infty}\leq \|c_0\|_{L^\infty},
   \ee
   thanks to the nonnegativity of $c_0$ and $w^0$.
  Multiplying the first equation of \eqref{e3} by $r^{n-1}$ then integrating the resulting equation over $(a,b)$ and using integration by parts, one gets
\ben
\frac{d}{dt}\int_a^b r^{n-1}w^0 dr=[r^{n-1}(w^0_r-w^0 c^0_r)]|_{r=a}^{r=b}=0.
\enn
Integrating this with respect to $t$, we have
\be\label{c14}
\int_a^b r^{n-1}w^0(t) dr=\int_a^b r^{n-1}w_0 dr,\quad \forall\ t\geq 0.
\ee
It is obvious that \eqref{a4} and \eqref{a6} still holds with $\va=0$. Thus by adding \eqref{a4} to \eqref{a6} and setting $\va=0$ in the resulting equality we arrive at
\ben
\begin{split}
\frac{d}{dt}&\int_a^b r^{n-1}(w^0\log w^0-w^0+1)dr+
2\frac{d}{dt}\int_a^b r^{n-1} (\partial_r\sqrt{c^0})^2dr\\
&+4\int_a^b r^{n-1}(\partial_r \sqrt{w^0})^2dr
+2\int_a^b r^{n-1}w^0
(\partial_r\sqrt{c^0})^2 dr=0.
\end{split}
\enn
 Integrating this equality with respect to $t$ leads to
 \be\label{c15}
\begin{split}
\sup_{0<t\leq T}&\left[\int_a^b r^{n-1}(w^0\log w^0-w^0+1)(t)dr+
\int_a^b r^{n-1} (\partial_r\sqrt{c^0})^2(t)dr\right]\\
+&\int_0^T\int_a^b r^{n-1}(\partial_r \sqrt{w^0})^2drdt
\leq
C,
\end{split}
\ee
which, along with \eqref{c13} gives rise to
\be\label{c16}
\begin{split}
\sup_{0<t\leq T}\int_a^b r^{n-1}(c^0_r)^2(t)dr
\leq 4\left[\sup_{0<t\leq T}\int_a^b r^{n-1}(\pt_r\sq{c^0})^2(t)dr
\right]\cdot
\left[\sup_{t>0}\|c^0(t)\|_{L^\infty}
\right]
\leq C.
\end{split}
\ee
Noting that estimate \eqref{b20} still holds true for $\va=0$, one gets
\ben
\begin{split}
\frac{d}{dt}&\int_{a}^b r^{n-1}(w^0)^2dr
+\int_{a}^b r^{n-1}(w^0_r)^2dr
\leq
C_0(\|r^{\frac{n-1}{2}}c_r^0\|^2_{L^2}
+\|r^{\frac{n-1}{2}}c_r^0\|^4_{L^2})
\|r^{\frac{n-1}{2}}w^0\|^2_{L^2},
\end{split}
\enn
which, along with Gronwall's inequality and \eqref{c16} entails that
\be\label{c17}
\sup_{0<t\leq T}\int_{a}^b r^{n-1}(w^0)^2(t)dr
+\int_0^T\int_{a}^b r^{n-1}(w^0_r)^2drdt\leq C.
\ee
It follows from the second equation of \eqref{e3}, \eqref{c13} and \eqref{c17} that
\be\label{c18}
\begin{split}
\int_0^T\int_a^b r^{n-1} (c^0_t)^2drdt
\leq
\int_0^T\int_a^b r^{n-1} (w^0)^2drdt\cdot\sup_{t>0}\|c^0(t)\|_{L^\infty}^2
\leq C.
\end{split}
\ee
Collecting estimates \eqref{c13}, \eqref{c16}-\eqref{c18} we arrive at
\be\label{a35}
\|c^0\|_{L^\infty(0,T;L^\infty)}+\|c^0_r\|_{L^\infty(0,T;L^2)}
+\|w^0\|_{L^\infty(0,T;L^2)}+\|w^0_r\|_{L^2(0,T;L^2)}
+\|c^0_t\|_{L^2(0,T;L^2)}\leq C.
\ee

It follows from the second equation of \eqref{e3} and Sobolev embedding inequality that
\ben
\begin{split}
\int_a^b r^{n-1}(c^0_{rt})^2dr\leq &C_0\|c^0\|_{L^\infty}^2\|w_r^0\|_{L^2}^2
+C_0\|w^0\|_{L^\infty}^2\|c_r^0\|_{L^2}^2\\
\leq & C_0\|c^0\|_{L^\infty}^2\|w_r^0\|_{L^2}^2
+C_0\|w^0\|_{H^1}^2\|c_r^0\|_{L^2}^2,
\end{split}
\enn
which, along with \eqref{a35} leads to
\be\label{b23}
\begin{split}
\int_0^T\int_a^b r^{n-1}(c^0_{rt})^2drdt\leq &C_0\|c^0\|_{L^\infty(0,T;L^\infty)}^2
\|w_r^0\|_{L^2(0,T;L^2)}^2\\
&+C_0\|w^0\|_{L^2(0,T;H^1)}^2
\|c_r^0\|_{L^\infty(0,T;L^2)}^2\\
\leq& C.
\end{split}
\ee
It is easy to check that \eqref{a19} and \eqref{a20} still hold true for $\va=0$. Thus adding \eqref{a20} to \eqref{a19} and letting $\va=0$ one gets \ben
\begin{split}
\frac{d}{dt}\big[&\int_a^b r^{n-1} (w^0_r)^2 dr
+\int_a^b r^{n-1} (w^0_t)^2 dr\big]
+\int_a^b r^{n-1} (w^0_t)^2 dr
+\int_a^b r^{n-1} (w^0_{rt})^2 dr\\
\leq &
C_0(\|r^{\frac{n-1}{2}}c^0_{r}\|_{L^2}^2+\|r^{\frac{n-1}{2}}c^0_{rt}\|_{L^2}^2)
(\|r^{\frac{n-1}{2}}w^0\|_{L^2}^2
+\|r^{\frac{n-1}{2}}w^0_r\|_{L^2}^2)\\
&+C_0(\|r^{\frac{n-1}{2}}c^0_r\|_{L^2}^2
+\|r^{\frac{n-1}{2}}c^0_r\|_{L^2}^4)
\|r^{\frac{n-1}{2}}w^0_t\|_{L^2}^2,
\end{split}
\enn
 which, along with Gronwall's inequality, \eqref{b23} and \eqref{a35} entails that
 \be\label{b24}
 \begin{split}
\sup_{0<t\leq T} \left[\int_a^b r^{n-1} (w^0_r)^2(t) dr
+\int_a^b r^{n-1} (w^0_t)^2(t) dr\right]
+\int_0^T\int_a^b r^{n-1} (w^0_{rt})^2 drdt
\leq  C.
\end{split}
 \ee
It follows from the first equation of \eqref{e3} and Sobolev embedding inequality that
\be\label{a36}
\begin{split}
\|w_{rr}^0\|_{L^2}^2\leq& C_0\left[
\|w^0_t\|_{L^2}^2+\|w^0_r\|_{L^2}^2+\|w^0c^0_r\|^2_{H^1}
\right]\\
\leq & C_0\left[
\|w^0_t\|_{L^2}^2+\|w^0_r\|_{L^2}^2+\|w^0\|_{H^1}^2\left(\|c^0_r\|^2_{L^2}
+\|c^0_{rr}\|_{L^2}^2\right)
\right].\\
\end{split}
\ee
Differentiating the second equation of \eqref{e3} with respect to $r$ twice, then multiplying the resulting equation with $2c^0_{rr}$ in $L^2$ and using \eqref{c13} and the Sobolev embedding inequality, we have
\ben
\begin{split}
\frac{d}{dt}\|c^0_{rr}\|^2_{L^2}
=&-2\int_a^b w^0_{rr}c^0c^0_{rr}dr
-4\int_a^b w^0_{r}c^0_{r}c^0_{rr}dr
-2\int_a^b w^0(c^0_{rr})^2dr\\
\leq &C_0\left(\|w^0_{rr}\|_{L^2}\|c^0\|_{L^\infty}\|c^0_{rr}\|_{L^2}
+\|w^0_{r}\|_{L^2}\|c^0_r\|_{L^\infty}\|c^0_{rr}\|_{L^2}
\right)\\
\leq &C_0\|c_0\|_{L^\infty}\|w^0_{rr}\|_{L^2}\|c^0_{rr}\|_{L^2}
+C_0\|w^0_{r}\|_{L^2}(\|c^0_r\|_{L^2}+\|c^0_{rr}\|_{L^2})\|c^0_{rr}\|_{L^2}
\\
\leq & \|w^0_{rr}\|_{L^2}^2
+C_0(\|w^0\|_{H^1}+\|c_0\|_{L^\infty}^2+1)
\|c^0_{rr}\|^2_{L^2}
+C_0\|w^0\|_{H^1}^2\|c^0\|_{H^1}^2,
\end{split}
\enn
which, in conjunction with \eqref{a36} gives rise to
\ben
\begin{split}
\frac{d}{dt}\|c^0_{rr}\|^2_{L^2}
\leq&  C_0(\|w^0\|_{H^1}^2+\|w^0\|_{H^1}
+\|c_0\|_{L^\infty}^2+1)\|c^0_{rr}\|_{L^2}^2\\
&+C_0
\left(\|w^0_t\|_{L^2}^2+\|w^0_r\|_{L^2}^2+\|w^0\|_{H^1}^2\|c^0\|_{H^1}^2
\right).
\end{split}
\enn
Applying Gronwall's inequality to this inequality, using \eqref{a35} and \eqref{b24}, we conclude that
\be\label{a38}
\|c^0_{rr}\|_{L^\infty(0,T;L^2)}^2\leq C.
\ee
\eqref{a36}, \eqref{a35}, \eqref{b24} and \eqref{a38} further lead to
\be\label{a39}
\|w^0_{rr}\|_{L^\infty(0,T;L^2)}^2\leq C.
\ee

By a similar argument used in deriving \eqref{a36}, one deduces from the first equation of \eqref{e3} that
\be\label{a40}
\begin{split}
\|w_{rrr}^0\|_{L^2}^2\leq& C_0\left[
\|w^0_t\|_{H^1}^2+\|w^0\|_{H^2}^2+\|w^0c^0_r\|^2_{H^2}
\right]\\
\leq & C_0\left[
\|w^0_t\|_{H^1}^2+\|w^0\|_{H^2}^2
+\|w^0\|_{H^2}^2\left(\|c^0\|^2_{H^2}
+\|c^0_{rrr}\|^2_{L^2}\right)
\right].
\end{split}
\ee
Applying $\pt^3_r$ to the second equation of \eqref{e3}, and multiplying the resulting equation with $2c^0_{rrr}$ in $L^2$ to get
\be\label{b63}
\begin{split}
\frac{d}{dt}\|c^0_{rrr}\|_{L^2}^2=&-2\int_a^b w^0_{rrr} c^0 c^0_{rrr}dr
-6\int_a^b (w^0_{rr}c^0_r+w^0_rc^0_{rr})c^0_{rrr}dr
-2\int_a^b w^0 (c^0_{rrr})^2 dr\\
\leq &2\|w^0_{rrr}\|_{L^2} \|c^0\|_{L^\infty}\|c^0_{rrr}\|_{L^2}
+C_0\|w^0\|_{H^2}\|c^0\|_{H^2}\|c^0_{rrr}\|_{L^2}\\
\leq & \|w^0_{rrr}\|_{L^2}^2+C_0(1+\|c_0\|_{L^\infty}^2)\|c^0_{rrr}\|_{L^2}^2
+C_0\|w^0\|_{H^2}^2\|c^0\|_{H^2}^2,
\end{split}
\ee
where the Sobolev embedding inequality and \eqref{c13} have been used. Substituting \eqref{a40} into \eqref{b63} yields
\be\label{a41}
\begin{split}
\frac{d}{dt}\|c^0_{rrr}\|_{L^2}^2
\leq& C_0\left(\|w^0\|_{H^2}^2+\|c_0\|_{L^\infty}^2+1\right)\|c^0_{rrr}\|_{L^2}^2\\
&+C_0\left(\|w^0_t\|_{H^1}^2+\|w^0\|_{H^2}^2
+\|w^0\|_{H^2}^2\|c^0\|^2_{H^2}\right).
\end{split}
\ee
Applying Gronwall's inequality to \eqref{a41},  using \eqref{a35}, \eqref{b24}, \eqref{a38} and \eqref{a39} one derives
\be\label{a42}
\|c^0_{rrr}\|_{L^\infty(0,T;L^2)}^2\leq C.
\ee
 Using \eqref{a42}, \eqref{a35}, \eqref{b24}, \eqref{a38} and \eqref{a39} one deduces from \eqref{a40} that
\be\label{a43}
\|w^0_{rrr}\|_{L^2(0,T;L^2)}^2\leq C.
\ee
Moreover, it follows from the second equation of \eqref{e3}, \eqref{a35}, \eqref{b24}, \eqref{a38}, \eqref{a39},\eqref{a42} and \eqref{a43} that
\be\label{a44}
\begin{split}
\|c^0_t\|_{L^\infty(0,T;H^2)}
+\|c^0_t\|_{L^2(0,T;H^3)}\leq &
C_0\|w^0\|_{L^\infty(0,T;H^2)}
\|c^0\|_{L^\infty(0,T;H^2)}\\
&+C_0\|w^0\|_{L^2(0,T;H^3)}
\|c^0\|_{L^\infty(0,T;H^3)}
\leq C.
\end{split}
\ee
Collecting \eqref{a35}, \eqref{b24}, \eqref{a38}, \eqref{a39} and \eqref{a42} - \eqref{a44}, we derive the desired estimate.
The proof is finished.

\end{proof}
\textbf{\emph{Proof of Theorem \ref{t1}}.}
From \eqref{a1}, \eqref{b18}, \eqref{c5} and \eqref{c11}, one gets \eqref{a0}. \eqref{c19} follows from \eqref{b25}, \eqref{b19}, \eqref{c6} and \eqref{c12}.
 \eqref{a000} follows from Lemma \ref{l5}. The proof is finished.

\endProof

\section{Proof of Theorem \ref{t2}}
 Recall that $(w^\va,c^\va)$ and $(w^0,c^0)$ denote solutions of system \eqref{e1}-\eqref{e2} with $\va>0$ and $\va=0$, respectively. For convenience, we set
  \begin{equation*}\tilde{w}=w^\va-w^0, \qquad\tilde{c}=c^\va-c^0.
  \end{equation*}
   Then one deduces from \eqref{e1}-\eqref{e2} that $(\tilde{w}, \tilde{c})$ satisfies the following initial-boundary problem:
\be\label{e4}
\left\{
\begin{array}{lll}
\ti{w}_t=\ti{w}_{rr}+\frac{n-1}{r}\ti{w}_{r}
-(w^\va\ti{c}_r+\ti{w}c^0_r)_r-\frac{n-1}{r}(w^\va\ti{c}_r+\ti{w}c^0_r),\quad (r,t)\in (a,b)\times(0,\infty),\\
\ti{c}_t=\va \ti{c}_{rr}+\va c^0_{rr}+\va\frac{n-1}{r}c^\va_r-(\ti{w}c^\va+w^0\ti{c}),\\
\ti{w}(r,0)=0,\quad \ti{c}(r,0)=0\\
(\ti{w}_r-w^\va\ti{c}_r-\ti{w}c^0_r)(a,t)
=(\ti{w}_r-w^\va\ti{c}_r-\ti{w}c^0_r)(b,t)=0,
\\
\ti{c}_r(a,t)=-\ka[\lambda-c^\va(a,t)]-c^0_r(a,t),\quad \ti{c}_r(b,t)=\ka[\lambda-c^\va(b,t)]-c^0_r(b,t).
\end{array}
\right.
\ee
Based on the reformulated system \eqref{e4}, we shall derive a series of results below.
\begin{lemma}\label{l0} Suppose $\kappa>0$ and the assumptions in Theorem \ref{t2} hold true. Let $(w^\va,c^\va)$ with $0\leq \va<1$ be the solutions of \eqref{e1}-\eqref{e2}, derived in Lemma \ref{l9}. \\
(i) If for each $t\in[0,T]$,
\be\label{a80}
\begin{split}
 c_0(a)e^{-\int_0^t w^0(a,\tau)d\tau}
\left\{
1-e^{-\int_0^t w^0(a,\tau)c^0(a,\tau)d\tau}
\right\}
-\lambda \left\{1-e^{-\int_0^t w^0(a,\tau)[c^0(a,\tau)+1]d\tau}\right\}=0.
\end{split}
\ee
Then
\be\label{a82}
\ti{c}_r(a,t)=\ka\ti{c}(a,t),\qquad \forall\ t\in [0,T].
\ee
(ii)
If for each $t\in[0,T]$,
\be\label{a81}
\begin{split}
 c_0(b)e^{-\int_0^t w^0(b,\tau)d\tau}
\left\{
1-e^{-\int_0^t w^0(b,\tau)c^0(b,\tau)d\tau}
\right\}
-\lambda \left\{1-e^{-\int_0^t w^0(b,\tau)[c^0(b,\tau)+1]d\tau}\right\}=0.
\end{split}
\ee
Then
\be\label{a83}
\ti{c}_r(b,t)=-\ka\ti{c}(b,t),\qquad \forall\ t\in [0,T].
\ee
\end{lemma}
\begin{proof}
It follows from the second equation of \eqref{e3} that
\ben
\frac{d}{dt}\left[
c^0(a,t)e^{\int_0^t w^0(a,\tau)d\tau}
\right]=0.
\enn
Thus
\be\label{a78}
c^0(a,t)=c_0(a)e^{-\int_0^t w^0(a,\tau)d\tau}.
\ee
Differentiating the second equation of \eqref{e3} with respect to $r$, one gets
\ben
c^0_{rt}=-w^0_r c^0-w^0 c^0_r,
\enn
which, along with the boundary conditions in \eqref{e3} gives rise to
\ben
c^0_{rt}(a,t)=-w^0(a,t)c_r^0(a,t)c^0(a,t)-w^0(a,t)c_r^0(a,t).
\enn
Thus
\ben
\frac{d}{dt}\left\{
c^0_r(a,t)e^{\int_0^t[w^0(a,\tau)c^0(a,\tau)+w^0(a,\tau)] d\tau}
\right\}=0,
\enn
which, integrated over $(0,t)$ leads to
\be\label{a79}
c^0_{r}(a,t)=-\ka[\lambda-c_0(a)]e^{-\int_0^t[w^0(a,\tau)c^0(a,\tau)+w^0(a,\tau)] d\tau},
\ee
where we have used the fact that $c^{0}_r(a,0)=c_{0r}(a)=-\ka[\lambda-c_0(a)].$
Substituting \eqref{a78} and \eqref{a79} into the third boundary condition in \eqref{e4} one deduces that
\be\label{b26}
\begin{split}
\ti{c}_r(a,t)=&\ka[c^\va(a,t)-\lambda]-c^0_r(a,t)\\
=&\ka[c^\va(a,t)-c^0(a,t)]+\ka[c^0(a,t)-\lambda]-c^0_r(a,t)\\
=&\ka[c^\va(a,t)-c^0(a,t)]+\ka[c_0(a)e^{-\int_0^t w^0(a,\tau)d\tau}-\lambda]\\
&+\ka[\lambda-c_0(a)]e^{-\int_0^t[w^0(a,\tau)c^0(a,\tau)+w^0(a,\tau)] d\tau},
\end{split}
\ee
which, in conjunction with \eqref{a80} leads to \eqref{a82}.  By a similar argument used in deriving \eqref{a78}, \eqref{a79} and \eqref{b26} one deduces that
\ben
c^0(b,t)=c_0(b)e^{-\int_0^t w^0(b,\tau)d\tau},
\enn
\be\label{b27}
c^0_{r}(b,t)=\ka[\lambda-c_0(b)]e^{-\int_0^t[w^0(b,\tau)c^0(b,\tau)+w^0(b,\tau)] d\tau}
\ee
and
\ben
\begin{split}
\ti{c}_r(b,t)=&\ka[\lambda-c^\va(b,t)]-c^0_r(b,t)\\
=&\ka[c^0(b,t)-c^\va(b,t)]+\ka[\lambda-c^0(b,t)]-c^0_r(b,t)\\
=&\ka[c^0(b,t)-c^\va(b,t)]+\ka[\lambda-c_0(b)e^{-\int_0^t w^0(b,\tau)d\tau}]\\
&-\ka[\lambda-c_0(b)]e^{-\int_0^t[w^0(b,\tau)c^0(b,\tau)+w^0(b,\tau)] d\tau},
\end{split}
\enn
which, along with \eqref{a81} gives \eqref{a83}. The proof is completed.

\end{proof}

\begin{lemma}\label{l11}
 Suppose the assumptions in Theorem \ref{t2} hold true. Let $(w^0,c^0)$ be the solutions of \eqref{e1}-\eqref{e2} with $\va=0$, derived in Lemma \ref{l9}. \\
(i) If $w^0(a,t)>0$ for some $t\in [0,T]$. Then there exists $t_1\in [0,T]$ such that
\be\label{b32}
\begin{split}
 c_0(a)e^{-\int_0^{t_1} w^0(a,\tau)d\tau}
\left\{
1-e^{-\int_0^{t_1} w^0(a,\tau)c^0(a,\tau)d\tau}
\right\}
-\lambda \left\{1-e^{-\int_0^{t_1} w^0(a,\tau)[c^0(a,\tau)+1]d\tau}\right\}\neq0
\end{split}
\ee
and
\be\label{b33}
w^0(a,t_1)>0.
\ee
(ii) If $w^0(b,t)>0$ for some $t\in [0,T]$. Then there exists $t_2\in [0,T]$ such that
\be\label{b34}
\begin{split}
 c_0(b)e^{-\int_0^{t_2} w^0(b,\tau)d\tau}
\left\{
1-e^{-\int_0^{t_2} w^0(b,\tau)c^0(b,\tau)d\tau}
\right\}
-\lambda\left\{1-e^{-\int_0^{t_2} w^0(b,\tau)[c^0(b,\tau)+1]d\tau}\right\}\neq0
\end{split}
\ee
and
\be\label{b35}
w^0(b,t_2)>0.
\ee
\end{lemma}

\begin{proof} We first prove Part (i) by argument of contradiction. Let $w^0(a,t_0)>0$ for some $t_0\in [0,T]$. By the continuity of $w^0(a,t)$ in variable $t$ we know that there exists $0\leq t_3<t_4\leq T$ such that
\be\label{b36}
w^0(a,t)>0,\quad \forall\  t\in [t_3,t_4].
\ee
From \eqref{a78} and the assumption $c_0>0$ on $[a,b]$ in Theorem \ref{t2}, one easily deduces that
\be\label{b39}
c^0(a,t)>0, \quad \forall \ t\in [0,T].
\ee
Assume
\be\label{b37}
\begin{split}
 c_0(a)e^{-\int_0^{t} w^0(a,\tau)d\tau}
\left\{
1-e^{-\int_0^{t} w^0(a,\tau)c^0(a,\tau)d\tau}
\right\}
-\lambda \left\{1-e^{-\int_0^{t} w^0(a,\tau)[c^0(a,\tau)+1]d\tau}\right\}=0
\end{split}
\ee
holds true for each $t\in [t_3,t_4].$
Then it follows from \eqref{b36}-\eqref{b37} that
\ben
\begin{split}
f(t):=&\frac{e^{-\int_0^{t} w^0(a,\tau)d\tau}
-e^{-\int_0^{t} w^0(a,\tau)[c^0(a,\tau)+1]d\tau}}{1-e^{-\int_0^{t} w^0(a,\tau)[c^0(a,\tau)+1]d\tau}}\\
=&\frac{e^{\int_0^{t} w^0(a,\tau)c^0(a,\tau)d\tau}-1}
{e^{\int_0^{t} w^0(a,\tau)[c^0(a,\tau)+1]d\tau}-1}\\
=&\frac{\lambda}{c_0(a)},\quad\quad\quad \forall \
t\in[t_3,t_4].
\end{split}
\enn
Thus,
\ben
\begin{split}
0=\frac{df(t)}{dt}=&
\frac{e^{\int_0^{t} w^0(a,\tau)c^0(a,\tau)d\tau} w^0(a,t)c^0(a,t)
\{e^{\int_0^{t} w^0(a,\tau)[c^0(a,\tau)+1]d\tau}-1\}
}
{\{e^{\int_0^{t} w^0(a,\tau)[c^0(a,\tau)+1]d\tau}-1\}^2}\\
&-
\frac{
\{e^{\int_0^{t} w^0(a,\tau)c^0(a,\tau)d\tau}-1\}
e^{\int_0^{t} w^0(a,\tau)[c^0(a,\tau)+1]d\tau}
w^0(a,t)[c^0(a,t)+1]
}
{\{e^{\int_0^{t} w^0(a,\tau)[c^0(a,\tau)+1]d\tau}-1\}^2}
\end{split}
\enn
for each $t\in [t_3,t_4]$, which, along with a direct computation gives rise to
\be\label{b38}
\begin{split}
g(t):=c^0(a,t)+1-c^0(a,t)e^{-\int_0^t w^0(a,\tau)d\tau}
-e^{\int_0^t w^0(a,\tau)c^0(a,\tau)d\tau}
=0,\quad \forall\  t\in [t_3,t_4].
\end{split}
\ee
Then it follows from \eqref{b38} and the second equation in \eqref{e3} that
\be\label{b40}
\begin{split}
0=\frac{dg(t)}{dt}=&
c^0_t(a,t)-c^0_t(a,t)e^{-\int_0^t w^0(a,\tau)d\tau}
+c^0(a,t)w^0(a,t)e^{-\int_0^t w^0(a,\tau)d\tau}\\
&-w^0(a,t)c^0(a,t)e^{\int_0^t w^0(a,\tau)c^0(a,\tau)d\tau}
\\
=&w^0(a,t)c^0(a,t)[-1+2e^{-\int_0^t w^0(a,\tau)d\tau}
-e^{\int_0^t w^0(a,\tau)c^0(a,\tau)d\tau}]
\end{split}
\ee
for each $t\in[t_3,t_4]$.
However, \eqref{b36} and \eqref{b39} indicate that
\ben
w^0(a,t)c^0(a,t)[-1+2e^{-\int_0^t w^0(a,\tau)d\tau}
-e^{\int_0^t w^0(a,\tau)c^0(a,\tau)d\tau}]<0, \quad \forall\ t\in [t_3,t_4],
\enn
which, contradicts with \eqref{b40}. Hence by argument of contradiction there exists a $t_1\in [t_3,t_4]$ such that
\ben
\begin{split}
 c_0(a)e^{-\int_0^{t_1} w^0(a,\tau)d\tau}
\left\{
1-e^{-\int_0^{t_1} w^0(a,\tau)c^0(a,\tau)d\tau}
\right\}
-\lambda \left\{1-e^{-\int_0^{t_1} w^0(a,\tau)[c^0(a,\tau)+1]d\tau}\right\}\neq0,
\end{split}
\enn
which, along with \eqref{b36} gives \eqref{b32} and \eqref{b33}. We thus prove Part (i). Part (ii) follows immediately from a similar argument used in proving Part (i). The proof is completed.
\end{proof}

\begin{lemma}\label{l6} Suppose $\kappa>0$ and the assumptions in Theorem \ref{t2} hold true. Let $(w^\va,c^\va)$ with $0\leq \va<1$ be the solutions of \eqref{e1}-\eqref{e2}, derived in Lemma \ref{l9}.  Then there exists a constant $C_{\lambda}$ depending on $a$, $b$, $n$, $\|w_0\|_{H^2}$, $\|c_0\|_{H^3}$, $T$ and $\lambda$ such that
\be\label{a60}
\begin{split}
\sup_{0<t\leq T}&\left[
\int_a^b \ti{w}^2(t)dr+
\int_a^b \ti{c}^2(t)dr+
\int_a^b \ti{c}_r^2(t)dr
\right]\\
+&\int_0^T\int_a^b  \ti{w}_r^2drdt\leq C_{\lambda}(1+\exp\{\exp\{\exp\{\exp\{C_{\lambda}\ka^2\}\}\}\})^2
\cdot\va^{\frac{1}{2}}.
\end{split}
\ee
Assume further that \eqref{a80} and \eqref{a81} hold.  Then
\be\label{a61}
\begin{split}
\sup_{0<t\leq T}&\left[
\int_a^b \ti{w}^2(t)dr+
\int_a^b \ti{c}^2(t)dr+
\int_a^b \ti{c}_r^2(t)dr
\right]\\
&+\int_0^T\int_a^b  \ti{w}_r^2drdt
\leq
 C_{\lambda}(1+\exp\{\exp\{\exp\{\exp\{C_{\lambda}\ka^2\}\}\}\})^2\cdot\va,
\end{split}
\ee
where the constant $C_{\lambda}$ depends on $a$, $b$, $n$, $\|w_0\|_{H^2}$, $\|c_0\|_{H^3}$, $T$ and $\lambda$.
\end{lemma}

\begin{proof}
Multiplying the first equation of \eqref{e4} with $\ti{w}$ in $L^2$ and using integration by parts, we have
\be\label{a24}
\begin{split}
\frac{1}{2}\frac{d}{dt}\int_a^b \ti{w}^2dr
+\int_a^b \ti{w}_r^2dr
=&\int_a^b (w^\va\ti{c}_r+\ti{w}c^0_r) \ti{w}_rdr
+\int_a^b \frac{n-1}{r}(\ti{w}_r-w^\va\ti{c}_r-\ti{w}c^0_r)\ti{w}dr\\
\leq &C_0 (\|w^\va\|_{L^\infty}\|\ti{c}_r\|_{L^2}+\|\ti{w}\|_{L^2}\|c^0_r\|_{L^\infty})
\|\ti{w}_r\|_{L^2}\\
&+C_0 (\|\ti{w}_r\|_{L^2}+\|w^\va\|_{L^\infty}\|\ti{c}_r\|_{L^2}
+\|\ti{w}\|_{L^2}\|c^0_r\|_{L^\infty})
\|\ti{w}\|_{L^2}\\
\leq& \frac{1}{4}\|\ti{w}_r\|^2_{L^2}
+C_0\|w^\va\|_{H^1}^2\|\ti{c}_r\|_{L^2}^2
+C_0(\|c^0\|_{H^2}^2+1)\|\ti{w}\|_{L^2}^2,
\end{split}
\ee
where in the last inequality we have used the Sobolev embedding inequality and Cauchy-Schwarz inequality.
Taking the $L^2$ inner product of the second equation of \eqref{e4} with $\ti{c}$ and using the Sobolev embedding inequality and Cauchy-Schwarz inequality, one has
\be\label{a25}
\begin{split}
\frac{1}{2}\frac{d}{dt}\int_a^b \ti{c}^2dr
=&\va\int_a^b  c^\va_{rr}\ti{c}\,dr
+\va\int_a^b \frac{n-1}{r}c^\va_r \ti{c}\,dr
-\int_a^b (\ti{w}c^\va+w^0\ti{c})\ti{c}\,dr\\
\leq &
C_0\va(\|c^\va_{rr}\|_{L^2}+\|c^\va_{r}\|_{L^2})\|\ti{c}\|_{L^2}
+C_0(\|c^\va\|_{L^\infty}\|\ti{w}\|_{L^2}
+\|w^0\|_{L^\infty}\|\ti{c}\|_{L^2})\|\ti{c}\|_{L^2}\\
\leq & C_0(\|w^0\|_{H^1}+1)\|\ti{c}\|^2_{L^2}
+C_0\|c^\va\|_{H^1}^2\|\ti{w}\|^2_{L^2}
+C_0\va^2(\|c^\va_{rr}\|^2_{L^2}+\|c^\va_{r}\|^2_{L^2}).
\end{split}
\ee
Differentiating the second equation of \eqref{e4} with respect to $r$ to have
\be\label{a30}
\ti{c}_{rt}=\va \ti{c}_{rrr}+\va c^0_{rrr}+\va\frac{n-1}{r}c^\va_{rr}
-\va\frac{n-1}{r^2}c^\va_r-(\ti{w}c^\va+w^0\ti{c})_r.
\ee
Multiplying \eqref{a30} with $\ti{c}_r$ in $L^2$, then employibg integration by parts to get
\be\label{a26}
\begin{split}
\frac{1}{2}\frac{d}{dt}&\int_a^b \ti{c}_r^2dr
+\va \int_a^b \ti{c}_{rr}^2 dr\\
=&\va [\ti{c}_{rr}\ti{c}_r]\big|_{r=a}^{r=b}
+\int_a^b  \left[\va c^0_{rrr}
+\va\frac{n-1}{r}c^\va_{rr}
-\va\frac{n-1}{r^2}c^\va_r
-(\ti{w}c^\va+w^0\ti{c})_r\right]\ti{c}_rdr\\
:=&K_1+K_2.
\end{split}
\ee
It follows from the third and fourth boundary conditions in \eqref{e4} and Gagliardo-Nirenberg interpolation inequality that
\ben
\begin{split}
K_1\leq &2\va  (\|c^\va_{rr}\|_{L^\infty}+\|c^0_{rr}\|_{L^\infty})
(\ka\lambda+\ka\|c^\va\|_{L^\infty}+\|c^0_r\|_{L^\infty})\\
\leq& C_0\va (\|c^\va_{rr}\|_{L^2}
+\|c^\va_{rr}\|_{L^2}^{\frac{1}{2}}
\|c^\va_{rrr}\|_{L^2}^{\frac{1}{2}}
+\|c^0\|_{H^3})(\ka\lambda
+\ka\|c^\va\|_{H^1}+\|c^0\|_{H^2})\\
\leq &C_0\va^{\frac{1}{2}}(\va^{\frac{1}{4}}\|c^\va_{rr}\|_{L^2}
+\va^{\frac{1}{8}}\|c^\va_{rr}\|_{L^2}^{\frac{1}{2}}
\cdot\va^{\frac{3}{8}}\|c^\va_{rrr}\|_{L^2}^{\frac{1}{2}} +\|c^0\|_{H^3})(\kappa\lambda +\kappa \|c^\va\|_{H^1}+\|c^0\|_{H^2})\\
\leq& C_0\va^{\frac{1}{2}}(
\va^{\frac{1}{2}}\|c^\va_{rr}\|_{L^2}^2
+\va^{\frac{3}{2}}\|c^\va_{rrr}\|_{L^2}^2
+\ka^2\|c^\va\|_{H^1}^2
+\|c^0\|_{H^3}^2+\ka^2\lambda^2),\\
\end{split}
\enn
where in the third inequality we have used the fact that $0<\va<1$.
The Sobolev embedding inequality entails that
\be\label{a86}
\begin{split}
K_2\leq & C_0\va(\|c^0_{rrr}\|_{L^2}+\|c^\va_{rr}\|_{L^2}+\|c^\va_{r}\|_{L^2})
\|\ti{c}_{r}\|_{L^2}\\
&+C_0[\|c^\va\|_{L^\infty}\|\ti{w}_r\|_{L^2}
+\|\ti{w}\|_{L^\infty}\|c^\va_r\|_{L^2}
+\|w^0\|_{L^\infty}\|\ti{c}_r\|_{L^2}
+\|w^0_r\|_{L^\infty}\|\ti{c}\|_{L^2}]\|\ti{c}_r\|_{L^2}\\
\leq & C_0\va(\|c^0\|_{H^3}+\|c^\va_{rr}\|_{L^2}+\|c^\va_{r}\|_{L^2})\|\ti{c}_{r}\|_{L^2}\\
&+C_0[\|c^\va\|_{H^1}\|\ti{w}_r\|_{L^2}
+(\|\ti{w}\|_{L^2}+\|\ti{w}_r\|_{L^2})\|c^\va_r\|_{L^2}]\|\ti{c}_r\|_{L^2}\\
&+C_0[\|w^0\|_{H^1}\|\ti{c}_r\|_{L^2}
+\|w^0\|_{H^2}\|\ti{c}\|_{L^2}]\|\ti{c}_r\|_{L^2}\\
\leq &\frac{1}{4}\|\ti{w}_r\|^2_{L^2}
+C_0(\|c^\va\|_{H^1}^2+\|w^0\|_{H^2}^2+1)
\|\ti{c}_r\|_{L^2}^2
+C_0(\|\ti{w}\|^2_{L^2}+\|\ti{c}\|^2_{L^2})\\
&+C_0\va^2(\|c^0\|^2_{H^3}+\|c^\va_{rr}\|_{L^2}^2+\|c^\va_{r}\|_{L^2}^2).
\end{split}
\ee
Substituting the above estimates for $K_1$ and $K_2$ into \eqref{a26} one deduces that
\be\label{a84}
\begin{split}
\frac{1}{2}\frac{d}{dt}&\int_a^b \ti{c}_r^2dr
+\va \int_a^b\ti{c}_{rr}^2 dr\\
\leq & \frac{1}{4}\|\ti{w}_r\|^2_{L^2}
+C_0(\|c^\va\|_{H^1}^2+\|w^0\|_{H^2}^2+1)
\|\ti{c}_r\|^2_{L^2}\\
&+C_0(\|\ti{w}\|^2_{L^2}+\|\ti{c}\|^2_{L^2})
+C_0\va^2(\|c^0\|^2_{H^3}+\|c^\va_{rr}\|^2_{L^2}+\|c^\va_{r}\|^2_{L^2})\\
&+C_0\va^{\frac{1}{2}}(
\va^{\frac{1}{2}}\|c^\va_{rr}\|_{L^2}^2
+\va^{\frac{3}{2}}\|c^\va_{rrr}\|_{L^2}^2
+\ka^2\|c^\va\|_{H^1}^2
+\|c^0\|_{H^3}^2+\ka^2\lambda^2).
\end{split}
\ee
Adding \eqref{a84} and \eqref{a25} to \eqref{a24} we arrive at
\be\label{a62}
\begin{split}
\frac{d}{dt}&\left[
\int_a^b \ti{w}^2dr+
\int_a^b \ti{c}^2dr+
\int_a^b \ti{c}_r^2dr
\right]
+\int_a^b  \ti{w}_r^2dr+\va\int_a^b \ti{c}_{rr}^2dr\\
&\leq C_0(\|w^\va\|_{H^1}^2+\|c^\va\|_{H^1}^2+\|w^0\|_{H^2}^2+1)\|\ti{c}_r\|^2_{L^2}
+C_0(\|c^\va\|_{H^1}^2+\|c^0\|_{H^2}^2+1)\|\ti{w}\|^2_{L^2}\\
&\ \ \ \ +C_0(\|w^0\|_{H^1}+1)\|\ti{c}\|^2_{L^2}
+C_0\va^2(\|c^\va_{rr}\|_{L^2}^2+\|c^\va_{r}\|_{L^2}^2+\|c^0\|_{H^3}^2)\\
&\ \ \ \ +C_0\va^{\frac{1}{2}}(
\va^{\frac{1}{2}}\|c^\va_{rr}\|_{L^2}^2
+\va^{\frac{3}{2}}\|c^\va_{rrr}\|_{L^2}^2
+\ka^2\|c^\va\|_{H^1}^2
+\|c^0\|_{H^3}^2+\ka^2\lambda^2).
\end{split}
\ee
Applying Gronwall's inequality to \eqref{a62}, using \eqref{a0}, \eqref{a000} and the fact $0<\va<1$, one derives \eqref{a60}.

We proceed to proving \eqref{a61}. Since \eqref{a80} and \eqref{a81} hold true, from Lemma \ref{l0} we know that the boundary term $K_1$ in \eqref{a26} deserves a better estimates as follows:
\be\label{a85}
\begin{split}
K_1=&\va [\ti{c}_{rr}\ti{c}_r]\big|_{r=a}^{r=b}\\
=&-\va\ka(\ti{c}_{rr}\ti{c})(a,t)-\va\ka(\ti{c}_{rr}\ti{c})(b,t)\\
\leq &2\va\ka (\|c^\va_{rr}\|_{L^\infty}+\|c^0_{rr}\|_{L^\infty})
\|\ti{c}\|_{L^\infty}\\
\leq &C_0\va\ka(\|c^\va_{rr}\|_{L^2}
+\|c^\va_{rr}\|^{\frac{1}{2}}_{L^2}
\|c^\va_{rrr}\|^{\frac{1}{2}}_{L^2}
+\|c^0\|_{H^3})(\|\ti{c}_r\|_{L^2}+\|\ti{c}\|_{L^2})\\
\leq& (\|\ti{c}_r\|^2_{L^2}
+\|\ti{c}\|^2_{L^2})
+C_0\ka^2\va^2(\|c^\va_{rr}\|^2_{L^2}+\|c^0\|_{H^3}^2)
+C_0\ka^2(\va^{\frac{3}{2}}\|c^\va_{rr}\|_{L^2}^2+
\va^{\frac{5}{2}}\|c^\va_{rrr}\|_{L^2}^2)\\
\leq & (\|\ti{c}_r\|_{L^2}^2
+\|\ti{c}\|^2_{L^2})
+C_0\ka^2\va(\va^{\frac{1}{2}}\|c^\va_{rr}\|_{L^2}^2+
\va^{\frac{3}{2}}\|c^\va_{rrr}\|_{L^2}^2+\|c^0\|_{H^3}^2),
\end{split}
\ee
where in the last inequality we have used the fact that $0<\va<1$.
Inserting \eqref{a85} and \eqref{a86} into \eqref{a26} we have
\be\label{a87}
\begin{split}
\frac{1}{2}\frac{d}{dt}&\int_a^b \ti{c}_r^2dr
+\va \int_a^b\ti{c}_{rr}^2 dr\\
\leq & \frac{1}{4}\|\ti{w}_r\|^2_{L^2}
+C_0(\|c^\va\|_{H^1}^2+\|w^0\|_{H^2}^2+1)
\|\ti{c}_r\|_{L^2}^2\\
&+C_0(\|\ti{w}\|^2_{L^2}+\|\ti{c}\|_{L^2}^2)
+C_0\va^2(\|c^0\|^2_{H^3}+\|c^\va_{rr}\|_{L^2}^2+\|c^\va_{r}\|_{L^2}^2)\\
&+C_0\ka^2\va(\va^{\frac{1}{2}}\|c^\va_{rr}\|_{L^2}^2+
\va^{\frac{3}{2}}\|c^\va_{rrr}\|_{L^2}^2+\|c^0\|_{H^3}^2).
\end{split}
\ee
Adding \eqref{a87} and \eqref{a25} to \eqref{a24} one gets
\be\label{a88}
\begin{split}
\frac{d}{dt}&\left[
\int_a^b \ti{w}^2dr+
\int_a^b \ti{c}^2dr+
\int_a^b \ti{c}_r^2dr
\right]
+\int_a^b  \ti{w}_r^2dr+\va\int_a^b \ti{c}_{rr}^2dr\\
&\leq C_0(\|w^\va\|_{H^1}^2+\|c^\va\|_{H^1}^2+\|w^0\|_{H^2}^2+1)\|\ti{c}_r\|_{L^2}^2
+C_0(\|c^\va\|_{H^1}^2+\|c^0\|_{H^2}^2+1)\|\ti{w}\|_{L^2}^2\\
&\ \ \ \ +C_0(\|w^0\|_{H^1}+1)\|\ti{c}\|_{L^2}^2
+C_0\va^2(\|c^\va_{rr}\|_{L^2}^2+\|c^\va_{r}\|_{L^2}^2+\|c^0\|_{H^3}^2)\\
&\ \ \ \ +C_0\ka^2\va(\va^{\frac{1}{2}}\|c^\va_{rr}\|_{L^2}^2+
\va^{\frac{3}{2}}\|c^\va_{rrr}\|_{L^2}^2+\|c^0\|_{H^3}^2).
\end{split}
\ee
Applying Gronwall's inequality to \eqref{a88} and using \eqref{a0} and \eqref{a000} we obtain \eqref{a61}.
The proof is completed.

\end{proof}

\begin{lemma} Suppose $\ka>0$ and the assumptions in Theorem \ref{t2} hold. Let $(w^\va,c^\va)$ with $0\leq \va<1$ be the solutions of \eqref{e1}-\eqref{e2}, derived in Lemma \ref{l9}. Then there exists a constant $C_{\lambda}$ depending on $a$, $b$, $n$, $\|w_0\|_{H^2}$, $\|c_0\|_{H^3}$, $T$ and $\lambda$ such that
\be\label{a66}
\begin{split}
\sup_{0<t\leq T}&\left[\int_a^b \ti{w}_t^2 dr
+\int_a^b \ti{w}_r^2 dr\right]
+\int_0^T\int_a^b \ti{w}_{t}^2 drdt
+\int_0^T\int_a^b \ti{w}_{rt}^2 drdt\\
&+\int_0^T\int_a^b \ti{c}_{rt}^2drdt
\leq C_{\lambda}(1+\exp\{\exp\{\exp\{\exp\{C_{\lambda}\ka^2\}\}\}\})^{4}
\cdot\va^{\frac{1}{2}}.
\end{split}
\ee
Assume further that \eqref{a80} and \eqref{a81} hold true. Then
\be\label{a67}
\begin{split}
\sup_{0<t\leq T}&\left[\int_a^b (\ti{w}_t)^2 dr
+\int_a^b \ti{w}_r^2 dr\right]
+\int_0^T\int_a^b \ti{w}_{t}^2 drdt
+\int_0^T\int_a^b \ti{w}_{rt}^2 drdt\\
&+\int_0^T\int_a^b \ti{c}_{rt}^2drdt
\leq C_{\lambda}(1+\exp\{\exp\{\exp\{\exp\{C_{\lambda}\ka^2\}\}\}\})^{4}\cdot\va,
\end{split}
\ee
where the constant $C_{\lambda}$ depends on $a$, $b$, $n$, $\|w_0\|_{H^2}$, $\|c_0\|_{H^3}$, $T$ and $\lambda$.
\end{lemma}
\begin{proof}
Taking the $L^2$ inner product of \eqref{a30} with $\ti{c}_{rt}$ and using integration by parts to have
\be\label{a90}
\begin{split}
\frac{1}{2}\va\frac{d}{dt}&\int_a^b \ti{c}_{rr}^2dr
+\int_a^b \ti{c}_{rt}^2dr\\
=&\va[\ti{c}_{rr}\ti{c}_{rt}]\big|_{r=a}^{r=b}
+\int_a^b \big[\va c^0_{rrr}+\va\frac{n-1}{r}c^\va_{rr}
-\va\frac{n-1}{r^2}c^\va_r
-(\ti{w}c^\va+w^0\ti{c})_r
\big]\ti{c}_{rt}dr\\
:=&K_3+K_4.
\end{split}
\ee
 By the third and fourth boundary condition in \eqref{e4} and Gagliardo-Nirenberg interpolation inequality one deduces that
\ben
\begin{split}
K_3=&-\va\ti{c}_{rr}(b,t)[\kappa c^\va_t(b,t)+c^0_{rt}(b,t)]
-\va\ti{c}_{rr}(a,t)[\kappa c^\va_t(a,t)-c^0_{rt}(a,t)]\\
\leq &2\va  (\|c^\va_{rr}\|_{L^\infty}+\|c^0_{rr}\|_{L^\infty})
(\ka\|c^\va_t\|_{L^\infty}+\|c^0_{rt}\|_{L^\infty})\\
\leq& C_0\va (\|c^\va_{rr}\|_{L^2}
+\|c^\va_{rr}\|_{L^2}^{\frac{1}{2}}\|c^\va_{rrr}\|_{L^2}^{\frac{1}{2}}
+\|c^0\|_{H^3})
(\ka\|c^\va_t\|_{H^1}+\|c^0_t\|_{H^2})\\
\leq &C_0\va^{\frac{1}{2}}(\va^{\frac{1}{4}}\|c^\va_{rr}\|_{L^2}
+\va^{\frac{1}{8}}\|c^\va_{rr}\|_{L^2}^{\frac{1}{2}}
\cdot\va^{\frac{3}{8}}\|c^\va_{rrr}\|_{L^2}^{\frac{1}{2}}
+\|c^0\|_{H^3})(\ka \|c^\va_t\|_{H^1}+\|c^0_t\|_{H^2})\\
\leq &C_0\va^{\frac{1}{2}}(
\va^{\frac{1}{2}}\|c^\va_{rr}\|_{L^2}^2
+\va^{\frac{3}{2}}\|c^\va_{rrr}\|_{L^2}^2
+\ka^2\|c^\va_t\|_{H^1}^2
+\|c^0\|_{H^3}^2+\|c^0_t\|_{H^2}^2),
\end{split}
\enn
where in the third inequality we have used the fact that $0<\va<1$.
The Sobolev embedding inequality indicates that
\be\label{a93}
\begin{split}
K_4\leq & C_0\va(\|c^0_{rrr}\|_{L^2}+\|c^\va_{rr}\|_{L^2}+\|c^\va_{r}\|_{L^2})
\|\ti{c}_{rt}\|_{L^2}\\
&+C_0(\|c^\va\|_{L^\infty}\|\ti{w}_r\|_{L^2}
+\|\ti{w}\|_{L^\infty}\|c^\va_r\|_{L^2}
+\|w^0\|_{L^\infty}\|\ti{c}_r\|_{L^2}
+\|w^0_r\|_{L^\infty}\|\ti{c}\|_{L^2})\|\ti{c}_{rt}\|_{L^2}\\
\leq & C_0\va(\|c^0\|_{H^3}+\|c^\va_{rr}\|_{L^2}+\|c^\va_{r}\|_{L^2})\|\ti{c}_{rt}\|_{L^2}\\
&+C_0(\|c^\va\|_{H^1}\|\ti{w}_r\|_{L^2}
+\|\ti{w}\|_{H^1}\|c^\va_r\|_{L^2}
+\|w^0\|_{H^1}\|\ti{c}_r\|_{L^2}
+\|w^0\|_{H^2}\|\ti{c}\|_{L^2})\|\ti{c}_{rt}\|_{L^2}\\
\leq &\frac{1}{2}\|\ti{c}_{rt}\|_{L^2}^2
+C_0(\|c^\va\|_{H^1}^2\|\ti{w}\|_{H^1}^2+\|w^0\|_{H^2}^2
\|\ti{c}\|_{H^1}^2)
+C_0\va^2(\|c^0\|^2_{H^3}+\|c^\va_{rr}\|_{L^2}^2+\|c^\va_{r}\|_{L^2}^2).
\end{split}
\ee
Substituting the above estimates for $K_3$ and $K_4$ into \eqref{a90} we arrive at
\ben
\begin{split}
\va\frac{d}{dt}&\int_a^b \ti{c}_{rr}^2dr
+\int_a^b \ti{c}_{rt}^2dr\\
\leq &C_0(\|c^\va\|_{H^1}^2\|\ti{w}\|_{H^1}^2+\|w^0\|_{H^2}^2
\|\ti{c}\|_{H^1}^2)
+C_0\va^2(\|c^0\|^2_{H^3}+\|c^\va_{rr}\|_{L^2}^2+\|c^\va_{r}\|_{L^2}^2)\\
&+C_0\va^{\frac{1}{2}}(
\va^{\frac{1}{2}}\|c^\va_{rr}\|_{L^2}^2
+\va^{\frac{3}{2}}\|c^\va_{rrr}\|_{L^2}^2
+\ka^2\|c^\va_t\|_{H^1}^2
+\|c^0\|_{H^3}^2+\|c^0_t\|_{H^2}^2).
\end{split}
\enn
 Integrating the above inequality over $(0,T)$ and using \eqref{a0}, \eqref{a000} and \eqref{a60} to have
\be\label{a91}
\begin{split}
\sup_{0<t\leq T}\va \int_a^b \ti{c}_{rr}^2dr
+\int_0^T\int_a^b \ti{c}_{rt}^2drdt
\leq
C_{\lambda}(1+\exp\{\exp\{\exp\{\exp\{C_{\lambda}\ka^2\}\}\}\})^{3}\cdot
\va^{\frac{1}{2}}.
\end{split}
\ee
Taking the $L^2$ inner product of the first equation of \eqref{e4} with $\ti{w}_t$, one gets
\ben
\begin{split}
\frac{1}{2}\frac{d}{dt}\int_a^b \ti{w}_r^2 dr
+\int_a^b \ti{w}_t^2 dr
=&\int_{a}^b \frac{n-1}{r}\ti{w}_r\ti{w}_t dr
+\int_a^b (w^\va \ti{c}_r+\ti{w}c^0_r)\ti{w}_{rt} dr\\
&-\int_a^b\frac{n-1}{r}(w^\va \ti{c}_r+\ti{w}c^0_r)\ti{w}_{t} dr,
\end{split}
\enn
where
\ben
\int_{a}^b \frac{n-1}{r}\ti{w}_r\ti{w}_t dr
\leq \frac{1}{4}\|\ti{w}_t\|^2_{L^2}+C_0\|\ti{w}_r\|_{L^2}^2
\enn
and
\ben
\begin{split}
\int_a^b &(w^\va \ti{c}_r+\ti{w}c^0_r)\ti{w}_{rt} dr
-\int_a^b\frac{n-1}{r}(w^\va \ti{c}_r+\ti{w}c^0_r)\ti{w}_{t} dr\\
&\leq \frac{1}{4}\|\ti{w}_{rt}\|_{L^2}^2+\frac{1}{4}\|\ti{w}_t\|_{L^2}^2
+C_0(\|w^\va\|_{L^\infty}^2\|\ti{c}_r\|_{L^2}^2
+\|\ti{w}\|_{L^\infty}^2\|c^0_r\|_{L^2}^2)\\
&\leq \frac{1}{4}\|\ti{w}_{rt}\|_{L^2}^2+\frac{1}{4}\|\ti{w}_t\|_{L^2}^2
+C_0\|w^\va\|^2_{H^1}\|\ti{c}_r\|_{L^2}^2
+C_0\|\ti{w}\|_{H^1}^2\|c^0\|^2_{H^1}.
\end{split}
\enn
Thus
\be\label{a27}
\begin{split}
\frac{d}{dt}\int_a^b \ti{w}_r^2 dr
+\int_a^b \ti{w}_t^2 dr
\leq& \frac{1}{2}\|\ti{w}_{rt}\|_{L^2}^2
+C_0\|w^\va\|^2_{H^1}\|\ti{c}_r\|_{L^2}^2\\
&+C_0(\|c^0\|^2_{H^1}+1)\|\ti{w}\|^2_{H^1}.
\end{split}
\ee
Differentiating the first equation of \eqref{e4} with respect to $t$ to have
\ben
\ti{w}_{tt}=\ti{w}_{rrt}+\frac{n-1}{r}\ti{w}_{rt}
-(w^\va\ti{c}_r+\ti{w}c^0_r)_{rt}
-\frac{n-1}{r}(w^\va \ti{c}_r+\ti{w}c^0_r)_t,
\enn
which, multiplied with $\ti{w}_t$ in $L^2$ leads to
\ben
\begin{split}
\frac{1}{2}\frac{d}{dt}\int_a^b \ti{w}_t^2 dr
+\int_a^b \ti{w}_{rt}^2 dr
=&\int_a^b \frac{n-1}{r}\ti{w}_{rt}\ti{w}_t dr
+\int_a^b (w^\va \ti{c}_r+\ti{w}c^0_r)_t \ti{w}_{rt} dr\\
&-\int_a^b\frac{n-1}{r}(w^\va \ti{c}_r+\ti{w}c^0_r)_t\ti{w}_{t} dr\\
:=&K_5+K_6+K_7.
\end{split}
\enn
The Cauchy-Schwarz inequality entails that
\ben
K_5\leq \frac{1}{8}\|\ti{w}_{rt}\|_{L^2}^2
+C_0\|\ti{w}_{t}\|_{L^2}^2.
\enn
It follows from the Sobolev embedding inequality that
\ben
\begin{split}
K_6=&\int_a^b (w^\va_t \ti{c}_r+ w^\va \ti{c}_{rt} + \ti{w}_tc^0_r+ \ti{w}c^0_{rt}) \ti{w}_{rt} dr\\
\leq &\frac{1}{8}\|\ti{w}_{rt}\|_{L^2}^2
+C_0(\|w^\va_t\|_{L^\infty}^2\|\ti{c}_r\|_{L^2}^2
+\|w^\va\|_{L^\infty}^2\|\ti{c}_{rt}\|_{L^2}^2
+\|\ti{w}_t\|^2_{L^2}\|c^0_r\|_{L^\infty}^2
+\|\ti{w}\|_{L^\infty}^2 \|c^0_{rt}\|_{L^2}^2)\\
\leq &\frac{1}{8}\|\ti{w}_{rt}\|_{L^2}^2
+C_0(\|w^\va_t\|_{H^1}^2\|\ti{c}_r\|_{L^2}^2
+\|w^\va\|_{H^1}^2\|\ti{c}_{rt}\|_{L^2}^2
+\|\ti{w}_t\|^2_{L^2}\|c^0\|_{H^2}^2
+
\|\ti{w}\|_{H^1}^2 \|c^0_{t}\|^2_{H^1}).
\end{split}
\enn
By a similar argument used in estimating $K_6$, one gets
\ben
K_7
\leq \|\ti{w}_{t}\|_{L^2}^2
+C_0(\|w^\va_t\|_{H^1}^2\|\ti{c}_r\|_{L^2}^2
+\|w^\va\|_{H^1}^2\|\ti{c}_{rt}\|_{L^2}^2
+\|\ti{w}_t\|^2_{L^2}\|c^0\|_{H^2}^2
+
\|\ti{w}\|^2_{H^1}\|c^0_{t}\|^2_{H^1}).
\enn
Collecting the above estimates for $K_5$, $K_6$ and $K_7$ we obtain
\ben
\begin{split}
\frac{d}{dt}\int_a^b \ti{w}_t^2 dr
+\frac{3}{2}\int_a^b \ti{w}_{rt}^2 dr
\leq& C_0(1+\|c^0\|_{H^2}^2)\|\ti{w}_{t}\|_{L^2}^2
+C_0(\|w^\va_t\|_{H^1}^2\|\ti{c}_r\|_{L^2}^2
+\|w^\va\|_{H^1}^2\|\ti{c}_{rt}\|_{L^2}^2)\\
&+C_0
\|\ti{w}\|_{H^1}^2 \|c^0_{t}\|^2_{H^1},
\end{split}
\enn
which, added to \eqref{a27} yields
\be\label{a28}
\begin{split}
\frac{d}{dt}&\left[\int_a^b \ti{w}_t^2 dr
+\int_a^b \ti{w}_r^2 dr\right]
+\int_a^b \ti{w}_{t}^2 dr
+\int_a^b \ti{w}_{rt}^2 dr\\
& \leq C_0(1+\|c^0\|_{H^2}^2)\|\ti{w}_{t}\|_{L^2}^2
+C_0(\|w^\va_t\|_{H^1}^2+\|w^\va\|_{H^1}^2)\|\ti{c}_r\|_{L^2}^2\\
&\ \ \ +C_0(\|c^0_{t}\|^2_{H^1}+\|c^0\|^2_{H^1}+1)
\|\ti{w}\|_{H^1}^2
+C_0\|w^\va\|_{H^1}^2\|\ti{c}_{rt}\|_{L^2}^2.
\end{split}
\ee
Applying Gronwall's inequality to \eqref{a28}, and using \eqref{a0}, \eqref{a000}, \eqref{a60} and \eqref{a91} one gets \eqref{a66}.

We next prove \eqref{a67}. Since \eqref{a80} and \eqref{a81} hold true, from Lemma \ref{l0} we know that the boundary term $K_3$ in \eqref{a90} enjoys a better estimate as follows:
\be\label{a92}
\begin{split}
K_3=&\va[\ti{c}_{rr}\ti{c}_{rt}]\big|_{r=a}^{r=b}\\
=&-\ka\va(\ti{c}_{rr}\ti{c}_t)(b,t)-\ka\va(\ti{c}_{rr}\ti{c}_t)(a,t)\\
\leq &2\ka\va (\|c^\va_{rr}\|_{L^\infty}+\|c^0_{rr}\|_{L^\infty})
\|\ti{c}_t\|_{L^\infty}\\
\leq &C_0\ka\va(\|c^\va_{rr}\|_{L^2}
+\|c^\va_{rr}\|^{\frac{1}{2}}_{L^2}
\|c^\va_{rrr}\|^{\frac{1}{2}}_{L^2}
+\|c^0\|_{H^3})(\|\ti{c}_{rt}\|_{L^2}
+\|\ti{c}_t\|_{L^2})\\
\leq& \frac{1}{4}\|\ti{c}_{rt}\|_{L^2}^2
+\|\ti{c}_t\|^2_{L^2}
+C_0\ka^2\va^2(\|c^\va_{rr}\|_{L^2}^2+\|c^0\|_{H^3}^2)
+C_0\ka^2(\va^{\frac{3}{2}}\|c^\va_{rr}\|_{L^2}^2+
\va^{\frac{5}{2}}\|c^\va_{rrr}\|_{L^2}^2)\\
\leq & \frac{1}{4}\|\ti{c}_{rt}\|^2_{L^2}
+\|\ti{c}_t\|^2_{L^2}
+C_0\ka^2\va(\va^{\frac{1}{2}}\|c^\va_{rr}\|_{L^2}^2+
\va^{\frac{3}{2}}\|c^\va_{rrr}\|^2_{L^2}+\|c^0\|_{H^3}^2),
\end{split}
\ee
where in the last inequality we have used the fact that $0<\va<1$. Inserting \eqref{a92} and \eqref{a93} into \eqref{a90} one gets
\ben
\begin{split}
\va\frac{d}{dt}&\int_a^b \ti{c}_{rr}^2dr
+\int_a^b \ti{c}_{rt}^2dr\\
\leq &\|\ti{c}_t\|_{L^2}^2+C_0(\|c^\va\|_{H^1}^2\|\ti{w}\|_{H^1}^2+\|w^0\|_{H^2}^2
\|\ti{c}\|_{H^1}^2)
+C_0\va^2(\|c^0\|^2_{H^3}+\|c^\va_{rr}\|_{L^2}^2+\|c^\va_{r}\|_{L^2}^2)\\
&
+C_0\ka^2\va(\va^{\frac{1}{2}}\|c^\va_{rr}\|_{L^2}^2+
\va^{\frac{3}{2}}\|c^\va_{rrr}\|^2_{L^2}+\|c^0\|_{H^3}^2),
\end{split}
\enn
where the first term on the right-hand side can be estimated by using the second equation of \eqref{e4} as follows:
\ben
\|\ti{c}_t\|_{L^2}^2
\leq C_0\va^2(\|c^\va_{rr}\|^2_{L^2}+\|c^\va_{r}\|_{L^2}^2)
+C_0(\|c^\va\|_{H^1}^2\|\ti{w}\|_{H^1}^2
+\|w^0\|_{H^1}^2\|\ti{c}\|_{H^1}^2).
\enn
Thus,
\ben
\begin{split}
\va\frac{d}{dt}&\int_a^b \ti{c}_{rr}^2dr
+\int_a^b \ti{c}_{rt}^2dr\\
\leq &C_0(\|c^\va\|_{H^1}^2\|\ti{w}\|_{H^1}^2+\|w^0\|_{H^2}^2
\|\ti{c}\|_{H^1}^2)
+C_0\va^2(\|c^0\|^2_{H^3}+\|c^\va_{rr}\|^2_{L^2}+\|c^\va_{r}\|^2_{L^2})\\
&
+C_0\ka^2\va(\va^{\frac{1}{2}}\|c^\va_{rr}\|_{L^2}^2+
\va^{\frac{3}{2}}\|c^\va_{rrr}\|_{L^2}^2+\|c^0\|_{H^3}^2).
\end{split}
\enn
Integrating this over $(0,T)$, using \eqref{a0}, \eqref{a000} and \eqref{a61} we have
\be\label{a94}
\begin{split}
\sup_{0<t\leq T}\va \int_a^b \ti{c}_{rr}^2dr
+\int_0^T\int_a^b \ti{c}_{rt}^2drdt\leq
C_{\lambda}(1+\exp\{\exp\{\exp\{\exp\{C_{\lambda}\ka^2\}\}\}\})^{3}
\cdot\va.
\end{split}
\ee
Applying Gronwall's inequality to \eqref{a28} and using \eqref{a0}, \eqref{a000}, \eqref{a61} and \eqref{a94} one gets \eqref{a67}.
The proof is finished.

\end{proof}

\begin{lemma}\label{l8}
 Suppose $\ka>0$ and the assumptions in Theorem \ref{t2} hold. Let $(w^\va,c^\va)$ with $0\leq \va<1$ be the solutions of \eqref{e1}-\eqref{e2}, derived in Lemma \ref{l9}. Then there exists a constant $C_{\lambda}$ depending on $a$, $b$, $n$, $\|w_0\|_{H^2}$, $\|c_0\|_{H^3}$, $T$ and $\lambda$ such that
\ben
\begin{split}
\sup_{0<t\leq T}&\left[\int_a^b(r-a)^2(r-b)^2\ti{c}_{rr}^2(t)dr
+\int_a^b(r-a)^2(r-b)^2\ti{w}_{rr}^2(t)dr
\right]\\
&+\va\int_0^T\int_a^b(r-a)^2(r-b)^2\ti{c}_{rrr}^2drdt
\leq C_{\lambda}(1+\exp\{\exp\{\exp\{\exp\{C_{\lambda}
\ka^2\}\}\}\})^{6}\cdot\va^{\frac{1}{2}}.
\end{split}
\enn
\end{lemma}
\begin{proof}
Taking the $L^2$ inner product of \eqref{a30} with  $-(r-a)^2(r-b)^2\ti{c}_{rrr}$ and using integration by parts, one gets
\be\label{a32}
\begin{split}
\frac{1}{2}\frac{d}{dt}&\int_a^b(r-a)^2(r-b)^2(\ti{c}_{rr})^2dr
+\va\int_a^b(r-a)^2(r-b)^2(\ti{c}_{rrr})^2dr\\
=&-2\int_a^b (r-a)(r-b)(2r-a-b)\ti{c}_{rr}\ti{c}_{rt}dr\\
&-2\int_a^b (r-a)(r-b)(2r-a-b)(\ti{w}c^\va+w^0\ti{c})_r\ti{c}_{rr}dr\\
&+\va\int_a^b\frac{n-1}{r^2}c^\va_r(r-a)^2(r-b)^2\ti{c}_{rrr}dr
-\va\int_a^b c^0_{rrr}(r-a)^2(r-b)^2\ti{c}_{rrr}dr\\
&-\va\int_a^b\frac{n-1}{r}c^\va_{rr}(r-a)^2(r-b)^2\ti{c}_{rrr}dr
-\int_a^b(\ti{w}c^\va+w^0\ti{c})_{rr}(r-a)^2(r-b)^2\ti{c}_{rr}dr\\
:=&\sum_{i=8}^{i=13}K_i.
\end{split}
\ee
It follows from the Cauchy-Schwarz inequality and Sobolev embedding inequality that
\ben
\begin{split}
K_8+K_9
\leq& \|(r-a)(r-b)\ti{c}_{rr}\|^2_{L^2}
+C_0\|\ti{c}_{rt}\|_{L^2}^2\\
&+C_0(\|\ti{w}_r\|^2_{L^2}\|c^\va\|_{L^\infty}^2
+\|\ti{w}\|_{L^\infty}^2\|c^\va_r\|^2_{L^2}
+\|w^0_r\|^2_{L^2}\|\ti{c}\|_{L^\infty}^2
+\|w^0\|_{L^\infty}^2\|\ti{c}_r\|^2_{L^2}
)\\
\leq &\|(r-a)(r-b)\ti{c}_{rr}\|^2_{L^2}
+C_0\|\ti{c}_{rt}\|^2_{L^2}
+C_0(\|\ti{w}\|_{H^1}^2\|c^\va\|_{H^1}^2
+\|w^0\|_{H^1}^2\|\ti{c}\|_{H^1}^2
)
\end{split}
\enn
and that
\ben
K_{10}+K_{11}\leq \frac{1}{4}\va \|(r-a)(r-b)\ti{c}_{rrr}\|_{L^2}^2
+C_0\va(\|c^\va_r\|_{L^2}^2+\|c^0\|_{H^3}^2).
\enn
From the definition of $\ti{c}$ we deduce that
\ben
\begin{split}
K_{12}=&-\va\int_a^b\frac{n-1}{r}\ti{c}_{rr}(r-a)^2(r-b)^2\ti{c}_{rrr}dr
-\va\int_a^b\frac{n-1}{r}c^0_{rr}(r-a)^2(r-b)^2\ti{c}_{rrr}dr\\
\leq &
\frac{1}{4}\va \|(r-a)(r-b)\ti{c}_{rrr}\|_{L^2}^2
+C_0\va (\|(r-a)(r-b)\ti{c}_{rr}\|_{L^2}^2+\|c^0\|_{H^2}^2).
\end{split}
\enn
A direct computation leads to
\be\label{a31}
\begin{split}
K_{13}=&\int_a^b
(r-a)^2(r-b)^2\ti{w}_{rr}c^\va\ti{c}_{rr}dr
+2\int_a^b
(r-a)^2(r-b)^2\ti{w}_r c^\va_r\ti{c}_{rr}dr\\
&+\int_a^b
(r-a)^2(r-b)^2\ti{w}c^\va_{rr}\ti{c}_{rr}dr\\
&+\int_a^b (w^0_{rr}\ti{c}+2w^0_r\ti{c}_r+w^0\ti{c}_{rr})
 (r-a)^2(r-b)^2\ti{c}_{rr}dr\\
 :=&\sum_{i=1}^{i=4}M_i.
\end{split}
\ee
By the first equation of \eqref{e4} and Sobolev embedding inequality, we estimate $M_1$ as follows:
\ben
\begin{split}
M_1=&\int_a^b \left(\ti{w}_t-\frac{n-1}{r}\ti{w}_r\right)
c^\va (r-a)^2(r-b)^2 \ti{c}_{rr}dr
+\int_a^b (\ti{w}_rc^0_r+\ti{w}c^0_{rr})c^\va(r-a)^2(r-b)^2 \ti{c}_{rr}dr\\
&+\int_a^b \frac{n-1}{r}(w^\va \ti{c}_r+\ti{w}c^0_r)
c^\va(r-a)^2(r-b)^2 \ti{c}_{rr}dr\\
&+\int_a^b(w^\va_r\ti{c}_r+w^\va\ti{c}_{rr})c^\va
(r-a)^2(r-b)^2 \ti{c}_{rr}dr\\
\leq& C_0(\|\ti{w}_t\|_{L^2}+\|\ti{w}_r\|_{L^2}
+\|\ti{w}_r\|_{L^2}\|c^0_r\|_{L^\infty}
+\|\ti{w}\|_{L^\infty}\|c^0_{rr}\|_{L^2})
\|c^\va\|_{L^\infty}\|(r-a)(r-b)\ti{c}_{rr}\|_{L^2}\\
&+C_0(
\|w^\va\|_{L^\infty}\|\ti{c}_r\|_{L^2}
+\|\ti{w}\|_{L^\infty}\|c^0_r\|_{L^2})\|c^\va\|_{L^\infty}
\|(r-a)(r-b)\ti{c}_{rr}\|_{L^2}\\
&+C_0\|w^\va_r\|_{L^2}
\|(r-a)(r-b)\ti{c}_{r}\|_{L^\infty}
\|c^\va\|_{L^\infty}\|(r-a)(r-b)\ti{c}_{rr}\|_{L^2}\\
&+C_0\|w^\va\|_{L^\infty}
\|c^\va\|_{L^\infty}\|(r-a)(r-b)\ti{c}_{rr}\|_{L^2}^2\\
\leq& C_0(\|\ti{w}_t\|_{L^2}+\|\ti{w}_r\|_{L^2}
+\|\ti{w}\|_{H^1}\|c^0\|_{H^2}
)\|c^\va\|_{L^\infty}
\|(r-a)(r-b)\ti{c}_{rr}\|_{L^2}\\
&+C_0(
\|w^\va\|_{H^1}\|\ti{c}_r\|_{L^2}
+\|\ti{w}\|_{H^1}\|c^0_r\|_{L^2})\|c^\va\|_{L^\infty}
\|(r-a)(r-b)\ti{c}_{rr}\|_{L^2}\\
&+C_0\|w^\va_r\|_{L^2}
(\|\ti{c}_{r}\|_{L^2}+\|(r-a)(r-b)\ti{c}_{rr}\|_{L^2})
\|c^\va\|_{L^\infty}\|(r-a)(r-b)\ti{c}_{rr}\|_{L^2}\\
&+C_0\|w^\va\|_{H^1}\|c^\va\|_{L^\infty}
\|(r-a)(r-b)\ti{c}_{rr}\|_{L^2}^2\\
\leq & C_0(\|w^\va\|_{H^1}\|c^\va\|_{L^\infty}
+\|c^\va\|_{L^\infty}^2)\|(r-a)(r-b)\ti{c}_{rr}\|_{L^2}^2\\
&+C_0(\|\ti{w}_t\|_{L^2}^2+\|\ti{w}_r\|_{L^2}^2
+\|\ti{w}\|_{H^1}^2\|c^0\|_{H^2}^2
+\|w^\va\|_{H^1}^2\|\ti{c}_r\|_{L^2}^2).
\end{split}
\enn
The
Sobolev embedding inequality gives
\ben
\begin{split}
M_2=&2\int_a^b \ti{w}_r\ti{c}_r(r-a)^2(r-b)^2\ti{c}_{rr}dr
+2\int_a^b \ti{w}_r c^0_r(r-a)^2(r-b)^2\ti{c}_{rr}dr\\
\leq &2\|\ti{w}_r\|_{L^2} \|(r-a)(r-b)\ti{c}_r\|_{L^\infty}
\|(r-a)(r-b)\ti{c}_{rr}\|_{L^2}
+C_0\|\ti{w}_r\|_{L^2} \|c^0_r\|_{L^\infty}
\|(r-a)(r-b)\ti{c}_{rr}\|_{L^2}\\
\leq &C_0\|\ti{w}_r\|_{L^2}( \|\ti{c}_r\|_{L^2}+\|(r-a)(r-b)\ti{c}_{rr}\|_{L^2})
\|(r-a)(r-b)\ti{c}_{rr}\|_{L^2}\\
&+C_0\|\ti{w}_r\|_{L^2} \|c^0\|_{H^2}
\|(r-a)(r-b)\ti{c}_{rr}\|_{L^2}\\
\leq& C_0(\|\ti{w}_r\|_{L^2}+1)\|(r-a)(r-b)\ti{c}_{rr}\|^2_{L^2}
+C_0(\|\ti{w}_r\|^2_{L^2} \|\ti{c}_r\|_{L^2}^2+\|\ti{w}_r\|_{L^2}^2\|c^0\|_{H^2}^2)
\\
\leq& C_0(\|w^\va_r\|_{L^2}+\|w^0_r\|_{L^2}+1)
\|(r-a)(r-b)\ti{c}_{rr}\|_{L^2}^2
+C_0(\|\ti{w}_r\|_{L^2}^2 \|\ti{c}_r\|_{L^2}^2+\|\ti{w}_r\|_{L^2}^2\|c^0\|_{H^2}^2).
\end{split}
\enn
It follows from the Sobolev embedding inequality that
\ben
\begin{split}
M_3=&\int_a^b \ti{w}\ti{c}_{rr}(r-a)^2(r-b)^2\ti{c}_{rr}dr
+\int_a^b \ti{w}c^0_{rr}(r-a)^2(r-b)^2\ti{c}_{rr}dr\\
\leq &\|\ti{w}\|_{L^\infty}\|(r-a)(r-b)\ti{c}_{rr}\|^2_{L^2}
+C_0\|\ti{w}\|_{L^\infty}\|c^0\|_{H^2}\|(r-a)(r-b)\ti{c}_{rr}\|_{L^2}\\
\leq &(\|\ti{w}\|_{H^1}+1)\|(r-a)(r-b)\ti{c}_{rr}\|_{L^2}^2
+C_0\|c^0\|_{H^2}^2\|\ti{w}\|_{H^1}^2\\
\leq &(\|w^\va\|_{H^1}+\|w^0\|_{H^1}+1)\|(r-a)(r-b)\ti{c}_{rr}\|_{L^2}^2
+C_0\|c^0\|_{H^2}^2\|\ti{w}\|_{H^1}^2.
\end{split}
\enn
The Cauchy-Schwarz inequality and Sobolev embedding inequality entail that
\ben
\begin{split}
M_4
\leq& C_0\|w^0\|_{H^2}\|\ti{c}\|_{H^1}\|(r-a)(r-b)\ti{c}_{rr}\|_{L^2}
+\|w^0\|_{L^\infty}\|(r-a)(r-b)\ti{c}_{rr}\|_{L^2}^2\\
\leq &C_0(\|w^0\|_{H^1}+1)\|(r-a)(r-b)\ti{c}_{rr}\|_{L^2}^2
+C_0\|w^0\|_{H^2}^2\|\ti{c}\|_{H^1}^2.
\end{split}
\enn
Substituting the above estimates for $M_1$-$M_4$ into \eqref{a31}, one gets
\be\label{a73}
\begin{split}
K_{13}\leq& C_0(\|w^0\|_{H^1}+\|w^\va\|_{H^1}+\|w^\va\|_{H^1}\|c^\va\|_{L^\infty}
+\|c^\va\|_{L^\infty}^2+1)\|(r-a)(r-b)\ti{c}_{rr}\|_{L^2}^2\\
&+C_0(\|\ti{w}_t\|^2_{L^2}+\|\ti{w}_r\|_{L^2}^2
+\|\ti{w}\|_{H^1}^2\|c^0\|_{H^2}^2
+\|w^\va\|_{H^1}^2\|\ti{c}_r\|_{L^2}^2)\\
&+C_0(\|\ti{w}_r\|_{L^2}^2 \|\ti{c}_r\|_{L^2}^2+\|w^0\|_{H^2}^2\|\ti{c}\|_{H^1}^2).
\end{split}
\ee
Plugging the above estimates for $K_8$-$K_{13}$ into \eqref{a32} and using the fact that $0<\va<1$, we end up with
\ben
\begin{split}
\frac{d}{dt}&\int_a^b(r-a)^2(r-b)^2\ti{c}_{rr}^2dr
+\va\int_a^b(r-a)^2(r-b)^2\ti{c}_{rrr}^2dr\\
\leq &C_0(\|w^0\|_{H^1}+\|w^\va\|_{H^1}+\|w^\va\|_{H^1}\|c^\va\|_{L^\infty}
+\|c^\va\|_{L^\infty}^2+1)\|(r-a)(r-b)\ti{c}_{rr}\|_{L^2}^2\\
&+C_0(\|\ti{c}_{rt}\|_{L^2}^2
+\|\ti{w}\|_{H^1}^2\|c^\va\|_{H^1}^2
+\|w^0\|_{H^2}^2\|\ti{c}\|_{H^1}^2
)+C_0\va(\|c^\va_r\|_{L^2}^2+\|c^0\|_{H^3}^2)\\
&+C_0(\|\ti{w}_t\|^2_{L^2}+\|\ti{w}_r\|_{L^2}^2
+\|\ti{w}\|_{H^1}^2\|c^0\|_{H^2}^2
+\|w^\va\|_{H^1}^2\|\ti{c}_r\|_{L^2}^2
+\|\ti{w}_r\|^2_{L^2} \|\ti{c}_r\|_{L^2}^2),
\end{split}
\enn
which, along with Gronwall's inequality, \eqref{a0}, \eqref{a000}, \eqref{a1}, \eqref{a60} and \eqref{a66} entails that
\be\label{a33}
\begin{split}
\sup_{0<t\leq T}&\int_a^b(r-a)^2(r-b)^2\ti{c}_{rr}^2(t)dr
+\va\int_0^T\int_a^b(r-a)^2(r-b)^2\ti{c}_{rrr}^2drdt\\
\leq& C_{\lambda}(1+\exp\{\exp\{\exp\{\exp\{C_{\lambda}
\ka^2\}\}\}\})^{5}\cdot\va^{\frac{1}{2}},
\end{split}
\ee
where the constant $C_{\lambda}$ depends on $a$, $b$, $n$, $\|w_0\|_{H^2}$, $\|c_0\|_{H^3}$, $T$ and $\lambda$.
It remains to estimate $\sup\limits_{0<t\leq T}\int_a^b(r-a)^2(r-b)^2\ti{w}_{rr}^2(t)dr$. From the first equation of \eqref{e4}, one deduces that
\be\label{a34}
\begin{split}
\int_a^b(&r-a)^2(r-b)^2\ti{w}_{rr}^2dr\\
\leq &C_0\int_a^b (r-a)^2(r-b)^2 (w^\va_r)^2 \ti{c}_r^2dr
+C_0\int_a^b (r-a)^2(r-b)^2 (w^\va)^2 \ti{c}_{rr}^2dr\\
&+C_0\int_a^b(r-a)^2(r-b)^2[\ti{w}_t^2+\ti{w}_r^2
+\ti{w}_r^2(c^0_r)^2+\ti{w}^2(c^0_{rr})^2
+(w^\va)^2\ti{c}_r^2+\ti{w}^2(c^0_r)^2]dr\\
:=&K_{14}+K_{15}+K_{16}.
\end{split}
\ee
The Sobolve embedding inequality and Cauchy-Schwarz inequality lead to
\ben
\begin{split}
K_{14}+K_{15}
\leq &C_0\|w^\va_r\|_{L^2}^2\|(r-a)(r-b)\ti{c}_r\|_{L^\infty}^2
+C_0\|w^\va\|_{L^\infty}^2\|(r-a)(r-b)\ti{c}_{rr}\|_{L^2}^2\\
\leq &C_0\|w^\va_r\|_{L^2}^2(\|\ti{c}_r\|_{L^2}^2+\|(r-a)(r-b)\ti{c}_{rr}\|_{L^2}^2)
+C_0\|w^\va\|_{H^1}^2\|(r-a)(r-b)\ti{c}_{rr}\|_{L^2}^2\\
\leq &C_0\|w^\va\|^2_{H^1}(\|\ti{c}_r\|_{L^2}^2+\|(r-a)(r-b)\ti{c}_{rr}\|_{L^2}^2)
\end{split}
\enn
and
\ben
\begin{split}
K_{16}\leq
&
C_0(\|\ti{w}_t\|_{L^2}^2+\|\ti{w}_r\|_{L^2}^2)
+C_0(\|\ti{w}_r\|_{L^2}^2\|c^0_r\|_{L^\infty}^2
+\|\ti{w}\|_{L^\infty}^2\|c^0_{rr}\|_{L^2}^2)\\
&+C_0(\|w^\va\|_{L^\infty}^2 \|\ti{c}_r\|_{L^2}^2
+\|\ti{w}\|_{L^\infty}^2\|c^0_r\|_{L^2}^2)\\
\leq
&C_0(\|\ti{w}_t\|_{L^2}^2+\|\ti{w}_r\|_{L^2}^2)
+C_0(\|\ti{w}\|_{H^1}^2\|c^0\|_{H^2}^2
+\|w^\va\|_{H^1}^2 \|\ti{c}_r\|_{L^2}^2).\\
\end{split}
\enn
Plugging the above estimates for $K_{14}$-$K_{16}$ into \eqref{a34}, we arrive at
\be\label{a75}
\begin{split}
\int_a^b(r-a)^2(r-b)^2\ti{w}_{rr}^2dr
\leq& C_0
\|w^\va\|^2_{H^1}(\|\ti{c}_r\|_{L^2}^2+\|(r-a)(r-b)\ti{c}_{rr}\|_{L^2}^2)\\
&+C_0(\|\ti{w}\|_{H^1}^2\|c^0\|_{H^2}^2
+\|w^\va\|_{H^1}^2 \|\ti{c}_r\|_{L^2}^2)\\
&+C_0(\|\ti{w}_t\|_{L^2}^2+\|\ti{w}_r\|_{L^2}^2),
\end{split}
\ee
which, in conjunction with \eqref{a0}, \eqref{a000},  \eqref{a60}, \eqref{a66} and \eqref{a33} entails that
\be\label{b64}
\sup_{0<t\leq T}\int_a^b(r-a)^2(r-b)^2\ti{w}_{rr}^2(t)dr\leq C_{\lambda}(1+\exp\{\exp\{\exp\{\exp\{C_{\lambda}\ka^2\}\}\}\})^{6}\cdot\va^{\frac{1}{2}}.
\ee
Collecting \eqref{a33} and \eqref{a75} we obtain the desired estimates. The proof is completed.

\end{proof}

\begin{lemma}
 Suppose $\ka>0$ and the assumptions in Theorem \ref{t2} hold. Let $(w^\va,c^\va)$ with $0\leq \va<1$ be the solutions of \eqref{e1}-\eqref{e2}, derived in Lemma \ref{l9}. Assume further for each $t\in [0,T]$ that \eqref{a80} and \eqref{a81} hold true.
Then there exists a constant $C_{\lambda}$ depending on $a$, $b$, $n$, $\|w_0\|_{H^2}$, $\|c_0\|_{H^3}$, $T$ and $\lambda$ such that
\be\label{a95}
\begin{split}
& \sup_{0<t\leq T}\left[\int_a^b\ti{c}_{rr}^2(t)dr
+\int_a^b\ti{w}_{rr}^2(t)dr
\right]
+\va\int_0^T\int_a^b\ti{c}_{rrr}^2drdt\\
&\ \ \leq C_{\lambda}(1+\exp\{\exp\{\exp\{\exp\{C_{\lambda}\ka^2\}\}\}\})^{6}\cdot\va\\
&\ \ \ \ \ \ +C_{\lambda}\ka^2(1+\exp\{\exp\{\exp\{\exp\{C_{\lambda}\ka^2\}\}\}\})^{6}
\cdot\va^{\frac{1}{2}}.
\end{split}
\ee
\end{lemma}
\begin{proof}
We first compute $c^0_r(a,t)$ and $c^0_r(b,t)$ for later use. Differentiating the second equation of \eqref{e3} with respect to $r$ and using the boundary conditions one deduces that
\ben
\begin{split}
c^0_{rt}(a,t)=-w^0_r(a,t)c^0(a,t)-w^0(a,t)c^0_r(a,t)
=-c^0_r(a,t)[w^0(a,t)c^0(a,t)+w^0(a,t)],
\end{split}
\enn
which, multiplied with $e^{\int_0^t [w^0(a,\tau)c^0(a,\tau)+w^0(a,\tau)]d\tau}$
gives rise to
\be\label{a97}
\frac{d}{dt}\left\{
c^0_r(a,t) e^{\int_0^t [w^0(a,\tau)c^0(a,\tau)+w^0(a,\tau)]d\tau}
\right\}=0.
\ee
 Using the compatibility condition $c_{0r}(a)=-\ka [\lambda-c_0(a)]$, one integrates \eqref{a97} over $(0,t)$ to have
\be\label{a98}
c^0_r(a,t)=-\ka [\lambda-c_0(a)]e^{-\int_0^t [w^0(a,\tau)c^0(a,\tau)+w^0(a,\tau)]d\tau}, \qquad \forall\ \ t\geq 0.
\ee
A similar argument used in deriving \eqref{a98} further leads to
\be\label{a99}
c^0_r(b,t)=\ka [\lambda-c_0(b)]e^{-\int_0^t [w^0(b,\tau)c^0(b,\tau)+w^0(b,\tau)]d\tau}, \qquad \forall\ \ t\geq 0.
\ee

Taking the $L^2$ inner product of \eqref{a30} with $-2\ti{c}_{rrr}$, one gets by integration by parts that
\be\label{a100}
\begin{split}
\frac{d}{dt}&\int_a^b\ti{c}_{rr}^2dr
+2\va\int_a^b\ti{c}_{rrr}^2dr\\
=&-2\va\int_a^b \left(c^0_{rrr}+\frac{n-1}{r}c^{\va}_{rr}
-\frac{n-1}{r^2}c^\va_r
\right)\ti{c}_{rrr}dr
-2\int_a^b (\ti{w}c^\va+w^0\ti{c})_{rr}\ti{c}_{rr}dr\\
&+2[\ti{c}_{rt}\ti{c}_{rr}]\big|_{r=a}^{r=b}
+2[(\ti{w}c^\va+w^0\ti{c})_{r}\ti{c}_{rr}]
\big|_{r=a}^{r=b}\\
:=&\sum_{i=17}^{20}K_i.
\end{split}
\ee
It follows from the Cauchy-Schwarz inequality that
\be\label{c31}
\begin{split}
K_{17}=&-2\va\int_a^b \left(c^0_{rrr}+\frac{n-1}{r}\ti{c}_{rr}
+\frac{n-1}{r}c^{0}_{rr}
-\frac{n-1}{r^2}c^\va_r
\right)\ti{c}_{rrr}dr\\
\leq &\frac{1}{8}\va \|\ti{c}_{rrr}\|^2_{L^2}
+C_0\va(\|c^0\|_{H^3}^2+\|\ti{c}_{rr}\|_{L^2}^2+\|c^\va_r\|_{L^2}^2).
\end{split}
\ee
By a similar argument used in deriving \eqref{a73}, one deduces that
\be\label{c32}
\begin{split}
K_{18}\leq& C_0(\|w^0\|_{H^1}+\|w^\va\|_{H^1}+\|w^\va\|_{H^1}\|c^\va\|_{L^\infty}
+\|c^\va\|_{L^\infty}^2+1)\|\ti{c}_{rr}\|_{L^2}^2\\
&+C_0(\|\ti{w}_t\|_{L^2}^2+\|\ti{w}_r\|_{L^2}^2
+\|\ti{w}\|_{H^1}^2\|c^0\|_{H^2}^2
+\|w^\va\|_{H^1}^2\|\ti{c}_r\|_{L^2}^2)\\
&+C_0(\|\ti{w}_r\|_{L^2}^2 \|\ti{c}_r\|_{L^2}^2+\|w^0\|_{H^2}^2\|\ti{c}\|_{H^1}^2).
\end{split}
\ee
 It follows from \eqref{a82}, \eqref{a83} and the Gagliardo-Nirenberg interpolation inequality that
\ben
\begin{split}
K_{19}=&-2\ka \ti{c}_t(b,t)\ti{c}_{rr}(b,t)
-2\ka \ti{c}_t(a,t)\ti{c}_{rr}(a,t)\\
\leq &4\ka \|\ti{c}_t\|_{L^\infty}\|\ti{c}_{rr}\|_{L^\infty}\\
\leq &C_0\ka \|\ti{c}_t\|_{H^1}(\|\ti{c}_{rr}\|_{L^2}
+\|\ti{c}_{rr}\|^{\frac{1}{2}}_{L^2}\|\ti{c}_{rrr}\|^{\frac{1}{2}}_{L^2})\\
\leq &\frac{1}{8}\va \|\ti{c}_{rrr}\|^2_{L^2}+\|\ti{c}_{rr}\|^2_{L^2}
+C_0\ka^2(1+\va^{-\frac{1}{2}})\|\ti{c}_t\|_{H^1}^2.
\end{split}
\enn
We rewrite $K_{20}$ as follows:
\ben
\begin{split}
K_{20}&=2[\ti{w}_rc^\va \ti{c}_{rr}]\big|_{r=a}^{r=b}
+2[\ti{w}c^\va_r \ti{c}_{rr}]\big|_{r=a}^{r=b}
+2[w^0_r\ti{c} \ti{c}_{rr}]\big|_{r=a}^{r=b}
+2[w^0\ti{c}_r \ti{c}_{rr}]\big|_{r=a}^{r=b}\\
&:=G_1+G_2+G_3+G_4.
\end{split}
\enn
From \eqref{a79}, \eqref{b27}, the nonnegativity of $w^0$, $c^0$ and Sobolev embedding inequality we deduce that
\be\label{b28}
|c^0_r(a,t)|+|c^0_r(b,t)|\leq \kappa [2\lambda+|c_0(a)|+|c_0(b)|]
\leq C_0\kappa(\lambda+\|c_0\|_{L^\infty})
\leq C_0\kappa(\lambda+\|c_0\|_{H^1})
\ee
for each $t>0$.
It follows from the first two boundary conditions in \eqref{e4}, \eqref{a82}, \eqref{a83}, \eqref{a1}, \eqref{b28} and the Gagliardo-Nirenberg interpolation inequality that
\ben
\begin{split}
G_1=&2[w^\va \ti{c}_rc^\va \ti{c}_{rr}]\big|_{r=a}^{r=b}
+2[\ti{w} c^0_rc^\va \ti{c}_{rr}]\big|_{r=a}^{r=b}\\
=&-2\ka[(w^\va \ti{c}\,c^\va\ti{c}_{rr})(b,t)
+(w^\va \ti{c}\,c^\va\ti{c}_{rr})(a,t)]
+2[\ti{w} c^0_rc^\va \ti{c}_{rr}]\big|_{r=a}^{r=b}
\\
\leq& 4\ka \|w^\va\|_{L^\infty}\|\ti{c}\|_{L^\infty}
\|c^\va\|_{L^\infty}\|\ti{c}_{rr}\|_{L^\infty}
+C_0\ka \|\ti{w}\|_{L^\infty}(\lambda+\|c_0\|_{H^1})
\|c^\va\|_{L^\infty}\|\ti{c}_{rr}\|_{L^\infty}\\
\leq& C_0\ka \|w^\va\|_{H^1}\|\ti{c}\|_{H^1}
(\lambda+\|c_0\|_{L^\infty})(\|\ti{c}_{rr}\|_{L^2}
+\|\ti{c}_{rr}\|^{\frac{1}{2}}_{L^2}\|\ti{c}_{rrr}\|^{\frac{1}{2}}_{L^2}
)\\
&+C_0\ka \|\ti{w}\|_{H^1}(\lambda+\|c_0\|_{H^1})
(\lambda+\|c_0\|_{L^\infty})(\|\ti{c}_{rr}\|_{L^2}
+\|\ti{c}_{rr}\|^{\frac{1}{2}}_{L^2}\|\ti{c}_{rrr}\|^{\frac{1}{2}}_{L^2}
)\\
\leq & \frac{1}{8}\va \|\ti{c}_{rrr}\|^2_{L^2}
+\|\ti{c}_{rr}\|^2_{L^2}
+C_{0}(\lambda+\|c_0\|_{H^1}+1)^4\ka^2(1+\va^{-\frac{1}{2}})
(\|w^\va\|_{H^1}^2\|\ti{c}\|_{H^1}^2+\|\ti{w}\|_{H^1}^2).
\end{split}
\enn
\eqref{a82}, \eqref{a83}, \eqref{b28} and the Gagliardo-Nirenberg interpolation inequality leads to
\ben
\begin{split}
G_2=&2[\ti{w} \ti{c}_r \ti{c}_{rr}]\big|_{r=a}^{r=b}
+2[\ti{w} c^0_r \ti{c}_{rr}]\big|_{r=a}^{r=b}\\
=&-2\ka[(\ti{w} \ti{c}\ti{c}_{rr})(b,t)+(\ti{w} \ti{c}\ti{c}_{rr})(a,t)]
+2[\ti{w} c^0_r \ti{c}_{rr}]\big|_{r=a}^{r=b}\\
\leq& 4\ka \|\ti{w}\|_{L^\infty}\|\ti{c}\|_{L^\infty}
\|\ti{c}_{rr}\|_{L^\infty}
+C_0\ka \|\ti{w}\|_{L^\infty}(\lambda+\|c_0\|_{H^1})
\|\ti{c}_{rr}\|_{L^\infty}\\
\leq& C_0\ka \|\ti{w}\|_{H^1}\|\ti{c}\|_{H^1}
(\|\ti{c}_{rr}\|_{L^2}
+\|\ti{c}_{rr}\|^{\frac{1}{2}}_{L^2}\|\ti{c}_{rrr}\|^{\frac{1}{2}}_{L^2}
)\\
&+C_0\ka \|\ti{w}\|_{H^1}
(\lambda+\|c_0\|_{H^1})(\|\ti{c}_{rr}\|_{L^2}
+\|\ti{c}_{rr}\|^{\frac{1}{2}}_{L^2}\|\ti{c}_{rrr}\|^{\frac{1}{2}}_{L^2}
)\\
\leq & \frac{1}{8}\va \|\ti{c}_{rrr}\|^2_{L^2}
+\|\ti{c}_{rr}\|^2_{L^2}
+C_{0}(\lambda+\|c_0\|_{H^1}+1)^2\ka^2(1+\va^{-\frac{1}{2}})
(\|\ti{w}\|_{H^1}^2\|\ti{c}\|_{H^1}^2+\|\ti{w}\|_{H^1}^2).
\end{split}
\enn
The boundary conditions in \eqref{e3} and \eqref{b28} entail that
\ben
\begin{split}
G_3=&2[w^0 c^0_r \ti{c}\,\ti{c}_{rr}]\big|_{r=a}^{r=b}\\
\leq & C_0\ka\|w^0\|_{L^\infty}(\lambda+\|c_0\|_{H^1})
\|\ti{c}\|_{L^\infty}\|\ti{c}_{rr}\|_{L^\infty}\\
\leq &C_0\ka\|w^0\|_{H^1}(\lambda+\|c_0\|_{H^1})\|\ti{c}\|_{H^1}
(\|\ti{c}_{rr}\|_{L^2}
+\|\ti{c}_{rr}\|^{\frac{1}{2}}_{L^2}\|\ti{c}_{rrr}\|^{\frac{1}{2}}_{L^2}
)\\
\leq & \frac{1}{8}\va \|\ti{c}_{rrr}\|^2_{L^2}
+\|\ti{c}_{rr}\|^2_{L^2}
+C_{0}(\lambda+\|c_0\|_{H^1})^2\ka^2(1+\va^{-\frac{1}{2}})
\|w^0\|_{H^1}^2\|\ti{c}\|_{H^1}^2.
\end{split}
\enn
It follows from \eqref{a82}, \eqref{a83} and the Gagliardo-Nirenberg interpolation inequality that
\ben
\begin{split}
G_4=&-2\ka[(w^0\ti{c}\ti{c}_{rr})(b,t)
+(w^0\ti{c}\ti{c}_{rr})(a,t)]\\
\leq &4\ka\|w^0\|_{L^\infty}\|\ti{c}\|_{L^\infty}
\|\ti{c}_{rr}\|_{L^\infty}\\
\leq &C_0\ka\|w^0\|_{H^1}\|\ti{c}\|_{H^1}
(\|\ti{c}_{rr}\|_{L^2}
+\|\ti{c}_{rr}\|^{\frac{1}{2}}_{L^2}\|\ti{c}_{rrr}\|^{\frac{1}{2}}_{L^2})\\
\leq & \frac{1}{8}\va \|\ti{c}_{rrr}\|^2_{L^2}
+\|\ti{c}_{rr}\|^2_{L^2}
+C_0\ka^2(1+\va^{-\frac{1}{2}})
\|w^0\|_{H^1}^2\|\ti{c}\|_{H^1}^2.
\end{split}
\enn
Collecting the above estimates for $G_1$ - $G_4$, one deduces that
\ben
\begin{split}
K_{20}&
\leq \frac{1}{2}\va \|\ti{c}_{rrr}\|_{L^2}^2
+4\|\ti{c}_{rr}\|_{L^2}^2\\
&+C_0\ka^2(\lambda+\|c_0\|_{H^1}+1)^4(1+\va^{-\frac{1}{2}})
(\|w^\va\|_{H^1}^2\|\ti{c}\|_{H^1}^2+\|\ti{w}\|_{H^1}^2
+\|\ti{w}\|_{H^1}^2\|\ti{c}\|_{H^1}^2
+\|w^0\|_{H^1}^2\|\ti{c}\|_{H^1}^2).
\end{split}
\enn
Substituting the above estimates for $K_{17}$ - $K_{20}$ into \eqref{a100}, we arrive at
\ben
\begin{split}
\frac{d}{dt}&\int_a^b\ti{c}_{rr}^2dr
+\va\int_a^b\ti{c}_{rrr}^2dr\\
\leq&
C_0(\|w^0\|_{H^1}+\|w^\va\|_{H^1}+\|w^\va\|_{H^1}\|c^\va\|_{L^\infty}+
\|c^\va\|_{L^\infty}^2+1)\|\ti{c}_{rr}\|^2_{L^2}
+C_0\va(\|c^0\|_{H^3}^2+\|c^\va_r\|_{L^2}^2)\\
&+C_0(\|\ti{w}_t\|^2_{L^2}+\|\ti{w}_r\|^2_{L^2}
+\|\ti{w}\|_{H^1}^2\|c^0\|_{H^2}^2
+\|w^\va\|_{H^1}^2\|\ti{c}_r\|_{L^2}^2
+\|\ti{w}_r\|^2_{L^2} \|\ti{c}_r\|_{L^2}^2+\|w^0\|_{H^2}^2\|\ti{c}\|_{H^1}^2)\\
&+C_0\ka^2(1+\va^{-\frac{1}{2}})(\lambda+\|c_0\|_{H^1}+1)^4
(\|w^\va\|_{H^1}^2\|\ti{c}\|_{H^1}^2+\|\ti{w}\|_{H^1}^2
+\|\ti{w}\|_{H^1}^2\|\ti{c}\|_{H^1}^2
)\\
&+C_0\ka^2(1+\va^{-\frac{1}{2}})(\lambda+\|c_0\|_{H^1}+1)^4
(
\|w^0\|_{H^1}^2\|\ti{c}\|_{H^1}^2
+\|\ti{c}_t\|_{H^1}^2),
\end{split}
\enn
which, along with Gronwall's inequality, \eqref{a0}, \eqref{a000}, \eqref{a1}, \eqref{a61} and \eqref{a67} gives rise to
 \be\label{b29}
\begin{split}
\sup_{0<t\leq T}\int_a^b\ti{c}_{rr}^2(t)dr
+\va\int_0^T\int_a^b\ti{c}_{rrr}^2drdt
\leq& C_{\lambda}(1+\exp\{\exp\{\exp\{\exp\{C_{\lambda}\ka^2\}\}\}\})^{5}\cdot\va\\
&+C_{\lambda}\ka^2(1+\exp\{\exp\{\exp\{\exp\{C_{\lambda}\ka^2\}\}\}\})^{5}
\cdot\va^{\frac{1}{2}}.
\end{split}
\ee
By a similar argument used in deriving \eqref{a75}, one gets
\be\label{c35}
\begin{split}
\int_a^b\ti{w}_{rr}^2dr
\leq& C_0
\|w^\va\|^2_{H^1}(\|\ti{c}_r\|_{L^2}^2+\|\ti{c}_{rr}\|_{L^2}^2)
+C_0(\|\ti{w}\|_{H^1}^2\|c^0\|_{H^2}^2
+\|w^\va\|_{H^1}^2 \|\ti{c}_r\|_{L^2}^2)\\
&+C_0(\|\ti{w}_t\|_{L^2}^2+\|\ti{w}_r\|_{L^2}^2),
\end{split}
\ee
which, in conjunction with \eqref{a0}, \eqref{a000},  \eqref{a61}, \eqref{a67} and \eqref{b29} leads to
\be\label{b30}
\begin{split}
\sup_{0<t\leq T}\int_a^b\ti{w}_{rr}^2(t)dr
\leq& C_{\lambda}(1+\exp\{\exp\{\exp\{\exp\{C_{\lambda}\ka^2\}\}\}\})^{6}\cdot\va\\
&+C_{\lambda}\ka^2(1+\exp\{\exp\{\exp\{\exp\{C_{\lambda}\ka^2\}\}\}\})^{6}
\cdot\va^{\frac{1}{2}}.
\end{split}
\ee
Collecting \eqref{b29} and \eqref{b30}
 one derives \eqref{a95}. The proof is finished.

\end{proof}

\begin{lemma}\label{l12}
 Let the assumptions in Theorem \ref{t2} hold and let $(w^\va,c^\va)$ with $0\leq \va<1$ be the solutions of \eqref{e1}-\eqref{e2}, derived in Lemma \ref{l9}. Assume further $\ka=0$.
Then there exists a constant $C$ depending on $a$, $b$, $n$, $\|w_0\|_{H^2}$, $\|c_0\|_{H^3}$ and $T$ such that
\be\label{c20}
\|w^\va-w^0\|_{L^\infty(0,T;H^2)}^2+\|c^\va-c^0\|_{L^\infty(0,T;H^2)}^2
\leq C\va^2.
\ee
\end{lemma}

\begin{proof} First, by setting $\kappa=0$ in \eqref{a98} and \eqref{a99} one derives
\be\label{c21}
c^0_r(a,t)=c^0_r(b,t)=0,\quad \ \ \forall \ t\geq 0,
\ee
which, along with the boundary conditions in \eqref{e3} indicates that
\be\label{c22}
w^0_r(a,t)=w^0_r(b,t)=0,\quad \ \ \forall \ t\geq 0.
\ee
Since $\kappa=0$, it follows from the boundary conditions in \eqref{e2} that
\be\label{c29}
c^\va_r(a,t)=c^\va_r(b,t)=w^\va_r(a,t)=w^\va_r(b,t)=0,\ \ \quad \forall \ t\geq 0,
\ee
which, in conjunction with \eqref{c21} and \eqref{c22} gives rise to
\be\label{c23}
\ti{c}_r(a,t)=\ti{c}_r(b,t)=\ti{w}_r(a,t)=\ti{w}_r(b,t)=0,
\ \ \quad \forall \ t\geq 0.
\ee
Then from \eqref{c23}, one easily gets
\be\label{c24}
\va [\ti{c}_{rr}\ti{c}_r]|_{r=a}^{r=b}=0,\quad \ \ \forall \ t\geq 0.
\ee
Substituting \eqref{a86} into \eqref{a26} and using \eqref{c24} leads to
\be\label{c25}
\begin{split}
\frac{1}{2}\frac{d}{dt}&\int_a^b \ti{c}_r^2dr
+\va \int_a^b\ti{c}_{rr}^2 dr\\
\leq & \frac{1}{4}\|\ti{w}_r\|^2_{L^2}
+C_0(\|c^\va\|_{H^1}^2+\|w^0\|_{H^2}^2+1)
\|\ti{c}_r\|^2_{L^2}\\
&+C_0(\|\ti{w}\|^2_{L^2}+\|\ti{c}\|^2_{L^2})
+C_0\va^2(\|c^0\|^2_{H^3}+\|c^\va_{rr}\|^2_{L^2}+\|c^\va_{r}\|^2_{L^2})
.
\end{split}
\ee
Adding \eqref{c25} and \eqref{a25} to \eqref{a24}, we have
\ben
\begin{split}
\frac{d}{dt}&\left[
\int_a^b \ti{w}^2dr+
\int_a^b \ti{c}^2dr+
\int_a^b \ti{c}_r^2dr
\right]
+\int_a^b  \ti{w}_r^2dr+\va\int_a^b \ti{c}_{rr}^2dr\\
&\leq C_0(\|w^\va\|_{H^1}^2+\|c^\va\|_{H^1}^2+\|w^0\|_{H^2}^2+1)\|\ti{c}_r\|^2_{L^2}
+C_0(\|c^\va\|_{H^1}^2+\|c^0\|_{H^2}^2+1)\|\ti{w}\|^2_{L^2}\\
&\ \ \ \ +C_0(\|w^0\|_{H^1}+1)\|\ti{c}\|^2_{L^2}
+C_0\va^2(\|c^\va_{rr}\|_{L^2}^2+\|c^\va_{r}\|_{L^2}^2+\|c^0\|_{H^3}^2).
\end{split}
\enn
Applying Gronwall's inequality to the above inequality, using \eqref{c19} and \eqref{a000} we arrive at
\be\label{c27}
\begin{split}
\sup_{0<t\leq T}\left[
\int_a^b \ti{w}^2(t)dr+
\int_a^b \ti{c}^2(t)dr+
\int_a^b \ti{c}_r^2(t)dr
\right]+\int_0^T\int_a^b  \ti{w}_r^2drdt
\leq
 C\va^2.
\end{split}
\ee
It follows from \eqref{c23} that
\be\label{c26}
\va[\ti{c}_{rr}\ti{c}_{rt}]\big|_{r=a}^{r=b}=0,\quad  \ \ \forall \ t\geq 0.
\ee
Substituting \eqref{a93} into \eqref{a90} and using \eqref{c26}, one gets
\ben
\begin{split}
\va\frac{d}{dt}&\int_a^b \ti{c}_{rr}^2dr
+\int_a^b \ti{c}_{rt}^2dr\\
\leq &C_0(\|c^\va\|_{H^1}^2\|\ti{w}\|_{H^1}^2+\|w^0\|_{H^2}^2
\|\ti{c}\|_{H^1}^2)
+C_0\va^2(\|c^0\|^2_{H^3}+\|c^\va_{rr}\|_{L^2}^2+\|c^\va_{r}\|_{L^2}^2).
\end{split}
\enn
Integrating the above inequality with respect to $t$ and using \eqref{c27}, \eqref{c19} and \eqref{a000} we get
\be\label{c28}
\begin{split}
\sup_{0<t\leq T}\va \int_a^b \ti{c}_{rr}^2dr
+\int_0^T\int_a^b \ti{c}_{rt}^2drdt
\leq
C
\va^2.
\end{split}
\ee
Applying Gronwall's inequality to \eqref{a28} and using \eqref{c27}, \eqref{c28}, \eqref{c19} and \eqref{a000}, we derive
\be\label{c33}
\begin{split}
\sup_{0<t\leq T}\left[\int_a^b \ti{w}_t^2 dr
+\int_a^b \ti{w}_r^2 dr\right]
+\int_0^T\int_a^b \ti{w}_{t}^2 drdt
+\int_0^T\int_a^b \ti{w}_{rt}^2 drdt
\leq C\va^2.
\end{split}
\ee
It follows from \eqref{c21}-\eqref{c23} that
\be\label{c30}
2[\ti{c}_{rt}\ti{c}_{rr}]\big|_{r=a}^{r=b}
=2[(\ti{w}c^\va+w^0\ti{c})_{r}\ti{c}_{rr}]
\big|_{r=a}^{r=b}=0,\quad \ \ \forall \ t\geq 0.
\ee
Substituting \eqref{c31} and \eqref{c32} into \eqref{a100} and using \eqref{c30} to have
\ben
\begin{split}
\frac{d}{dt}&\int_a^b\ti{c}_{rr}^2dr
+\va\int_a^b\ti{c}_{rrr}^2dr\\
\leq&
C_0(\|w^0\|_{H^1}+\|w^\va\|_{H^1}+\|w^\va\|_{H^1}\|c^\va\|_{L^\infty}+
\|c^\va\|_{L^\infty}^2+1)\|\ti{c}_{rr}\|^2_{L^2}
+C_0\va(\|c^0\|_{H^3}^2+\|c^\va_r\|_{L^2}^2)\\
&+C_0(\|\ti{w}_t\|^2_{L^2}+\|\ti{w}_r\|^2_{L^2}
+\|\ti{w}\|_{H^1}^2\|c^0\|_{H^2}^2
+\|w^\va\|_{H^1}^2\|\ti{c}_r\|_{L^2}^2
+\|\ti{w}_r\|^2_{L^2} \|\ti{c}_r\|_{L^2}^2+\|w^0\|_{H^2}^2\|\ti{c}\|_{H^1}^2).
\end{split}
\enn
Applying Gronwall's inequality to the above inequality and using \eqref{c27}, \eqref{c33}, \eqref{c19}, \eqref{a000} and \eqref{b25}, one gets
\be\label{c34}
\begin{split}
\sup_{0<t\leq T}\int_a^b\ti{c}_{rr}^2(t)dr
+\va\int_0^T\int_a^b\ti{c}_{rrr}^2drdt
\leq C\va^2.
\end{split}
\ee
By \eqref{c35}, \eqref{c27}, \eqref{c33}, \eqref{c34}, \eqref{c19} and \eqref{a000}, one further deduces that
\be\label{c36}
\begin{split}
\sup_{0<t\leq T}\int_a^b\ti{w}_{rr}^2(t)dr
\leq C\va^2.
\end{split}
\ee
Collecting \eqref{c27}, \eqref{c33}, \eqref{c34} and \eqref{c36}, we end up with the desired estimate. The proof is finished.
\end{proof}

Based on the results derived in Lemma \ref{l0} - Lemma \ref{l12}, we can prove Theorem \ref{t2} now.\\

\textbf{\emph{Proof of Theorem \ref{t2}.}} We first prove Part (i).
 One deduces from the Sobolev embedding inequality, \eqref{a60} and \eqref{a66} that
\ben
\begin{split}
\|&w^\va-w^0\|_{L^\infty(0,T;C[a,b])}
+\|c^\va-c^0\|_{L^\infty(0,T;C[a,b])}\\
&\leq C_0\|w^\va-w^0\|_{L^\infty(0,T;H^1)}
+C_0\|c^\va-c^0\|_{L^\infty(0,T;H^1)}\\
&\leq C_{\lambda,\ka}\va^{\frac{1}{4}},
\end{split}
\enn
where the constant $C_{\lambda,\ka}$ depends on $a$, $b$, $n$, $\|w_0\|_{H^2}$, $\|c_0\|_{H^3}$, $T$, $\lambda$ and $\kappa$.

We thus derive \eqref{b1} and proceed to proving \eqref{b2}.
A direct computation leads to
\be\label{a103}
\delta^2\leq \frac{4}{(b-a)^2}(r-a)^2(r-b)^2,\ \  \quad r\in (a+\delta,b-\delta)
 \ee
  provided that $0<\delta<\frac{b-a}{2}$. Thus it follows from \eqref{a103} and Lemma \ref{l8} that
\ben
\delta^2\int_{a+\delta}^{b-\delta}
\ti{w}_{rr}^2(r,t)\,dr
\leq \frac{4}{(b-a)^2}
\int_{a+\delta}^{b-\delta}(r-a)^2(r-b)^2 \ti{w}_{rr}^2(r,t)\,dr
\leq C_{\lambda,\ka}\va^{1/2},\quad \ \ \forall \ t\in [0,T]
\enn
with the constant $C_{\lambda,\ka}$ depending on $a$, $b$, $n$, $\|w_0\|_{H^2}$, $\|c_0\|_{H^3}$, $T$, $\lambda$ and $\kappa$,
which, in conjunction with \eqref{a66} and the Gagliardo-Nirenberg interpolation inequality entails that
\be\label{a102}
\begin{split}
\|w_r^\va-w_r^0\|_{L^\infty(0,T;C[a+\delta,b-\delta])}
\leq& C_0 \big(\|\ti{w}_r\|_{L^\infty(0,T;L^2(a+\delta,b-\delta))}\\
&+\|\ti{w}_r\|_{L^\infty(0,T;L^2(a+\delta,b-\delta))}^{1/2}
\|\ti{w}_{rr}\|_{L^\infty(0,T;L^2(a+\delta,b-\delta))}^{1/2}
\big)\\
\leq &C_{\lambda,\ka}\big(\va^{1/4}+\va^{1/8}\cdot \va^{1/8}\delta^{-1/2}\big)\\
\leq& C_{\lambda,\ka}\va^{1/4}\delta^{-1/2},
\end{split}
\ee
provided $0<\delta<\min\{1,\frac{b-a}{2}\}$. By a similar argument used in deriving \eqref{a102}, one deduces from Lemma \ref{l8} and \eqref{a60} that
\ben
\begin{split}
\|c_r^\va-c_r^0\|_{L^\infty(0,T;C[a+\delta,b-\delta])}
\leq C_{\lambda,\ka}\va^{1/4}\delta^{-1/2},
\end{split}
\enn
which, along with \eqref{a102} gives \eqref{b2}.

 We next prove the equivalence between \eqref{b3} and \eqref{sc1}. We shall prove that \eqref{b3} implies \eqref{sc1} by argument of contradiction. Assume \eqref{sc1} is false, that is
 \ben
 w^0(a,t)=w^0(b,t)=0,\quad \forall\ t\in [0,T].
 \enn
 Then from Lemma \ref{l0} we know that \eqref{a82} and \eqref{a83} hold true. Thus it follows from the Gagliardo-Nirenberg interpolation inequality, \eqref{a61}, \eqref{a67} and \eqref{a95} that
\ben
\begin{split}
\|&w^\va_r-w^0_r\|_{L^\infty(0,T;C[a,b])}
+\|c^\va_r-c^0_r\|_{L^\infty(0,T;C[a,b])}\\
&\leq C_0(\|\ti{w}_r\|_{L^2}
+\|\ti{w}_r\|_{L^2}^{\frac{1}{2}}\|\ti{w}_{rr}\|_{L^2}
^{\frac{1}{2}})
+C_0(\|\ti{c}_r\|_{L^2}
+\|\ti{c}_r\|_{L^2}^{\frac{1}{2}}\|\ti{c}_{rr}\|_{L^2}
^{\frac{1}{2}})\\
&\leq
 C_{\lambda,\ka}\va^{\frac{3}{8}},
\end{split}
\enn
where the constant $C_{\lambda,\ka}$ depends on $a$, $b$, $n$, $\|w_0\|_{H^2}$, $\|c_0\|_{H^3}$, $T$, $\lambda$ and $\kappa$. Thus
\ben
\lim_{\va\to 0}(\|w^\va_r-w^0_r\|_{L^\infty(0,T;C[a,b])}
+\|c^\va_r-c^0_r\|_{L^\infty(0,T;C[a,b])})
=0,
\enn
which, contradicts with \eqref{b3}. Thus \eqref{b3} indicates \eqref{sc1}. We proceed to proving that \eqref{sc1} implies \eqref{b3}. Without loss of generality we assume that there exists $t_0\in [0,T]$ such that
\be\label{b43}
w^0(a,t_0)>0.
\ee
Then it follows from Lemma \ref{l11} that there exists $t_1\in [0,T]$ such that
\be\label{b41}
w^0(a,t_1)>0
\ee
and
\be\label{b31}
c_0(a)e^{-\int_0^{t_1} w^0(a,\tau)d\tau}
\left\{
1-e^{-\int_0^{t_1} w^0(a,\tau)c^0(a,\tau)d\tau}
\right\}
-\lambda \left\{1-e^{-\int_0^{t_1} w^0(a,\tau)[c^0(a,\tau)+1]d\tau}\right\}
\neq0.
\ee
From \eqref{b26} we know that
\ben
\begin{split}
c^\va_r&(a,t_1)-c^0_r(a,t_1)\\
=&\ka[c^\va(a,t_1)-c^0(a,t_1)]\\
&+\ka c_0(a)e^{-\int_0^{t_1} w^0(a,\tau)d\tau}
\left\{
1-e^{-\int_0^{t_1} w^0(a,\tau)c^0(a,\tau)d\tau}
\right\}
-\ka\lambda \left\{1-e^{-\int_0^{t_1} w^0(a,\tau)[c^0(a,\tau)+1]d\tau}\right\},
\end{split}
\enn
which, along with \eqref{b1}, \eqref{b31} and the fact $\ka>0$ leads to
\be\label{b45}
\begin{split}
&\liminf_{\va\to 0} |c^\va_r(a,t_1)-c^0_r(a,t_1)|\\
\geq&
\ka \Big|c_0(a)e^{-\int_0^{t_1} w^0(a,\tau)d\tau}
\left\{
1-e^{-\int_0^{t_1} w^0(a,\tau)c^0(a,\tau)d\tau}
\right\}
-\lambda \left\{1-e^{-\int_0^{t_1} w^0(a,\tau)[c^0(a,\tau)+1]d\tau}\right\}\Big|\\
&-\ka\limsup_{\va\to 0}|c^\va(a,t_1)-c^0(a,t_1)|\\
=& \ka \Big|c_0(a)e^{-\int_0^{t_1} w^0(a,\tau)d\tau}
\left\{
1-e^{-\int_0^{t_1} w^0(a,\tau)c^0(a,\tau)d\tau}
\right\}
-\lambda \left\{1-e^{-\int_0^{t_1} w^0(a,\tau)[c^0(a,\tau)+1]d\tau}\right\}\Big|\\
>&0.
\end{split}
\ee
Hence,
\be\label{b42}
\liminf_{\va\to 0} \|c^\va_r(a,t)-c^0_r(a,t)\|_{L^\infty(0,T;C[a,b])}
\geq \liminf_{\va\to 0} |c^\va_r(a,t_1)-c^0_r(a,t_1)|>0.
\ee
From the first boundary condition in \eqref{e4} one gets
\be\label{b44}
w^\va_r(a,t_1)-w^0_r(a,t_1)
=w^\va(a,t_1)[c^\va_r(a,t_1)-c^0_r(a,t_1)]
+[w^\va(a,t_1)-w^0(a,t_1)]c^0_r(a,t_1).
\ee
Noting $w^\va\geq 0$ in $[a,b]\times [0,T]$, it follows from \eqref{b1}, \eqref{b45} and \eqref{b41} that
\ben
\begin{split}
&\liminf_{\va\to 0} [w^\va(a,t_1)|c^\va_r(a,t_1)-c^0_r(a,t_1)|]\\
\geq& \liminf_{\va\to 0} w^\va(a,t_1)\cdot\liminf_{\va\to 0}|c^\va_r(a,t_1)-c^0_r(a,t_1)|\\
=& w^0(a,t_1)\cdot\liminf_{\va\to 0}|c^\va_r(a,t_1)-c^0_r(a,t_1)|\\
>&0,
\end{split}
\enn
which, in conjunction with \eqref{b44} and \eqref{b1} gives rise to
\ben
\begin{split}
&\liminf_{\va\to 0}|w^\va_r(a,t_1)-w^0_r(a,t_1)|\\
\geq &\liminf_{\va\to 0} [w^\va(a,t_1)|c^\va_r(a,t_1)-c^0_r(a,t_1)|]
-\limsup_{\va\to 0}|w^\va(a,t_1)-w^0(a,t_1)|\,|c^0_r(a,t_1)|\\
=&\liminf_{\va\to 0} [w^\va(a,t_1)|c^\va_r(a,t_1)-c^0_r(a,t_1)|]\\
> &0.
\end{split}
\enn
Hence,
\be\label{b46}
\liminf_{\va\to 0} \|w^\va_r(a,t)-w^0_r(a,t)\|_{L^\infty(0,T;C[a,b])}
\geq \liminf_{\va\to 0} |w^\va_r(a,t_1)-w^0_r(a,t_1)|>0.
\ee
Collecting \eqref{b42} and \eqref{b46} we proved \eqref{b3} by assuming \eqref{b43}. By a similar argument, one can deduce \eqref{b3} on assumption $w^0(b,t)>0$ for some $t\in [0,T]$. Thus \eqref{sc1} indicates \eqref{b3}. Part (ii) follows directly from Lemma \ref{l12}. The proof is finished.
\endProof

\section*{Acknowledgements}
This work is supported by National Natural Science Foundation of China (No. 11901139), China Postdoctoral Science Foundation
(Nos. 2019M651269, 2020T130151).

\setlength{\bibsep}{0.5ex}
\bibliography{rf}

\begin{thebibliography}{10}

\bibitem{atkins-paula2006}
P.W. Atkins and J.~de Paula.
\newblock {\em Atkins' Physical Chemistry}.
\newblock Oxford University Press, New York, 2006.

\bibitem{braukhoff2017}
M.~Braukhoff.
\newblock Global (weak) solution of the chemotaxis-{N}avier-{S}tokes equations
  with non-homogeneous boundary conditions and logistic growth.
\newblock {\em Ann. Inst. H. Poincar\'{e} C Anal. Non Lin\'{e}aire},
  34:1013--1039, 2017.

\bibitem{braukhoff-lankeit2019}
M.~Braukhoff and J.~Lankeit.
\newblock Stationary solutions to a chemotaxis-consumption model with realistic
  boundary conditions for the oxygen.
\newblock {\em Math. Models Methods Appl. Sci.}, 29:2033--2062, 2019.

\bibitem{braukhoff-tang2020}
M.~Braukhoff and B.~Q. Tang.
\newblock Global solutions for chemotaxis-{N}avier-{S}tokes system with robin
  boundary conditions.
\newblock {\em J. Differential Equations}, 269:10630--10669, 2020.

\bibitem{carrillo-li-wang2021}
J.~A. Carrillo, J.Y. Li, and Z.A. Wang.
\newblock Boundary spike-layer solutions of the singular {K}eller-{S}egel
  system: existence and stability.
\newblock {\em Proc. Lond. Math. Soc. (3)}, 122:42--68, 2021.

\bibitem{chae-kang-li2012}
M.~Chae, K.~Kang, and J.~Lee.
\newblock Existence of smooth solutions to coupled chemotaxis-fluid equations.
\newblock {\em Discrete Contin. Dyn. Syst. A}, 33:2271--2297, 2012.

\bibitem{chae-kang-li2014}
M.~Chae, K.~Kang, and J.~Lee.
\newblock Global existence and temporal decay in {K}eller-{S}egel models
  coupled to fluid equations.
\newblock {\em Comm. Partial Differential Equations}, 39:1205--1235, 2014.

\bibitem{duan-lorz-markowich2010}
R.~Duan, A.~Lorz, and P.~Markowich.
\newblock Global solutions to the coupled chemotaxis-fluid equations.
\newblock {\em Comm. Partial Differential Equations}, 35:1635--1673, 2010.

\bibitem{frid-shelukhin99}
H.~Frid and V.~Shelukhin.
\newblock Boundary layers for the {N}avier-{S}tokes equations of compressible
  fluids.
\newblock {\em {\it Comm. Math. Phys.}}, 208:309--330, 1999.

\bibitem{frid-shelukhin20}
H.~Frid and V.~Shelukhin.
\newblock Vanishing shear viscosity in the equations of compressible fluids for
  the flows with the cylinder symmetry.
\newblock {\em {\it SIAM J. Math. Anal.}}, 31:1144--1156, 2000.

\bibitem{fuest-lankeit-mizukami2021}
M.~Fuest, J.~Lankeit, and M.~Mizukami.
\newblock Long-term behaviour in a parabolic-elliptic chemotaxis-consumption
  model.
\newblock {\em J. Differential Equations}, 271:254--279, 2021.

\bibitem{hoff1992}
D.~Hoff.
\newblock Spherically symmetric solutions of the {N}avier-{S}tokes equations
  for compressible, isothermal flow with large, discontinuous initial data.
\newblock {\em Indiana Univ. Math. J.}, 41:1225--1302, 1992.

\bibitem{HLWW}
Q.~Hou, C.~J. Liu, Y.~G. Wang, and Z.~Wang.
\newblock Stability of boundary layers for a viscous hyperbolic system arising
  from chemotaxis: one dimensional case.
\newblock {\em SIAM J. Math. Anal.}, 50:3058--3091, 2018.

\bibitem{hou2019convergence}
Q.~Hou and Z.~Wang.
\newblock Convergence of boundary layers for the {K}eller--{S}egel system with
  singular sensitivity in the half-plane.
\newblock {\em J. Math. Pures. Appl.}, 130:251--287, 2019.

\bibitem{HWZ}
Q.~Hou, Z.~Wang, and K.~Zhao.
\newblock Boundary layer problem on a hyperbolic system arising from
  chemotaxis.
\newblock {\em J. Differential Equations}, 261:5035--5070, 2016.

\bibitem{lankeit-winkler2022}
J.~Lankeit and M.~Winkler.
\newblock Radial solutions to a chemotaxis-consumption model involving
  prescribed signal concentrations on the boundary.
\newblock {\em Nonlinearity}, 35:719--749, 2022.

\bibitem{lee-wang-yang2020}
C.C. Lee, Z.A. Wang, and W.~Yang.
\newblock Boundary-layer profile of a singularly perturbed nonlocal semi-linear
  problem arising in chemotaxis.
\newblock {\em Nonlinearity}, 33:5111--5141, 2020.

\bibitem{morales-rodrigo2010}
G.~Li\c{t}canu and C.~Morales-Rodrigo.
\newblock Global solutions and asymptotic behavior for a parabolic degenerate
  coupled system arising from biology.
\newblock {\em Nonlinear Anal.}, 72:77--98, 2010.

\bibitem{liu-lorz2011}
J.~Liu and A.~Lorz.
\newblock A coupled chemotaxis-fluid model: Global existence.
\newblock {\em Ann. Inst. H. Poincar\'{e} Anal. Non Lin\'{e}aire}, 28:643--652,
  2011.

\bibitem{lorz2010}
A.~Lorz.
\newblock Coupled chemotaxis fluid equations.
\newblock {\em Math. Models Methods Appl. Sci.}, 20:987--1004, 2010.

\bibitem{peng-wang-zhao-zhu2018}
H.Y. Peng, Z.A. Wang, K.~Zhao, and C.J. Zhu.
\newblock Boundary layers and stabilization of the singular {K}eller-{S}egel
  system.
\newblock {\em Kinet. Relat. Models}, 11:1085--1123, 2018.

\bibitem{peng-xiang2019}
Y.P. Peng and Z.Y. Xiang.
\newblock Global existence and convergence rates to a chemotaxis-fluids system
  with mixed boundary conditions.
\newblock {\em J. Differential Equations}, 267:1277--1321, 2019.

\bibitem{tian-xiang2020}
Y.~Tian and Z.~Y. Xiang.
\newblock Global solutions to a 3{D} chemotaxis-{S}tokes system with nonlinear
  cell diffusion and {R}obin signal boundary condition.
\newblock {\em J. Differential Equations}, 269:2012--2056, 2020.

\bibitem{Tuval}
I.~Tuval, L.~Cisneros, C.~Dombrowski, C.W. Wolgemuth, J.O. Kessler, and R.E.
  Goldstein.
\newblock Bacterial swimming and oxygen transport near contact lines.
\newblock {\em Proc. Natl. Acad. Sci. USA}, 102:2277--2282, 2005.

\bibitem{wang-winkler-xiang2018}
Y.~Wang, M.~Winkler, and Z.~Xiang.
\newblock The small-convection limit in a two-dimensional
  chemotaxis-{N}avier-{S}tokes system.
\newblock {\em Math. Z.}, 289:71--108, 2018.

\bibitem{wang-winkler-xiang2021}
Y.L. Wang, M.~Winkler, and Z.Y. Xiang.
\newblock Local energy estimates and global solvability in a three-dimensional
  chemotaxis-fluid system with prescribed signal on the boundary.
\newblock {\em Comm. Partial Differential Equations}, 46:1058--1091, 2021.

\bibitem{winkler2012}
M.~Winkler.
\newblock Global large-data solutions in a chemotaxis-({N}avier-){S}tokes
  system modeling cellular swimming in fluid drops.
\newblock {\em Comm. Partial Differential Equations}, 37:319--351, 2012.

\bibitem{winkler2014arma}
M.~Winkler.
\newblock Stabilization in a two-dimensional chemotaxis-{N}avier-{S}tokes
  system.
\newblock {\em Arch. Rational Mech. Anal.}, 211:455--487, 2014.

\bibitem{winkler2016}
M.~Winkler.
\newblock Global weak solutions in a three-dimensional
  chemotaxis-{N}avier-{S}tokes system.
\newblock {\em Ann. Inst. H. Poincar\'{e} Anal. Non Lin\'{e}aire},
  33:1329--1352, 2016.

\bibitem{winkler2017}
M.~Winkler.
\newblock How far do chemotaxis-driven forces influence regularity in the
  {N}avier-{S}tokes system?
\newblock {\em Trans. Amer. Math. Soc.}, 369:3067--3125, 2017.

\bibitem{wu-xiang2020}
C.Y. Wu and Z.Y. Xiang.
\newblock Asymptotic dynamics on a chemotaxis-{N}avier-{S}tokes system with
  nonlinear diffusion and inhomogeneous boundary conditions.
\newblock {\em Math. Models Methods Appl. Sci.}, 30:1325--1374, 2020.

\bibitem{zhang-li2015}
Q.~Zhang and Y.~Li.
\newblock Convergence rates of solutions for a two-dimensional
  chemotaxis-{N}avier-{S}tokes system.
\newblock {\em Discrete Contin. Dyn. Syst. B}, 20:2751--2759, 2015.

\bibitem{zhang-zheng2014}
Q.~Zhang and X.~Zheng.
\newblock Global well-posedness for the two-dimensional incompressible
  chemotaxis-{N}avier-{S}tokes equations.
\newblock {\em SIAM J. Math. Anal.}, 46:3078--3105, 2014.

\end{thebibliography}
\bibliographystyle{plain}
\end{document}